\documentclass[12pt,reqno]{amsart}
\usepackage[foot]{amsaddr}
\usepackage{amssymb,amsmath,amsthm,enumerate,bbm,color,verbatim}
\usepackage{geometry}
\usepackage{times, graphicx}
\usepackage{url}
\usepackage{dsfont}
\usepackage{tikz}

\DeclareMathAlphabet{\pazocal}{OMS}{zplm}{m}{n}

\newcommand{\id}{\mathds{1}}

\newcommand{\Z}{\mathbb {Z}}
\newcommand{\N}{\mathbb {N}}

\newcommand{\Sc}{\mathcal {S}}

\newcommand{\area}{\operatorname {Area}}

\newcommand{\Uc}{\mathcal {U}}
\newcommand{\Vc}{\mathcal {V}}

\newcommand{\R}{\mathbb{R}}
\newcommand{\C}{\mathbb{C}}

\newcommand{\var}{\operatorname{Var}}
\newcommand{\cov}{\operatorname{Cov}}

\DeclareMathOperator{\diam}{diam}

\numberwithin{equation}{section}

\theoremstyle{plain}
\newtheorem{theorem}{\bf Theorem}[section]
\newtheorem*{theorem*}{Theorem 1.1$'$}
\newtheorem{lemma}[theorem]{\bf Lemma}
\newtheorem{proposition}[theorem]{\bf Proposition}
\newtheorem{corollary}[theorem]{\bf Corollary}

\theoremstyle{remark}
\newtheorem*{remark*}{\bf Remark}
\newtheorem{remark}[theorem]{\bf Remark}

\newtheorem{examples}[theorem]{\bf Examples}
\newtheorem{definition}[theorem]{\bf Definition}
\newtheorem{claim}[theorem]{\bf Claim}

\newcommand{\eps}{\varepsilon}

\newcommand{\Nc}{\mathcal{N}}
\newcommand{\E}{\mathbb{E}}
\newcommand{\Dc}{\mathcal{D}}

\newcommand{\prob}{\pazocal{P}r}

\begin{document}

\title[The giant component of excursion sets of spherical ensembles]{The giant component of excursion sets of spherical Gaussian ensembles: existence, uniqueness, and volume concentration}

\author{Stephen Muirhead$^{1,2}$}
\email{stephen.muirhead@monash.edu}
\address{\parbox{\linewidth}{$^1$School of Mathematics and Statistics, University of Melbourne \\ $^2$School of Mathematics, Monash University (current)}}

\author{Igor Wigman$^3$}
\email{igor.wigman@kcl.ac.uk}
\address{$^3$Department of Mathematics, King's College London}

\begin{abstract}
We establish the existence and uniqueness of a well-concentrated giant component in the supercritical excursion sets of three important ensembles of spherical Gaussian random fields: Kostlan's ensemble, band-limited ensembles, and the random spherical harmonics. Our main results prescribe quantitative bounds for the volume fluctuations of the giant that are essentially optimal for non-monochromatic ensembles, and suboptimal but still strong for monochromatic ensembles.

Our results support the emerging picture that giant components in Gaussian random field excursion sets have similar large-scale statistical properties to giant components in supercritical Bernoulli percolation. The proofs employ novel decoupling inequalities for spherical ensembles which are of independent interest.
\end{abstract}
\date{\today}
\maketitle

\setcounter{tocdepth}{1}

\section{Introduction}
\label{s:intro}

\subsection{The giant component in Kostlan's ensemble}

\textit{Kostlan's ensemble} is the sequence of random homogeneous polynomials $f_{n}:\R^{m+1}\rightarrow\R$
\begin{equation}
\label{eq:fn Kostlan}
f_{n}(x) = \sum\limits_{|J|=n}\sqrt{{n\choose J}}a_{J}x^{J} \ , \quad n \ge 1,
\end{equation}
where $x=(x_{0},\ldots,x_{m})\in\R^{m+1}$, $J=(j_{0},\ldots,j_{m})\in\Z_{\ge 0}^{m+1}$, $x^{J} = x_{0}^{j_{0}}\cdot\ldots\cdot x_{m}^{j_{m}}$, and the $\{a_{J}\}$ are i.i.d.\ standard Gaussian random variables. The law of $f_{n}$ extended to $\C^{m+1}$ is the unique probability distribution over degree-$n$ Gaussian homogeneous polynomials that is unitary invariant. Hence its importance in quantum mechanics and complex algebraic geometry, where the zero set $f_{n}^{-1}(0)$ on the projective space $\C P^{m}$ is a natural model for a random algebraic variety, see~\cite[\S 1.1]{BMW RSW}.

\vspace{2mm}
Henceforth we fix the dimension $m=2$. Due to homogeneity, it is natural to restrict $f_{n}$ to the unit sphere $\Sc^{2}\subseteq \R^{3}$. Then $f_{n}:\Sc^{2}\rightarrow\R$ is a centred isotropic Gaussian random field with the covariance kernel
\begin{equation}
\label{e:kappa}
\kappa_{n}(x,y) = \E\left[ f_{n}(x)\cdot f_{n}(y) \right] = \left(  \left\langle x,y\right\rangle\right)^{n} ,
\end{equation}
where $\left\langle x,y\right\rangle = \cos(d_{\mathcal{S}^2}(x,y))$ is the inner product inherited from $\R^3$, and $d_{\Sc^2}$ denotes the spherical distance. Some elementary analysis~\cite[(1.5)]{BMW RSW}  reveals that $\kappa_{n}(x,y)$ is rapidly decaying around the diagonal at the scale $n^{-1/2}$, and is positive for $x,y$ on the same hemisphere.

\vspace{2mm}
For $t\in\R$ let $\Uc(t)=\Uc_{n}(t)$ be the excursion (or `sub-level') set
\begin{equation}
\label{e:sublevel}
\Uc(t):= f_{n}^{-1}((-\infty,t]) = \{x\in\Sc^{2}:\: f_{n}(x)  \le t\} .
\end{equation}
In this paper we are interested in the {\em percolative} properties of $\Uc(t)$. In particular we ask:

\vspace{2mm}
\textit{Does $\Uc(t)$ contain a unique, ubiquitous, giant component with well-concentrated volume?}
\vspace{2mm}

\noindent By a \textit{giant} component we mean one with volume comparable to that of the sphere, and by \textit{ubiquitous} we mean that it intersects every spherical cap of radius slightly larger than the `local scale' $n^{-1/2}$.

\vspace{2mm}
It has recently been understood that the level $t = 0$ is `critical' for the percolative properties of excursion sets of homogeneous Gaussian fields on the plane $\R^2$ and the sphere $\mathcal{S}^2$. In particular, the Russo-Seymour-Welsh (RSW) theory, developed in ~\cite{BG, BMW RSW}, shows that while $\Uc(0)$ likely contains components of diameter comparable to that of the sphere, these components typically have negligible area. To be more precise, letting\footnote{The `d' in `$\Vc^{d}$' stands for `diametric'. Later we will introduce another, {\em volumetric}, notion of the largest component, denoted $\Vc^{a}$, see \S \ref{sec:proof outline}. As a by-product of the analysis below, one has $\Vc^{a}=\Vc^{d}$ with almost full probability.} $\Vc^d(t)=\Vc^d_{n}(t)$ denote the component of $\Uc(t)$ of largest diameter, one can deduce from \cite{BMW RSW} that for all $\eps > 0$ there exists $c > 0$ such that, for $n$ sufficiently large
\begin{equation}
\label{e:nogiant}
 \prob \big( \textrm{diam}(\Vc^d(0)) > c \big) > 1-\eps \quad \text{and} \quad  \prob \big( \area(\Vc^d(0)) > \epsilon \big) < n^{-c} .
 \end{equation}
 By contrast, for subcritical levels $t < 0$ an extension of the methods in \cite{RV,MV20} (which apply to the local scaling limit, see \S \ref{sec:scal lim Euclid}) show that a giant component of $\Uc(t)$ is stretched-exponentially unlikely, in the sense that there exists $c > 0$ so that
\begin{equation}
\label{e:nogiant2}
 \prob \big(\diam(\Vc^d(t)) >  \eps \big) <  e^{-c n^{1/2}}  \quad \text{and} \quad  \prob \big( \area(\Vc^d(t)) > \epsilon \big) <   e^{- c n^{1/2} }  .
 \end{equation}

\vspace{2mm}
Our main result on Kostlan's ensemble asserts that, at levels $t>0$, $\Uc(t)$ {\em does} contain a unique ubiquitous giant component with well-concentrated volume:

\begin{theorem}[Giant component for Kostlan's ensemble]
\label{thm:unique giant Kostlan}
Let $f_{n}$ be as in \eqref{eq:fn Kostlan}, $t>0$, and $\Uc(t)$ and $\Vc^d(t)$ be the excursion set \eqref{e:sublevel} and its component of largest diameter respectively. Then there exists a constant $\vartheta=\vartheta(t) \in (0,1)$ such that $\Vc^d(t)$ satisfies the following properties:

\begin{enumerate}[i.]
\item {\bf Existence, uniqueness, and volume concentration.}
For every $\epsilon>0$ there exists a number $c=c(t,\epsilon)>0$ such that, outside of an event of probability $<\exp(-cn^{1/2})$,
\begin{equation*}
|\area(\Vc^d(t))-4\pi\cdot \vartheta(t) |< \epsilon,
\end{equation*}
and $\Uc(t) \setminus \Vc^d(t)$ contains no component of diameter $>\epsilon$.
\item {\bf Ubiquity and local uniqueness.} There exist $c_{1}=c_1(t) ,c_{2}=c_2(t) > 0$ such that, outside of an event of probability $<c_1 n^{-c_2}$, for every spherical cap $\Dc\subseteq \Sc^{2}$ of radius $r> c_1 \log{n}/\sqrt{n}$, the restriction $\Uc(t)\cap \Dc$ contains exactly one component of diameter $>\frac{r}{100}$, and this component is contained in $\Vc^d(t)$.
\end{enumerate}
\end{theorem}

We also establish stronger \textit{upper deviation} bounds on the area of the giant which, in addition, improve the bounds \eqref{e:nogiant} and \eqref{e:nogiant2} on the absence of the giant component at levels $t \le 0$.

\begin{theorem}[Upper deviations]
\label{thm:unique giant Kostlan2}
Let $t \in \R$ and $\Vc^d(t)$ be as in Theorem  \ref{thm:unique giant Kostlan}, and extend the definition of $\vartheta(t)$ as in Theorem \ref{thm:unique giant Kostlan} for all $t \in \R$ by setting $\vartheta(t) = 0$ if $t \le 0$. Then for every $\epsilon >0$ there exists a number $c=c(t,\epsilon)>0$ such that
\begin{equation*}
\prob \big( \area(\Vc^d(t)) > 4 \pi \cdot \vartheta(t) + \epsilon \big) < e^{-c n} .
\end{equation*}
\end{theorem}

The bounds $\exp(-c n^{1/2})$ and $\exp(-c n)$ on the concentration of $\area(\Vc^d(t))$ in Theorems \ref{thm:unique giant Kostlan} and \ref{thm:unique giant Kostlan2} respectively are optimal up to the constant in the exponent (see Proposition \ref{p:lowerbound}), and match known results for classical percolation models such as Bernoulli and Poisson--Boolean percolation (see \S \ref{sss:giant}). The difference in the decay orders $n^{1/2}$ and $n$ -- which in terms of the local scale $n^{-1/2}$ can be viewed as `surface-order' and `volume-order' respectively -- reflects the distinction between \textit{upper} and \textit{lower} deviations of $\area(\Vc^d(t))$: the former requires atypical behaviour of $\Uc(t)$ on a subset of the sphere of macroscopic volume, whereas the latter can occur if $\Uc(t)$ has atypical behaviour on a macroscopic \textit{loop} (for instance if $\Sc^{2}\setminus \Uc(t)$ contains a great circle, then $\Vc^d(t)$ is constrained to lie in a hemisphere, so its volume is likely to be atypically small).

\vspace{2mm}
The polynomial decay of the exceptional event in Theorem \ref{thm:unique giant Kostlan}(ii) is also optimal, but in this case one may trade-off this decay with the scales for which the conclusion is asserted. For example, if one restricts to macroscopic spherical caps (i.e.\ of size comparable to $\Sc^2$), then it is possible to show that the exceptional event has probability $<\exp(-cn^{1/2})$, as in Theorem \ref{thm:unique giant Kostlan}(i). The constant $1/100$ in Theorem \ref{thm:unique giant Kostlan}(ii) is arbitrary, and we have not attempted to optimise it.

\vspace{2mm}
The asymptotic density $\vartheta(t)$ appearing in Theorems \ref{thm:unique giant Kostlan} and \ref{thm:unique giant Kostlan2} is also the density of the (unique) unbounded component of the excursion set at level $t$ of the \textit{Bargmann-Fock} random field $h_{BF}$ on $\R^{2}$, which is the local scaling limit of $f_{n}$ around every reference point, see \S\ref{sec:scal lim Euclid}. The following proposition asserts some fundamental properties of $\vartheta$:

\begin{proposition}
\label{prop:theta}
The function $t \mapsto \vartheta(t)$ is continuous and strictly increasing on $\R_{> 0}$. Moreover,
\begin{equation*}
\lim\limits_{t\rightarrow 0}\vartheta(t) =0 \qquad \text{and} \qquad \lim\limits_{t\rightarrow \infty}\vartheta(t) =1 .
\end{equation*}
\end{proposition}

Our proof of Theorem \ref{thm:unique giant Kostlan}(i) (and also Theorem \ref{thm:unique giant Kostlan2} in the case $t > 0$) relies crucially on the continuity of $\vartheta$ on $\R_{> 0}$. The limit $\vartheta(t) \to 1$ as $t \to \infty$ means that, at large $t$, the non-giant components of the excursion set $\Uc(t)$ have negligible total area. The limit $\vartheta(t) \to 0$ as $t \to 0$ is related to the non-existence of a giant component in $\Uc(0)$, and should be thought of as the continuity of $\vartheta$ at $t=0$.

\subsection{The giant component in random spherical harmonics and band-limited ensembles}
For $\ell \in \N$ let $\{Y_{\ell,m}\}_{m=-\ell}^{\ell}$ be the standard basis of spherical harmonics. The \textit{random spherical harmonics} $T_{\ell}:\Sc^{2}\rightarrow\R$ is the sequence of Gaussian fields on $\Sc^2$
\begin{equation}
\label{eq:Tl spher harm}
T_{\ell}(x) = \frac{\sqrt{4 \pi}}{\sqrt{2\ell+1}}\sum\limits_{m=-\ell}^{\ell}a_{m}\cdot Y_{\ell,m}(x) \ , \quad \ell \ge 1,
\end{equation}
where the $\{a_{m}\}_{m=-\ell}^{\ell}$ are i.i.d.\ standard Gaussians. Equivalently, $T_{\ell}$ is the centred isotropic Gaussian field on $\Sc^{2}$ with covariance kernel
\begin{equation}
\label{eq:covar harm Legendre}
\E\left[ T_{\ell}(x)\cdot T_{\ell}(y)  \right] = P_{\ell}(\langle x,y\rangle),
\end{equation}
where $P_{\ell}$ is the Legendre polynomial of degree $\ell$. The random fields $T_{\ell}$ are the components of the (stochastic) Fourier expansion of every isotropic Gaussian field on $\Sc^{2}$, and therefore are important in diverse fields such as statistical mechanics and cosmology.

\vspace{2mm}
The \textit{band-limited ensembles (`random waves')} are obtained by superimposing random spherical harmonics of different degrees, as follows. For $\alpha\in [0,1]$, define
\begin{equation}
\label{eq:gell band-limited sum}
 g_\ell = g_{\alpha;\ell}(x) = C_{\ell}\sum\limits_{\ell'=\lfloor \alpha\ell\rfloor}^{\ell} \frac{\sqrt{2\ell'+1}}{\sqrt{4\pi}} T_{\ell'}(x),  \quad C_{\ell} = \frac{ \sqrt{4 \pi} }{ \sqrt{ (\ell+1)^2- \lfloor \alpha\ell\rfloor^2 }},
\end{equation}
where the $T_{\ell'}$ are independent random spherical harmonics, and the normalising constant $C_{\ell}$ is chosen so that $\E[g_{\ell}(x)^{2}]=1$ for every $x\in\Sc^{2}$. In the `monochromatic' regime $\alpha=1$, we understand the sum \eqref{eq:gell band-limited sum} as
\begin{equation}
\label{eq:gell mono sum}
g_{\ell}= g_{\beta;\ell}(x) = C_{\ell}\sum\limits_{\ell'= \ell -\lfloor \eta \rfloor}^{\ell}\frac{\sqrt{2\ell'+1}}{\sqrt{4\pi}} T_{\ell'}(x),  \quad C_\ell = \frac{\sqrt{4 \pi} } {  \sqrt{ (\ell+1)^2- (\ell - \lfloor \eta \rfloor)^2 }},
\end{equation}
where $\eta = \eta(\ell) =  \ell^\beta$ with some $\beta\in (0,1)$. The width of the energy window $\eta$ in \eqref{eq:gell mono sum} is chosen to be $\ell^{\beta}$ for notational convenience only, and could be replaced by a more general width function subject to a sufficient growth condition (see the remark after Theorem \ref{thm:spher harm}). Mind that \eqref{eq:gell mono sum} abuses the notation \eqref{eq:gell band-limited sum}, but the distinction will be clear in context.

\vspace{2mm}
More generally, one may introduce ~\cite[\S 1.1]{SW} band-limited ensembles on arbitrary manifolds by superimposing Laplace eigenfunctions corresponding to eigenvalues in a suitable energy window, with i.i.d.\ Gaussian coefficients. For a wide class of manifolds, the local limits of such band-limited ensembles are universal, depending only on the width of the energy window and not on the manifold nor the reference point \cite[\S 2.1-\S 2.2]{SW}.

\vspace{2mm}

The percolative properties of the excursion sets of random spherical harmonics and band-limited ensembles are much more challenging to analyse than for Kostlan's ensemble, because of the slower decay of correlations and their oscillatory nature. In particular, the naturally conjectured analogues of \eqref{e:nogiant} and \eqref{e:nogiant2} are not known to date. Despite these difficulties, our main result for the band-limited ensembles asserts the existence of a unique, ubiquitous, well-concentrated giant component, including in the most challenging monochromatic regime:

\begin{theorem}[Giant component for band-limited ensembles]
\label{thm:unique giant band-lim}
Let $\alpha\in [0,1)$ and $g_{\ell}$ be given by \eqref{eq:gell band-limited sum}, or $\alpha =1$, $\beta \in (0,1)$ and $g_{\ell}$ be given by \eqref{eq:gell mono sum}. Let $t>0$, and let $\Uc(t) = g_\ell^{-1}(( -\infty,t])$ and $\Vc^d(t)$ be the excursion set of $g_{\ell}$ and its component of largest diameter respectively. Define
\begin{equation}
\label{e:p}
p= \begin{cases}   1   &\alpha \in [0,1),  \\ \beta &\alpha=1, \beta \in (0,1).
\end{cases}
\end{equation}
Then there exists $\varphi=\varphi(\alpha,t) \in (0,1)$ (independent of $\beta$ for $\alpha=1$) such that $\Vc^d(t)$ satisfies the following properties:

\begin{enumerate}[i.]

\item \textbf{Existence, uniqueness, and volume concentration.}
For every $\epsilon,\delta>0$ there exists a number $c=c(t,\epsilon,\delta)>0$ such that, outside of an event of probability $<\exp(-c\ell^{p-\delta})$,
\begin{equation*}
|\area(\Vc^d(t))-4\pi\cdot \varphi(\alpha,t) |< \epsilon,
\end{equation*}
and $\Uc(t) \setminus \Vc^d(t)$ contains no component of diameter $>\epsilon$.

\item \textbf{Ubiquity and local uniqueness.}
For every $\delta>0$ there exist $c_1 = c_1(t,\delta) > 0$ and $c_2 = c_2(t,\delta)$ such that, outside an event of probability $<c_1 \exp\left( - \ell^{c_2}\right)$, for every spherical cap $\Dc\subseteq \Sc^{2}$ of radius $r > c_1  \ell^{-p+\delta}$, the restriction $\Uc(t) \cap \Dc$ contains exactly one component of diameter $>\frac{r}{100}$ and this component is contained in $\Vc^d(t)$.
\end{enumerate}
\end{theorem}

In the non-monochromatic regime $\alpha \in [0,1)$, the concentration bound $\exp(-\ell^{1+o(1)})$ in Theorem \ref{thm:unique giant band-lim} is `surface-order' in terms of the local scale $1/\ell$, and so is optimal up to the $o(1)$ term. Our bounds are weaker in the monochromatic regime $\alpha = 1$, and we do not expect the exponent $p = \beta < 1$ to be optimal; it is plausible that `surface-order' concentration bounds (i.e.\ with exponent $1$) continue to hold in this regime.

\vspace{2mm}
Similarly to Theorem \ref{thm:unique giant Kostlan2}, we also establish stronger upper deviation bounds for the giant which extend to the subcritical case $t < 0$ (note however that we exclude the critical case $t = 0$):

\begin{theorem}
\label{thm:unique giant bl2}
Let $t \in \R \setminus \{0\}$ and $\Vc^d(t)$ be as in Theorem \ref{thm:unique giant band-lim}, and extend the definition of $\varphi(\alpha,t)$ as in Theorem \ref{thm:unique giant band-lim} for all $ t \in \R \setminus \{0\}$ by setting $\varphi(\alpha,t) = 0$ if $t < 0$. Let $p$ be as in \eqref{e:p}. Then for every $\epsilon > 0$ there exists a number $c=c(t,\epsilon)>0$ such that
\begin{equation*}
\prob \big( \area(\Vc^d(t)) > 4 \pi \cdot \varphi(\alpha,t) + \epsilon \big) < \exp\big(-c\ell^{4p/3 } \big) .
\end{equation*}
\end{theorem}

\vspace{2mm}
Theorems \ref{thm:unique giant band-lim} and \ref{thm:unique giant bl2} exclude the most important and challenging case of `pure' spherical harmonics~\eqref{eq:Tl spher harm}. This is treated as a separate case by the following theorem, asserting the existence, uniqueness and concentration of a giant component of {\em the same} density $\varphi(1,t)$ as in Theorem \ref{thm:unique giant band-lim}, but with weaker concentration bounds, and without claiming the ubiquity and the local uniqueness properties:

\begin{theorem}[Giant component for random spherical harmonics]
\label{thm:spher harm}
Let $T_{\ell}$ be as in \eqref{eq:Tl spher harm}, and let $\Uc(t) =  T_\ell^{-1}( (-\infty,t])$ and $\Vc^d(t)$ be the excursion set of $T_{\ell}$ and its component of largest diameter respectively. Let $\varphi(1,t)$ be as in Theorem \ref{thm:unique giant band-lim}. Then for every $t > 0$ and $\epsilon,\delta>0$ there exists a number $c=c(t,\epsilon,\delta)>0$ such that
\begin{equation*}
\prob \big( |\area(\Vc^d(t))-4\pi\cdot \varphi(1,t) | \ge \epsilon \big) < \exp(-c \sqrt{ \log \ell})
\end{equation*}
and
\begin{equation}
\label{eq:spher harm uniqueness}
\prob \big( \text{$\Uc(t) \setminus \Vc^d(t)$ contains a component of diameter $>\epsilon$} \big) < c (\log \ell)^{2+\delta}/ \ell^{1/2} .
\end{equation}
Moreover, for every $t < 0$ and $\epsilon,\delta>0$ there exists a number $c=c(t,\epsilon,\delta)>0$ such that
\begin{equation}
\label{eq:spher harm subcrit}
\prob \big( \area(\Vc^d(t)) > \epsilon \big) < c (\log \ell)^{2+\delta} / \ell^{1/2}.
\end{equation}
\end{theorem}

The monochromatic cases $\alpha = 1$ of Theorems \ref{thm:unique giant band-lim} and \ref{thm:unique giant bl2} can be regarded as an `interpolation' between the band-limited cases $\alpha \in [0,1)$ of Theorem \ref{thm:unique giant band-lim} and \ref{thm:unique giant bl2} with stretched-exponential concentration bounds, and the random spherical harmonics of Theorem \ref{thm:spher harm}, where the concentration bounds are no longer stretched-exponential. Variants of Theorems \ref{thm:unique giant band-lim}, \ref{thm:unique giant bl2} and \ref{thm:spher harm} continue to hold for band-limited functions with more general energy windows $[ \ell - \eta,\ell]$ for $\eta = \eta(\ell)$. Indeed the assertions of Theorems \ref{thm:unique giant band-lim}, \ref{thm:unique giant bl2} only require that $\eta$ satisfy $\eta \ge \ell - \ell \alpha$ (in the case $\alpha \in [0,1)$) or $\eta \ge \ell^\beta$ (in the case $\alpha = 1$ and $\beta \in (0,1)$), and the weaker assertions in Theorem \ref{thm:spher harm} hold for arbitrary $\eta \ge 1$. Moreover, by suitably adapting our arguments, one could obtain super-polynomial upper deviations bounds for certain $\eta$ growing \textit{subpolynomially} in $\ell$: more precisely, we believe that the conclusion of Theorem \ref{thm:unique giant bl2} holds for  $\eta \gg \log \ell$ with error probability
\[  <  \begin{cases}  \exp (-c  \eta^{4/3} ) , & \text{if } \eta \gg (\log \ell)^3   , \\    \exp (-c  \eta^{3/2} (\log \ell)^{-1/2}  ) & \text{if }   \log \ell \ll \eta \ll (\log \ell)^3    .\end{cases}  \]
However our arguments for the lower deviation bounds in Theorem \ref{thm:unique giant band-lim} do not extend to the sub-polynomial case, since in that case the $o(1)$ term in the exponent overwhelms the bounds.

\vspace{2mm}
Similarly to Theorem \ref{thm:unique giant Kostlan}, the limiting density $\varphi(\alpha,t)$ in Theorems \ref{thm:unique giant band-lim}, \ref{thm:unique giant bl2} and \ref{thm:spher harm} coincides with the density of the (unique) unbounded component of the excursion set at level $t$ of the \textit{band-limited random plane waves} $h_\alpha$ on $\R^{2}$, which are the local scaling limits of $g_\ell$ (and also of $T_\ell$, for $\alpha = 1$) around every reference point, see \S\ref{sec:scal lim Euclid}. For example, for $\alpha=1$ the limiting field is the \textit{monochromatic random plane wave}. The following result concerns the fundamental properties of $\varphi(\alpha,t)$ as a bivariate function:

\begin{proposition}
\label{prop:phi}
The function $(\alpha,t) \mapsto \varphi(\alpha,t)$ is continuous on $[0,1]\times (0,\infty)$.
Moreover, for every $\alpha\in [0,1]$, $t\mapsto\varphi(\alpha,t)$ is strictly increasing on $\R_{> 0}$, and $\lim\limits_{t\rightarrow \infty}\varphi(\alpha,t) =1$.
\end{proposition}

As for Kostlan's ensemble, the proof of Theorem \ref{thm:spher harm} relies crucially on the continuity of $t \mapsto \varphi(t,\alpha)$ on $\R_{>0}$. Observe that Proposition \ref{prop:phi} does not prescribe the limiting behaviour of $\varphi(\alpha,t)$ as $t\rightarrow 0$. This is related to the fact that, for the band-limited random plane waves $h_\alpha$ (in particular, the notoriously difficult monochromatic case $h_1$), it is not known whether there is percolation at criticality. This is also the reason why, unlike for Kostlan's ensemble, the case $t=0$ is omitted in Theorem~\ref{thm:unique giant bl2}. The percolative properties of $\Uc(0)$ for these ensembles is a fundamental open problem.

\subsection{Discussion}
\label{sec:disc}

\subsubsection{Giant components in the supercritical regime}
\label{sss:giant} Motivated by questions in mathematical physics and spectral geometry, one is often interested in the geometry and topology of the {\em nodal components} (i.e.\ the connected components of $F^{-1}(0)$) of a smooth random field $F:\R^{m}\rightarrow \R$, $m \ge 2$. Similarly, one may study the \textit{nodal domains} -- the connected components of $\R^m \setminus F^{-1}(0)$ -- or more generally the components of the \textit{excursion sets} $F^{-1}((-\infty,t])$ for $t \in \R$.

\vspace{2mm}
A basic question is the \textit{number} of nodal components contained in a large domain such as a box of side length $R$. A seminal result of Nazarov and Sodin ~\cite{NaSoGen,SoSPB} is that, if $F$ is a smooth stationary Gaussian field satisfying some other mild assumptions, the number of components is \textit{volumetic}: there exists a deterministic constant $c_{NS}=c_{NS}(F)>0$ such that this number is $\sim c_{NS}\cdot R^m$ almost surely.

\vspace{2mm}
While this result theoretically holds in all dimensions, when simulating fields in large domains one observes a marked difference in the nodal topology in dimension $m=2$ compared to $m \ge 3$, especially for the \textit{macroscopic} components (i.e.\ of diameter comparable to the simulation scale). Indeed while for $m=2$ one typically sees a volumetric number of components, some of them macroscopic, for $m=3$ Barnett ~\cite{Barnett} numerically discovered that both the nodal set and the nodal excursions are dominated by a single \textit{giant} component consuming almost the entire the nodal structure ($>95$\%), with a very small number of other components (so $c_{NS}$ must be \textit{tiny} $\approx 10^{-5}$).

\vspace{2mm}
The possible existence of a giant {\em percolating} component in nodal and excursion sets, and their qualitative and quantitative properties, is reminiscent of questions in {\em percolation theory} initiated by Broadbent and Hammersley some $70$ years ago. The classical model is {\em Bernoulli percolation} defined by a lattice $\Lambda=\Z^{m}$, $m \ge 2$ (or any discrete graph of `sites'), and a probability $p\in [0,1]$ that a given bond (or site) is open, independently of all others. The connected components are the {\em clusters}, and, by the classical theory, there exists a {\em critical} probability $p_{c}(m)\in (0,1)$ (different for bond and site percolation) so that in the {\em supercritical} regime $p>p_{c}$ there is a.s.\ a unique unbounded cluster (`percolation occurs'), whereas in the {\em subcritical} regime $p<p_{c}$ all clusters are bounded. For bond percolation on $\Z^{2}$ a celebrated result of Kesten~\cite{Kesten} is that
\begin{equation}
\label{eq:Kesten pc=1/2}
p_{c}(2)=1/2 ,
\end{equation}
and moreover, by RSW theory~\cite{R RSW,SW RSW}, no percolation occurs {\em at criticality} $p = p_c = 1/2$. On the other hand, for $m \ge 3$ it is known \cite{cr85} that
\begin{equation}
\label{eq:d>=3 crit prob < 1/2}
p_c(m) < 1/2.
\end{equation}

In the supercritical regime $p>p_{c}$, one is interested in the properties of the unique unbounded cluster, in particular its asymptotic density in large domains. By ergodicity, as $R \to \infty$ the density $d_{R}$ of the unbounded cluster inside a box of size $R$ converges to $\theta(p) \in (0,1)$, the probability that the origin belongs to this cluster. As for finer properties, it was shown \cite{DS88,gan89,DP96} that this density has upper and lower deviations of \textit{volume-order} and \textit{surface-order} respectively, meaning that the density $d_R$ satisfies
\begin{equation}
\label{e:deviations}
\prob \big( d_R > \theta(p) + \eps \big) < e^{-cR^{m}}   \quad \text{and} \quad  \prob \big( d_R < \theta(p) - \eps \big) < e^{-cR^{m-1}}  .
\end{equation}
This has been extended to other percolation models, such as Poisson--Boolean percolation ~\cite{Penrose-Pisztora}.

\vspace{2mm}
Interest in the connections between nodal geometry and percolation theory was rekindled by Bogomolny-Schmit \cite{BoSch} who related, in $2$d, the number and area distribution of the nodal domains of the \textit{monochromatic random wave} (the local scaling limit of random spherical harmonics and band-limited ensemble with $\alpha=1$, see \S\ref{sec:scal lim Euclid}) to the corresponding properties of {\em critical} Bernoulli bond percolation clusters. They further argued that excursion sets at non-zero levels should correspond to {\em non-critical} percolation clusters, although they avoided the subtle questions around the existence and properties of the unbounded domains. One might take this one step further and conjecture, in the supercritical regime, the existence of a unique unbounded excursion set which satisfies similar fine properties (such as \eqref{e:deviations}) to the unbounded percolation cluster.

\vspace{2mm}
The existence of unbounded components of the nodal set (or excursion sets) of random fields was first addressed by Molchanov-Stepanov \cite{MS1,MS2,MS3}, who gave sufficient conditions for the existence of a {\em critical level} $t_c =t_c(F) \in (-\infty,\infty)$ associated to a
stationary Gaussian field $F:\R^{m}\rightarrow \R$ such that, if $t< t_c$, $F^{-1}( (-\infty,t])$ almost surely has bounded components, whereas if $t>t_c$, $F^{-1}((-\infty,t])$ contains an unbounded component a.s. Since, for every $x\in\R^{m}$, $\prob(F(x) \le 0) = 1/2$, by analogy with Kesten's \eqref{eq:Kesten pc=1/2} it is natural to conjecture that $t_c = 0$ for `generic' fields in dimension $m=2$. Recently this has been confirmed for a wide class of planar Gaussian fields~\cite{BG,RV,MRVK}, including the monochromatic random waves and the Bargmann-Fock field (the local scaling limit of Kostlan's ensemble, see \S\ref{sec:scal lim Euclid}). For the Bargman-Fock field it is further known that there is no percolation at criticality \cite{al96,BG}, but at the moment this is wide open for the monochromatic random waves.

\vspace{2mm}
Returning to Barnett's discovery that the empirical nodal structures in $3$d are dominated by a single giant component, Sarnak \cite{Sarnak ansatz} attributed this to the fact that, in dimension $m \ge 3$, the critical level $t_c$ associated to a random field $F:\R^{m}\rightarrow\R$ is \textit{strictly negative}, in accordance with \eqref{eq:d>=3 crit prob < 1/2}, so that $F^{-1}((-\infty,0])$ corresponds to supercritical percolation. Recently both the existence ~\cite{DCRRV} and uniqueness ~\cite{Se uniq} of an unbounded nodal component has been established rigorously, as well as certain finer aspects such as Gaussian fluctuations for its volume ~\cite{McAuley}. These results apply to a somewhat restrictive class of fields in dimensions $m \ge 3$ with positive and rapidly decaying correlations, which includes the Bargmann-Fock field, but not the monochromatic waves (for the latter there are some results in {\em sufficiently large} dimension~\cite{Rivera monochromatic}). As yet there are no results, analogous to \eqref{e:deviations}, on concentration properties of the unbounded component.

\subsubsection{Giant components on the sphere} So far our discussion has been limited to \textit{Euclidean} fields, for which the supercritical regime is characterised by \textit{unbounded} nodal and excursion set components. In mathematical physics and spectral geometry one is often interested in fields on \textit{compact} manifolds such as the sphere~$\Sc^2$. In that context, the components of the nodal and the excursion sets are all automatically bounded. Nevertheless, when simulating these fields, one still observes a \textit{supercritical} regime characterised by \textit{giant} components of macroscopic volume, see Figure~\ref{f:1}.

\begin{figure}[h]
  \centering
  \includegraphics[height=6cm]{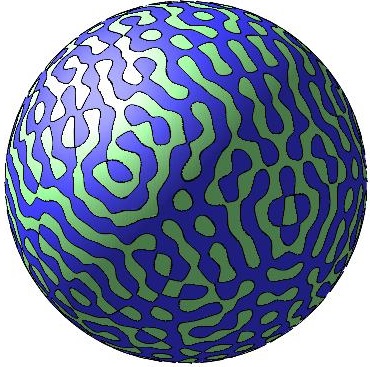}
  \hspace{0.4cm}
\includegraphics[height=6cm]{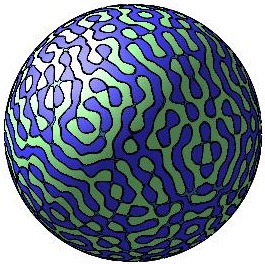}
\caption{Excursion sets of a randomly drawn spherical harmonic $T_\ell$ of degree $\ell = 40$ at levels $t=-0.1$ (left) and $t=+0.1$ (right). The green area is $T_\ell^{-1}((-\infty,-0.1))$, the green shaded area on the right is $T_\ell^{-1}((-\infty,0.1)) \setminus T_\ell^{-1}((-\infty,-0.1))$. Left: All green components are relatively small. Right: Considered together with the green shaded area, there is a now a unique giant green component, with all other green components small.}
\label{f:1}
\end{figure}

\vspace{2mm}
 If a spherical field has a local Euclidean scaling limit (which is the case for the spherical ensembles we study, see \S\ref{sec:scal lim Euclid}), one can attempt to explain such giant components as reflecting the presence of unbounded components in the scaling limit field. However this argument only goes so far: since one may only couple the field and its scaling limit on \textit{mesoscopic} scales, at best one may deduce the existence of `local giants' in spherical caps of mesoscopic size. Then one needs to argue that the various `local giants' merge into a {\em unique} global giant spanning the entire sphere. Similarly, one cannot deduce concentration properties of the giant akin to \eqref{e:deviations} solely using the analogues in the scaling limit.

 \vspace{2mm} In this paper we address the existence, uniqueness and concentration properties of the giant component in Kostlan's ensemble, the random spherical harmonics, and band-limited ensembles. While we use the Euclidean theory of their scaling limits, and the associated techniques, as input, these are far from sufficient to yield the non-Euclidean results: even the mere existence, with high probability, of a giant component on the sphere does not follow from Euclidean methods. In addition our results go beyond even the current Euclidean theory in two important aspects:
  \begin{enumerate}
 \item We establish concentration properties akin to \eqref{e:deviations}, which were not yet established for the Euclidean scaling limits. For Kostlan's ensemble our concentration results match exactly the order of concentration in Bernoulli percolation \eqref{e:deviations}. For non-monochromatic band-limited ensembles, our concentration results are still surface-order, despite the slowly decaying and oscillatory correlations. For random spherical harmonics our concentration results are weaker, but still quantified.
 \item We develop specific tools to handle the random spherical harmonics and band-limited ensembles, whose local scaling limits lie well outside the current Euclidean theory (once again, due to the slow decay and oscillatory nature of their correlations).
  \end{enumerate}

 \vspace{2mm}
Similar questions have previously been studied for Bernoulli percolation on sequences of finite graphs, such as the complete graph, hypercubes, and Euclidean tori (see \cite{EH21} and references therein). In particular, the recent work \cite{EH21} on general sequence of finite transitive graphs has inspired our analysis of the random spherical harmonics. Closer to the setting of the current paper, several recent works have studied supercritical excursion sets of the (zero-averaged) Gaussian free field on a sequence of finite graphs such as tori \cite{aba19} and the complete graph \cite{con23}, establishing the existence, uniqueness, and limiting density of a unique giant component, but not concentration properties akin to \eqref{e:deviations}.

 \subsection{Overview of the remainder of the paper}
 In section \ref{s:outline} we discuss preliminaries, including the local convergence of the spherical ensembles, and percolative properties of the local scaling limits. We also give in section \ref{sec:proof outline} a brief outline of the proof our concentration estimates. In section \ref{s:fra} we develop the finite-range approximation and sprinkled decoupling estimates for spherical ensembles, which is of independent interest. In section \ref{s:lu} we develop a theory of local uniqueness based on renormalisation of crossing events. In section \ref{s:mr} we prove most of our main results, including the existence and uniqueness of the giant component, and also the concentration results for Kostlan's ensemble and the band-limited ensembles using the renormalisation procedure outlined in section \ref{sec:proof outline}. In section \ref{s:rsh} we use a separate argument to establish concentration properties for the random spherical harmonics. Finally, in section \ref{s:rr} we prove the remaining results, Propositions \ref{prop:theta} and \ref{prop:phi}, and also establish the optimality of the volume concentration for Kostlan's ensemble in Theorems \ref{thm:unique giant Kostlan}-\ref{thm:unique giant Kostlan2}.

 \subsection*{Acknowledgements}

S.M. is supported by the Australian Research Council Discovery Early Career Researcher Award DE200101467. The authors are grateful to Peter Sarnak for many stimulating discussions, and sharing his insight into the discrepancies between the nodal structures in $2$d and higher dimensions, and to Michael McAuley for comments on an earlier draft.

\medskip

\section{Preliminaries and outline of the proof}
\label{s:outline}

\subsection{Preliminaries}

We begin by recalling some preliminary properties of the spherical ensembles under consideration, namely their local scaling limits and certain percolative and stability properties of these limits.

\subsubsection{Local scaling limits}
\label{sec:scal lim Euclid}
Let $F_n$ be a sequence of smooth centred isotropic (i.e.\ invariant w.r.t.\ rotations) Gaussian fields on the sphere $\Sc^2$, and $F_\infty$ be a centred smooth stationary Gaussian field on $\R^2$ with covariance kernel $K_\infty(x,y)$. We say that $F_\infty$ is \textit{non-degenerate} if $(F_\infty(0),\nabla F_\infty(0))$ is a non-degenerate Gaussian vector; if $F_\infty$ is isotropic this is equivalent to $F_\infty$ not being constant. Let $s_n \to 0$ be a sequence of positive numbers (`scales'). For a reference point $P \in \Sc^2$, let $\exp_P:T_P(\Sc^{2})\rightarrow\Sc^{2}$ be the exponential map based at $P \in \Sc^2$.

\begin{definition}[Local scaling limit]
\label{d:lsl}
We say that $F_n$ \textit{converges locally to} $F_\infty$ \text{at the scale} $s_n$ if for every reference point $P \in \Sc^2$, the {\em scaled covariance} on $T_P(\Sc^2)$
\begin{equation}
\label{eq:Kn Kostlan->Kinf BF}
K_{n}(x,y) = K_{n;P}(x,y) = \E \big[F_{n} ( \exp_{P} (s_n x ) )\cdot F_{n}( \exp_{P}(s_n y) ) \big]
\end{equation}
satisfies $K_n(x,y) \to K_\infty(x,y)$ in the $C^{\infty,\infty}$-norm on every compact subset of $\R^2 \times \R^{2}$. By the rotational invariance, $K_n$ is independent of both the reference point $P$ and the identification $\R^{2}\cong T_P(\Sc^{2})$.
\end{definition}

We denote by $B_R\subseteq \R^{2}$ the Euclidean ball centred at the origin. An important consequence of Definition \ref{d:lsl} is the following:
\begin{proposition}
\label{p:lsl}
Suppose $F_n$ converges locally to $F_\infty$ at the scale $s_n$. Then for every reference point $P \in \Sc^2$
there exists a coupling of $(F_n)_{n \ge 1}$ and $F_\infty$ such that, for every $k,R,\eps> 0$, as $n \to \infty$,
\[  \prob \big( \|F_n( \exp_P(s_n \cdot) ) - F_\infty(\cdot) \|_{C^k(B_R)} > \eps  \big) \to 0  \]
uniformly with respect to $P$.
\end{proposition}
\begin{proof}
Definition \ref{d:lsl} implies that $F_n( \exp_P(s_n \cdot) ) \to F_\infty(\cdot)$ in law in the $C^k$-topology on compact sets, and the claim follows by Skorokhod's representation theorem.
\end{proof}

\noindent The spherical ensembles introduced above satisfy Definition \ref{d:lsl}:

\begin{examples}
\label{e:lsl}
$\,$
\begin{enumerate}
\vspace{0.2mm}
\item \textit{Kostlan's ensemble.} $F_n$ converges locally to $h_{BF}$ at the scale  $s_n = 1/\sqrt{n}$, where $h_{BF}$ is the \textit{Bargmann-Fock} field with the covariance kernel $e^{-\|x-y\|^{2}/2}$ (see e.g.\ ~\cite[\S 1.3]{BMW RSW}).

\vspace{0.2mm}
\item \textit{Random spherical harmonics.} $T_\ell$ converges locally to $h_1$ at the scale $s_\ell = 1/\ell$, where $h_1$ is the \textit{monochromatic random plane wave} with covariance kernel $J_0(\|x-y\|)$. Note that $J_0(\cdot) = \mathcal{F}[d\mu_1](\cdot)$ where $\mathcal{F}$ denotes the Fourier transform on $\R^{2}$, and $\mu_1$ is the uniform measure on the unit circle $\Sc^{1}\subseteq \R^{2}$.

\vspace{0.2mm}
\item \textit{Band-limited ensembles.} Let $\alpha \in [0,1]$. Then $g_\ell$ converges locally to $h_\alpha$ at the scale $s_\ell = 1/\ell$, where $h_\alpha$ is the  \textit{band-limited random plane wave} with covariance kernel $\mathcal{F}[ d\mu_\alpha ](\|x-y\|)$, and $\mu_\alpha$, $\alpha \in [0,1)$, is the uniform measure on the annulus $B_1 \setminus B_\alpha$, and $\mu_1$ is as in the previous case.
\end{enumerate}

\end{examples}

\subsubsection{Percolative properties of the local scaling limits}
\label{sec:perc prop loc}
We will make use of the following known results.  Let $h$ be any of the scaling limits in Example \ref{e:lsl}, i.e.\ $h \in \{ h_{BF}, (h_\alpha)_{\alpha \in [0,1]}\}$, and abbreviate $\{h \le t\} = \{ x \in \R^2 : h(x) \le t \}$. For $t \in \R$ and $r > 0$, define the \textit{annulus crossing event} $\textrm{AnnCross}_\infty(t, r)$  that $[-r,r]^2$ is connected to $\partial [-2r,2r]^2$ in $\{ h \le t\}$.

\begin{proposition}[\cite{RV} (for $h_{BF}$), \cite{MRVK} (for $h_\alpha$)]
\label{p:limitfield}
We have
 \[  \prob \big( \{  h \le t \} \text{ contains an unbounded component} \big) = \begin{cases}   1 & \text{if } t > 0, \\ 0 & \text{if } t < 0, \text{ or if } t = 0 \text{ and } h = h_{BF}, \end{cases} \]
 and if $t > 0$ the unbounded component is unique almost surely. Moreover,
 \[ \lim_{r \to \infty} \prob [ \textrm{AnnCross}_\infty(t,r) ]  =  \begin{cases} 1 & \text{if } t > 0, \\ 0 & \text{if } t < 0, \text{ or if } t = 0 \text{ and } h = h_{BF},\end{cases}  \]
with the convergence uniform over $h \in \{h_{BF},(h_\alpha)_{\alpha \in [0,1]} \}$.
\end{proposition}
Quantitative estimates on $\prob[ \textrm{AnnCross}_\infty(t,r)]$ are also available \cite{BG,RV,MRVK,MV20}, but we shall not need them.

\subsubsection{Stability}
\label{ss:stab}
Let $h$ be a centred non-degenerate smooth isotropic Gaussian field on $\R^2$. Recall the event $\textrm{AnnCross}_\infty(t,r)$ from section \ref{sec:perc prop loc}, and define the \textit{arm event}
\begin{equation}
\label{eq:Arm event def}
\textrm{Arm}_\infty(t,r) = \{ 0 \textit{ is connected to } \partial B_r \text{ in } \{h  \le t\}  \} .
\end{equation}
Note that both $\textrm{AnnCross}_\infty(t,r)$ and $\textrm{Arm}_\infty(t,r)$ are non-decreasing in $t$, whereas $\textrm{Arm}_\infty(t,r)$ is also non-increasing in~$r$.

\begin{proposition}[Stability]
\label{p:stab}
Let $t \in \R$ and $r,\delta > 0$. Then there exists $\eps > 0$ such that
\[  \prob \big( \textrm{AnnCross}_\infty(t,r)  \setminus \textrm{AnnCross}_\infty(t-\eps,r) \big) < \delta  , \]
and
\[  \prob \big( \textrm{Arm}_\infty(t,r)  \setminus \textrm{Arm}_\infty(t-\eps,r+\eps) \big) < \delta  . \]
\end{proposition}
\begin{proof}
By the said monotonicity properties, it suffices to prove that $$\textrm{AnnCross}_\infty(t,r)  \setminus \bigcup\limits_{\eps > 0} \textrm{AnnCross}_\infty(t-\eps,r)$$ and
$$\textrm{Arm}_\infty(t,r)  \setminus \bigcup\limits_{\eps > 0} \textrm{Arm}_\infty(t-\eps,r+\eps)$$ have probability zero. These events imply, respectively, the existence of a stratified critical point of value $t$ in $[0,2r]^2 \setminus [0,r]^2$ (i.e.\ either a critical point of $h$ in the interior of $[0,2r]^2 \setminus [0,r]^2$, or of $h$ restricted to a boundary segment of this domain, or a corner) or
in $B_r$ (i.e.\ either a critical point of $h$ in the interior of $B_r$, or of $h$ restricted to $\partial B_r$). All of these have probability zero by Bulinskaya's lemma \cite[Proposition 6.11]{aw09}.
\end{proof}

\subsection{Outline of the proof}
\label{sec:proof outline}
The most challenging aspect of our results is the quantitative concentration of the giant's volume. For this we use a renormalisation argument combined with a finite-range approximation and sprinkled decoupling procedure; since the arguments differ depending on the ensemble, and whether we consider upper or lower deviations, for simplicity we outline the proof only for the upper deviations in Kostlan's ensemble.

\vspace{2mm}
Throughout the proof we work with the \textit{volumetric} notion of a giant, rather than the \textit{diametric} notion introduced above, i.e.\ we will consider the component $\Vc^a(t)$ of $\Uc(t)$ of largest \textit{volume} rather than \textit{diameter}. This is to exploit the monotonicity of $t \mapsto \textrm{Area}(\Vc^a(t))$, which is not necessarily true for the diametric giant. Since, by definition,  $\area(\Vc^a(t)) \ge \area(\Vc^d(t))$, the upper deviation bounds on $\Vc^a(t)$ transfer directly to $\Vc^d(t)$.

\vspace{0.2cm}
\noindent There are four steps to the proof:

\begin{enumerate}[1.]

\item \textbf{Qualitative concentration.}
Fix $t ,\epsilon > 0$. For $0 \le u \le \pi/2$ let $\Sc_u$ be a `square' of side-length $2u$ centred at the north pole (see \S \ref{s:dfp} for the precise definition), and let $\Vc^a_u$ denote the component of $\Uc(t) \cap \Sc_u$ of largest volume; this acts as a `local proxy' for $\Vc^a(t)$. The first step is to establish the qualitative concentration estimate
\begin{equation}
\label{e:qualconc}
 \prob \big( \area(\Vc^a_u) > (\vartheta(t) +\epsilon) \area(\Sc_u) \big) \to 0
 \end{equation}
as $u/\sqrt{n} \to \infty$. This is achieved with a soft ergodic argument using the local convergence of Kostlan's ensemble to the Bargmann-Fock field. It is at this stage that we identify $\vartheta(t)$ with the density of the unique unbounded component of $\{h_{BF} \le t \}$ guaranteed to exist by Proposition \ref{p:limitfield}.

\item \textbf{Setting up the renormalisation.}  Next we consider two scales $ 1/\sqrt{n} \le u \ll v \le \pi$, and claim that an upper deviation event at the scale $v$ implies $\delta (v/u)^2$ well-separated `copies' of this event on the smaller scale $u$, for some small $\delta > 0$. More precisely, we first `tile' $\Sc_v$ with rotated copies of $\Sc_u$ in such a way that the number of copies is at most $c (v/u)^2$, and the proportion of $\Sc_v$ that is covered by multiple copies of $\Sc_u$ is negligible (note that the curvature of $\mathcal{S}^2$ prevents exact tiling). Suppose that $S_v$ has an upper deviation in its total area, meaning that $\area(\Vc^a_v) > (\vartheta(t) +\epsilon) \area(\Sc_v)$. Then we show that there exist $\delta, \epsilon'> 0$ and a collection of at least $\delta (v/u)^2$ of the copies $S^i_u$ of $\Sc_u$ such that $\area(\Vc^a_v \cap S^i_u) > (\vartheta(t) + \epsilon') \area(\Sc_u) $. By suitable extraction, we may find such a collection which is mutually separated by distance $u$. Moreover, assuming that at least a positive fraction of this collection satisfies a certain \textit{local uniqueness} property, we may substitute $\area(\Vc^a_v \cap S^i_u) $ in the definition of this collection with the area of the largest component of $\Uc(t) \cap S_u^i$. This gives the claim, modulo the reduction of the deviation from $\eps$ to $\eps'$.

\item \textbf{Decoupling the renormalisation.} We now wish to decouple the upper deviation events at the scale $u$, which we use a finite-range approximation and a sprinkling procedure. More precisely, we show in \S \ref{s:fra} that Kostlan's ensemble can be coupled with an isotropic Gaussian field on $\Sc^2$ which has dependency range bounded by $u$, up to the addition of a small error field which stays below a small fixed threshold $\delta \in (0,t)$ on every local patch with overwhelming probability $\approx 1 - \exp ( - e^{c_\delta (u \sqrt{n})^2} )$. Taking advantage of the monotonicity properties of the volumetric giant, by using this approximation and slightly raising the level $t \mapsto t + \delta$ (`sprinkling') we decouple the deviation events at the scale $u$. This sets up a renormalisation equation
\[  P(v; t, \eps) \le  P(u; t - \delta, \epsilon')^{c (v/u)^2 } + e^{- c (u \sqrt{n}) \cdot(v/u)^2 }  + e^{ - (v/u)^2 e^{c_\delta (u \sqrt{n})^2}} , \]
where $P(v; t, \eps) = \prob( \area(\Vc^a_v) > (\vartheta(t) +\epsilon) \area(\Sc_v) )$, and where the second and third terms on the right-hand side arise from the failure of local uniqueness and the error in the decoupling procedure. We have omitted an additional `union bound' factor relating to the choice of $S^i_u$ for which deviations occurs, but this turns out to be negligible.

\item \textbf{Implementing the renormalisation.} To conclude we implement the renormalisation at the scales $u = c_0/\sqrt{n}$ and $v = \pi/2$, where $c_0 > 0$ is chosen sufficiently large. Since $P(u; t -\delta, \epsilon') \to 0$ by the qualitative concentration in Step 1, we conclude that $P(v; t, \eps) \le e^{-cn}$. This is not quite the conclusion we are after, since it shows concentration in the square $\Sc_{\pi/2}$, but it is simple to modify the argument to obtain the conclusion for $\Sc^2$.
\end{enumerate}

For lower deviations we proceed similarly, except in Step 2 we argue only that the deviation event requires the failure of local uniqueness on $\asymp v/u$ squares on the smaller scale (as opposed to $\asymp (v/u)^2$), which makes use of a simple isoperimetric estimate. The relevant renormalisation equation is then
\[  P'(v; t, \eps) \le  P'(u; t - \delta, \epsilon')^{c (v/u)^2 } + e^{- c (u \sqrt{n}) \cdot(v/u) }  + e^{ - (v/u)^2 e^{c_\delta (u \sqrt{n})^2}} , \]
where $P'(v; t, \eps) = \prob( \area(\Vc^a_v) < (\vartheta(t) -\epsilon) \area(\Sc_v) )$. Due to the smaller order in the exponent of the the term $e^{- c (u \sqrt{n}) \cdot(v/u) }$ compared to the previous case, the extra union bound factor that arises is no longer negligible if we work at the scales $u = c_0 / \sqrt{n}$ and $v =\pi/2$. Instead we choose scales $u =  n^{-1/4}$ and $v = \pi/2$, and the output is the concentration bound for $\Vc^a(t)$ of order $e^{-c n^{1/2}}$. While this bound does not transfer immediately to $\Vc^d(t)$, our uniqueness results (proved separately) imply that $\Vc^d(t) = \Vc^a(t)$ outside an event of the same order, allowing for the transfer.

For the band-limited ensembles the proof is similar, except the slower rate of covariance decay leads to weaker decoupling error in the renormalisation equation, and consequentially weaker concentration bounds. More precisely, our renormalisation equations are
\[  P(v; t, \eps) \le  P(u; t - \delta, \epsilon')^{c (v/u)^2 } + e^{- c (u \ell^{p-\delta'}) \cdot(v/u)^2 / (\log \ell )}  + e^{ - c_\delta (v/u)^{3/2} (u \ell^{1-p}) }  \]
for upper deviations, and
\[  P'(v; t, \eps) \le  P'(u; t - \delta, \epsilon')^{c (v/u)^2 } + e^{- c (u \ell^{p-\delta'}) \cdot(v/u)  /( \log \ell)}  + e^{ - c_\delta (v/u)^{3/4} (u \ell^{1-p})  } \]
for lower deviations, where $p \in \{1,\beta\}$ is as in \eqref{e:p}, and $\delta' > 0$ is arbitrarily small. For upper deviations we again use a one-step renormalisation, however for lower deviations we further optimise by \textit{iterating} the renormalisation along a carefully chosen sequence of scales.

For the spherical harmonics the decay of correlations is so slow that the renormalisation approach fails completely. Instead we exploit the symmetry of the sphere and a hypercontractivity argument to upgrade the qualitative concentration in \eqref{e:qualconc}; this is inspired by an argument in the setting of Bernoulli percolation due to Easo and Hutchcroft~\cite{EH21}.

\subsection{Remarks}
Let us record a few important subtleties which we glossed over in this outline:

\begin{enumerate}
\item With one important exception, our arguments all rely crucially on the fact we work in \textit{two dimensions}. Recall that we make use of estimates on the failure of a local uniqueness property (in Steps 1 and 3). In two dimensions we obtain such estimates by renormalising \textit{crossing events} akin to those appearing in Proposition \ref{p:limitfield}, combined with the same decoupling arguments described in Step 3. This has the added benefit of defining a proxy for local uniqueness which is an increasing event; this is crucial when we apply the sprinkling in Step 3. While it is natural to expect that a similar local uniqueness property holds in all dimensions, currently this is not known. Even for Euclidean fields it has so far only been shown for the Gaussian free field \cite{dgrs23}.

The notable exception is our proof of the qualitative concentration (Step 1) of upper deviations of the giant's volume (see the first statement of Proposition \ref{p:qualcon}), whose argument only relies on the ergodicity of the local limit and not on local uniqueness, and so is valid in all dimensions.

\item Our arguments exploit two crucial properties of the functional $\area(\Vc^a(t))$ and its proxies $\area(\Vc^a_u)$: (i) they are monotone w.r.t.\ the level and the field; and (ii) they satisfy the deterministic bounds
\begin{equation}
\label{e:detbounds}
 0 \le  \area(\Vc^a_u)  \le \area(\Sc_u) .
 \end{equation}
Indeed the monotonicity is crucial in applying sprinkling in Step 3, and the bounds in \eqref{e:detbounds} are used in Step 2 (the lower bound for upper deviations, the upper bound for lower deviations) in order to replicate the exceedence event at lower scales. This makes it non-trivial to adapt our argument to other geometric or topological functionals of the giant component, such as its boundary length or Euler characteristic.
\end{enumerate}


\medskip

\section{Finite-range approximation and decoupling for spherical ensembles}
\label{s:fra}

In this section we establish the decoupling inequalities for spherical ensembles, which are of independent interest. First we obtain a finite-range approximation for Kostlan's ensemble using an approach tailored to this ensemble. Then we present a general method of constructing finite-range approximations for isotropic Gaussian fields, which we apply to the band-limited ensembles. Finally we apply the finite-range approximations to obtain the stated decoupling inequalities. For the random spherical harmonics we rely instead on a general decoupling estimate proven in~\cite{m23}.

\subsection{Uniform bounds on the covariance kernels}
We begin by asserting a uniform bound on the covariance kernel and its derivatives for the spherical ensembles under consideration. Since an isotropic kernel $K(x,y)$ on $\Sc^2 \times \Sc^2$ only depends on the spherical angle $\theta = d_{\Sc^2}(x,y)$, we often view it as a univariate function of $\theta \in [0,\pi]$.

\smallskip
 For Kostlan's ensemble and the random spherical harmonics our bounds are standard. Recall that the covariance kernels of Kostlan's ensemble \eqref{e:kappa} and the random spherical harmonic $T_\ell$ \eqref{eq:Tl spher harm} are, respectively, $\kappa_n(\theta)  = (\cos \theta)^n$, and the Legendre polynomial $P_\ell(\theta)$.

\begin{lemma}
\label{l:ubke}
 For every $n \ge 1$,
\[ | \kappa_n(\theta) | = |(\cos \theta)^n|   \le   \begin{cases}  e^{- \theta^2 n/4} &  0 \le \theta \le \pi/2,  \\
e^{-\tilde{\theta}^2 n/4 } & 0 \le  \tilde{\theta} = \pi-\theta \le \pi/2.
\end{cases} \]
\end{lemma}
\begin{proof}
 This follows from $\cos(\theta) \le  1 - \theta^2/4$ for $\theta \in [0,\pi/2]$ and $\log(1-x) \le -x$ for $x > 0$.
 \end{proof}

\begin{lemma}
\label{l:ubrsh}
There exists a constant $c > 0$ such that, for every $\ell \ge 1$,
\[ \max \Big\{  |P_\ell(\cos \theta) | , \ell^{-1} |P'_\ell(\cos \theta) |\Big\}   \le   \begin{cases} c \theta^{-1/2} \ell^{-1/2}  &  0 \le \theta \le \pi/2,  \\
 c  \tilde{\theta}^{-1/2} \ell^{-1/2} & 0 \le  \theta' = \pi-\theta \le \pi/2.
\end{cases} \]
\end{lemma}
\begin{proof}
Hilb's asymptotics \cite[Theorem 8.21.6]{szego} state that
\[ P_\ell(\cos \theta)   = (\theta/ \sin \theta)^{1/2} J_0( (\ell+1/2)\theta) + O(\ell^{-1/2} )  \ , \quad 0 \le \theta \le \pi/2 , \]
uniformly over $\theta \in [0,\pi/2]$, where $J_0$ is the order-$0$ Bessel function. The bound on $P_\ell$ follows from $|J_0(x)| \le c x^{-1/2}$ \cite[(1.71.11)]{szego} and the symmetry in $\theta \mapsto \pi-\theta$. For the bound on $P'_\ell$, see the proof of Lemma \ref{l:ubble} below (and recall that $P_\ell$ is equal to the Jacobi polynomial $P^{(0,0)}_\ell$).
\end{proof}

We next derive the corresponding bounds for the covariance kernel $\Gamma_\ell$ of the band-limited ensembles $g_\ell$ in \eqref{eq:gell band-limited sum}. Specifically, in the case $\alpha \in [0,1)$,
\begin{equation}
\label{eq:Gamma sum}
\Gamma_\ell(\theta) =  C_\ell^2 \sum_{\ell' = \lfloor \alpha \ell \rfloor}^\ell N_\ell P_\ell( \cos \theta) ,
\end{equation}
where $C_\ell$ is defined in \eqref{eq:gell band-limited sum} and $N_\ell = (2\ell+1)/(4\pi)$, whereas in the case $\alpha = 1$ the summation range in \eqref{eq:Gamma sum} is replaced with $ \ell -\lfloor \ell^\beta \rfloor$ and $C_{\ell}$ is defined as in \eqref{eq:gell mono sum}. We observe that if $\alpha \in [0,1)$ then, as $\ell \to \infty$,
\begin{equation}
\label{e:cbounds}
C^2_{\ell} =  4 \pi ((\ell+1)^2 - \lfloor \alpha \ell \rfloor^2)^{-1} \sim  4\pi/(1-\alpha^2) \cdot  \ell^2,
\end{equation}
 whereas if $\alpha = 1$ and $\beta \in (0,1)$ then
\begin{equation}
\label{e:cbounds2}
 C^2_{\ell} =  4 \pi ((\ell+1)^2 - (  \ell  - \lfloor \ell^\beta  \rfloor)^2)^{-1}  \sim c _\beta / \ell^{1+\beta} .
 \end{equation}

\begin{lemma}
\label{l:ubble}
Suppose $\alpha \in [0,1)$. Then there exists a constant $c_\alpha > 0$ such that
\[ \max \big\{ |  \Gamma_\ell(\theta) |  ,  \ell^{-1}   |  \Gamma_\ell'(\theta) | \big\} \le   \begin{cases} c_\alpha \theta^{-3/2} \ell^{-3/2}   &  0 \le \theta \le \pi/2,  \\
  c_\alpha \tilde{\theta}^{-1/2} \ell^{-3/2}   & 0 \le  \tilde{\theta} = \pi-\theta \le \pi/2.
\end{cases} \]
Suppose $\alpha = 1$ and $\beta \in (0,1)$. Then there exists a constant $c_\beta > 0$ such that
\[  \max \big\{ |  \Gamma_\ell(\theta) |  ,  \ell^{-1}   |  \Gamma'_\ell(\theta) | \big\} \   \le   \begin{cases}  c_\beta  \theta^{-3/2} \ell^{-1/2-\beta}   & 0 \le \theta \le \pi/2,  \\
 c_\beta   \tilde{\theta}^{-1/2} \ell^{-1/2-\beta}   & 0 \le  \tilde{\theta} = \pi-\theta \le \pi/2.
\end{cases} \]
Finally, there exists an absolute constant $c > 0$ such that, for all $\alpha \in [0,1]$,
\[  \max \big\{ |  \Gamma_\ell(\theta) |  ,  \ell^{-1}   |  \Gamma'_\ell(\theta) | \big\} \   \le    \begin{cases} c \theta^{-1/2} \ell^{-1/2}   &  0 \le \theta \le \pi/2,  \\
  c \tilde{\theta}^{-1/2} \ell^{-1/2}   & 0 \le  \tilde{\theta} = \pi-\theta \le \pi/2.
  \end{cases} \]
\end{lemma}
\begin{proof}
For integers $0 \le \ell_0 < \ell$ let
 \[ \Gamma_{\ell,\ell_0}(\theta) =  C^2_{\ell,\ell_0}  \sum_{\ell' \ge \ell_0}^\ell   N_{\ell'}  P_{\ell'}(\cos \theta)   \]
where $C^2_{\ell,\ell_0} = (\sum_{\ell' \ge \ell_0}^\ell N_{\ell'})^{-1} = 4 \pi ((\ell+1)^2 - (\ell_0)^2)^{-1}$. Let $P_\ell^{(u,v)}$ denote the Jacobi polynomials. Using the Christoffel-Darboux formula \cite[(4.5.3)]{szego}
\[ \sum_{\ell'=0}^\ell N_{\ell'} P_{\ell'}(\cos \theta) = \frac{\ell+1}{4 \pi} P_\ell^{(1,0)}(\cos \theta) ,\]
we may express $\Gamma_{\ell,\ell_0}$ as
\begin{equation}
\label{e:jacobi1}
  \Gamma_{\ell,\ell_0}(\theta) = C^2_{\ell,\ell_0} \Big( \frac{\ell+1}{4\pi} P_\ell^{(1,0)}(\cos \theta) -  \frac{ \ell_0}{4\pi} P^{(1,0)}_{ \ell_0-1} (\cos \theta) \Big) .
  \end{equation}
 The derivatives of the Jacobi polynomials have the form
\begin{equation}
\label{e:jacobi2} \partial^{(k)}_s P_\ell^{(u,v)}(s) =  \frac{ \Gamma(\ell+u+v+1+k)}{2^k \Gamma(\ell+u+v+1)} P_{\ell-k}^{(u+k,v+k)}(s) .
\end{equation}
Combining \eqref{e:jacobi1} and \eqref{e:jacobi2} and using the chain rule we deduce the bounds
\begin{equation}
\label{e:jacobi3}
 |   \Gamma_{\ell,\ell_0}(\theta) |  \le  c_1 C^2_{\ell,\ell_0}  \sin(\theta)  \max_{\ell' \le \ell}  \ell'  \big|  P_{\ell'}^{(1,0)}(\cos \theta)   \big|
 \end{equation}
 and
 \begin{equation}
\label{e:jacobi4}
 |  \Gamma'_{\ell,\ell_0}(\theta) |  \le  c_1 C^2_{\ell,\ell_0}  \sin(\theta)  \max_{\ell' \le \ell}  (\ell')^2  \big|  P_{\ell'-1}^{(2,1)}(\cos \theta)   \big|
 \end{equation}
 for an absolute constant $c_1 > 0$. The `Hilb-type' asymptotics for Jacobi polynomials \cite[Theorem 8.21.12]{szego} state that
\[ P_\ell^{(u,v)}(\cos \theta)  = \tilde{\ell}^{-u} \frac{\Gamma(\ell+u+1)}{\Gamma(\ell+1)} (\sin (\theta/2) )^{-u} (\cos (\theta/2) )^{-v} \Big( \frac{\theta}{\sin \theta} \Big)^{1/2} J_u( \tilde{\ell} \theta) +O\big( \theta^{1/2} \ell^{-3/2} \big)\]
uniformly over $\theta \in [0, \pi-\epsilon]$, where $\tilde{\ell} = \ell + (u+v+1)/2$, and $J_u$ is the Bessel function. By $|J_u(x)| \le c_u /\sqrt{x}$ \cite[(1.71.11)]{szego} and the symmetry in $\theta \mapsto \pi-\theta$, one gets the bound
\[ | P_\ell^{(u,v)}(\cos \theta) |  \le c_{u,v}   \begin{cases}  \theta^{-u-1/2} \ell^{-1/2} &  \theta \le \pi/2 , \\
\tilde{\theta}^{-v-1/2} \ell^{-1/2}   &\tilde{\theta} = \pi - \theta \le \pi/2 . \end{cases} \]
Inputting this into \eqref{e:jacobi3} and \eqref{e:jacobi4}, and using that $0 \le \sin(\theta) \le \theta$ for $\theta \in [0,\pi/2]$, we have respectively
\[  \max\Big\{ |   \Gamma_{\ell,\ell_0}(\theta) | , \ell^{-1}   |   \Gamma_{\ell,\ell_0}(\tilde{\theta}) |  \Big\} \le  c_2 C^2_{\ell,\ell_0} \begin{cases}   \theta^{-3/2} \ell^{1/2}   &  \theta \le \pi/2 \\
 \tilde{\theta}^{-1/2} \ell^{1/2}  & \tilde{\theta} = \pi - \theta \le \pi/2  \end{cases},\]
 with an absolute constant $c_2 > 0$. Applying \eqref{e:cbounds} gives the first statement of the lemma, and applying \eqref{e:cbounds2} gives the second statement of the lemma.

For the final statement of Lemma \ref{l:ubble}, we split into cases $\alpha \le 1/2$ and $\alpha \ge 1/2$. In the former case, noticing that the constant in \eqref{e:cbounds} is uniform over $\alpha \le 1/2$, these bounds are weaker than what we have already obtained. In the latter case we observe that, for all $\ell \ge \ell_0$,
 \begin{equation*}
 | \Gamma_{\ell,\ell_0}(\theta) | \le C_{\ell,\ell_0}^2 \sum_{\ell' \ge \ell_0}^\ell N_{\ell'} |P_{\ell'}(\cos \theta) |    \le  \max_{\ell_0 \le \ell' \le \ell}  |P_{\ell'}(\cos \theta) |
 \end{equation*}
where we used that $C_{\ell,\ell_0}^{-2} = \sum_{\ell' \ge \ell_0}^\ell N_{\ell'}$. Then inserting the bound in Lemma \ref{l:ubrsh} gives the claim for  $| \Gamma(\theta)|$, and the claim for $| \Gamma'_{\ell_0,\ell}(\theta)|$ is proven similarly.
\end{proof}

\subsection{Finite-range approximation for Kostlan's ensemble}
\label{s:frak}

For $x \in \Sc^2$ and $r \in [0, \pi]$ let $\Dc_{r}(x)$ denote the spherical cap of radius $r$ centred at $x$, denote $\eta \in \Sc^2$ to be the north pole, and abbreviate $\Dc_r = \Dc_r(\eta)$. If $r < 0$ or $r > \pi$ then we set $\Dc_r(x)$ to be the empty set and the whole sphere $\Sc^2$ respectively.

\begin{definition}[Finite-range approximation]
$\,$
\begin{enumerate}[i.]
\item A Gaussian field $f$ on $\Sc^2$ is said to be \textit{$r$-range dependent} if $\textrm{Cov}[f(x), f(y)]  = 0$ for all $x,y \in \Sc^2$ such that $d_{\Sc^2}(x,y) \ge r$. If $f$ is isotropic, this is equivalent to the covariance kernel $K(x) = \E[f(\eta) \cdot f(x) ]$ being supported in ~$\Dc_{r}$.
\item A \textit{finite-range approximation} of a Gaussian field $f$ on $\Sc^2$ is a coupling between $f$ and a continuous $r$-range dependent Gaussian field $f_r$ such that the difference $f-f_r$ is `small' in an appropriate sense. Mind that $f_{r}$ does not have to be smooth or isotropic.
\end{enumerate}
\end{definition}

\smallskip
We exhibit a finite-range approximation of Kostlan's ensemble by observing that the basis functions in \eqref{eq:fn Kostlan} are spatially localised, and hence we may truncate them with minimal loss. In fact, this localisation property is only true up to axis symmetry, so we shall work inside the restriction of the sphere to the (strictly) positive orthant $\mathcal{O} = \{x \in \R^3 : x_i > 0\}$.

\begin{proposition}[Finite-range approximation of Kostlan's ensemble]
\label{p:frak}
Let $\Uc$ be a compact subset of $\Sc^2 \cap \mathcal{O}$. Then there exist constants $c_1,c_2 > 0$ such that, for every $n \ge 1$ and $r \ge 1/\sqrt{n}$, there exists a coupling of $f_n$ with a smooth $r$-range dependent Gaussian field $f_n^{(r)}$ on $\Sc^2$ satisfying, for every $x,y \in \Uc$,
  \begin{equation}
  \label{e:frakvar}
 \left| \textrm{Cov} \left( (f_n-f_n^{(r)})(x) , (f_n-f_n^{(r)})(y) \right) \right|  \le c_1 e^{ - c_2 n \cdot \max\{ r^2 , d_{\Sc^2}(x,y)^2 \}   }
    \end{equation}
    and
      \begin{equation}
  \label{e:frakesup}
 \E \big[ \| \nabla_{\Sc^2}(f-f_r)(x)\|_2^2 \big]  \le  c_1  n  e^{-c_2 n r^2}  ,
   \end{equation}
   where  $\nabla_{\Sc^2}$ denotes the spherical gradient.
\end{proposition}

As the following lemma demonstrates, \eqref{e:frakvar} and \eqref{e:frakesup} control the `size' of $f_n - f_n^{(r)}$:

\begin{lemma}
\label{l:sup}
 There exists an absolute constant $c > 0$ such that, for every $a,b > 0$, every subset $\Uc \subseteq \Sc^2$, and every smooth Gaussian field $f$ on $\Sc^2$ satisfying
\[ \sup\limits_{x \in \Uc} \var\left( f(x)\right) \le a^2 \qquad \text{and} \qquad \sup_{x \in \Uc} \E[  \|\nabla_{\Sc^2} f(x)  \|^2_2 ] \le b^2, \]
it holds that, for every $x \in \Uc$ and $a/b \le u \le \pi$,
\[  \E \left[ \sup_{y \in \Dc_u(x)  \cap \Uc} f(y) \right]  \le  c a \sqrt{ \log(2 u b/a )  }  \]
and moreover, for every $t \ge 2  c a \sqrt{ \log(2ub/a) } $,
\[\prob \left( \sup_{y \in \Dc_u(x) \cap \Uc } f(y)   \ge  t \right) \le e^{  -t^2 / (8a^2)  } . \]
\end{lemma}

 We defer the proof of Lemma \ref{l:sup} to the end of the section. To establish Proposition \ref{p:frak} we show that the basis functions in \eqref{eq:fn Kostlan} have Gaussian decay at the scale $1/\sqrt{n}$ around a unique point in the (closure of the) positive orthant. For $n \ge 1$ and $J = (j_1,j_2,j_3)$ such that $|J|=n$ let $b_{n,J}(x) = \sqrt{{n\choose J}}x^{J}$ denote the basis functions appearing in \eqref{eq:fn Kostlan}, and define the point
\begin{equation}
\label{eq:xnJ def}
v_{n,J} = \Big( \sqrt{j_1/n}, \sqrt{j_2/n}, \sqrt{j_2/n} \Big)  \in \Sc^2 \cap \overline{\mathcal{O}}.
\end{equation}

\begin{proposition}[Spatial localisation of the Kostlan basis]
\label{p:lock}
Let $\Uc$ be a compact subset of $\Sc^2 \cap \mathcal{O}$. Then there exist $c_1,c_2 > 0$ such that, for every $n \ge 1$, $|J| = n$, and $x \in \Uc$,
\[ 0 \le b_{n,J}(x) \le  c_1 n^{-1/2} e^{-c_2 n \cdot d_{\Sc^2}( x, v_{n,J} )^2}  \quad \text{and} \quad  \|\nabla_{\Sc^2} b_{n,J}(x) \|_2 \le  c_1 e^{-c_2 n\cdot d_{\Sc^2}( x, v_{n,J} )^2 }  ,   \]
where $v_{n,J}$ is as in \eqref{eq:xnJ def}.
\end{proposition}

Before proving this, let us conclude the proof of Proposition \ref{p:frak}:

\begin{proof}[Proof of Proposition \ref{p:frak} assuming Proposition \ref{p:lock}]
Let $\phi : [0,\infty) \to [0,1]$ be a smooth function with the properties that $\phi(x) = 0$ for $x \in [0,1/4]$, $\phi(x) = 1$ for $x \ge 1/2$, and $\|\phi'\|_\infty \le 5$ (such a function can be obtained as a smooth approximation of a piecewise linear function). For $r \in [0,\pi]$ define $\phi_r(x) = \phi(\theta/r)$. Recall from \eqref{eq:fn Kostlan} that $f_n(x) = \sum_{|J|=n} a_J b_{n,J}(x)$. For each $|J|=n$ and $r > 0$, define the truncated basis function
\[ b_{n,J}^{(r)}(x) = b_{n,J}(x) \big(1 - \phi_r \big( d_{\Sc^2}(x, v_{n,J})  \big) \big) \]
and set  $ f_n^{(r)}(x) =  \sum_{|J|=n} a_J b_{n,J}^{(r)}(x)$. By construction each $b_{n,J}^{(r)}$ is smooth and supported on a spherical cap of radius $r/2$, and so $f_n^{(r)}$ is a smooth $r$-dependent Gaussian field.

 It remains to verify \eqref{e:frakvar} and \eqref{e:frakesup}. We will use the following claim:

 \begin{claim}
 \label{c:num}
 There exists an absolute constant $c_1 > 0$ such that, for all $r \ge 1/n$ and $x  \in \Sc^2$,
\[ \big| \big\{ |J| =n : d_{\Sc^2}( x, v_{n,J} ) \le r  \big\} \big| \le  c_1 r^2 n^2 . \]
 \end{claim}
 \begin{proof}[Proof of Claim \ref{c:num}]
Since one may cover $\{y \in \Sc^2 : d_{\Sc^2}(x,y) \le r \}$ with $c_1 r^2 n^2$  spherical caps $\Dc_i$ of radius $1/(6n)$, it remains to show that each $\Dc_i$ contains at most one point  $v_{n,J}$. Define
$d_n = \min\limits_{|J| = |J'| = n, J \neq J'} \| v_{n,J}  - v_{n,J'} \|_2$. Since $J$ and $J'$ have coordinates which are positive integers at most $n$, and differ in at least one coordinate, we have
\[ d_n \ge \min_{0 \le j \le n-1} \left(\sqrt{(j+1)/n} - \sqrt{j/n} \right)   \ge 1/(3n)\]
where we used that $\sqrt{x+1} - \sqrt{x} \ge 1/(3\sqrt{x})$ for all $x \ge 1$. Since the spherical distance is larger than the Euclidean distance, we have the claim.
\end{proof}

To verify \eqref{e:frakvar}, by Proposition \ref{p:lock} and Claim \ref{c:num}, for all $x \in \Uc$,
\begin{align}
 \label{e:frakvar2} \textrm{Var} \big[f_n(x) - f_n^{(r)}(x) \big] &  =   \sum_{|J|=n} \big(b_{n,J}(x)  \phi_r \big( d_{\Sc^2}(x, v_{n,J}) \big)  \big)^2   \\
\nonumber & \le \sum_{k \ge 1} \sum_{|J|=n} b_{n,J}(x)^2 \id_{d_{\Sc^2}(x, v_{n,J})  \in [kr/4,(k+1)r/4]}(x) \\
 \nonumber &  \le  \sum_{k \ge 1} c_1 k^2 r^2 n^2  \big( c_2 n^{-1}  e^{-c_3 k^2 nr^2} \big) \le c_4  e^{-c_5  nr^2} ,
 \end{align}
 where $c_2,c_3,c_4,c_5> 0$ are constants depending only on $\Uc$.  Moreover, since $0 \le b_{n,J}^{(r)} \le b_{n,J}$ for all $x,y \in \Uc$,
\[ 0 \le  \textrm{Cov}\left( (f_n-f_n^{(r)})(x) , (f_n-f_n^{(r)})(y) \right)  \le  \textrm{Cov}\left( f_n(x) , f_n(y)  \right) = \kappa_n(x,y) \le e^{-c_6 n \cdot d_{\Sc^2}( x, y)^2} , \]
where we used Lemma \ref{l:ubke} in the final step. Combining these proves \eqref{e:frakvar}.

Similarly to \eqref{e:frakvar2}, by Proposition \ref{p:lock} and Claim \ref{c:num}, for all $x \in \Uc$,
\begin{align*}
&  \E \big[ \| \nabla_{\Sc^2} (f_n(y) - f_n^{(r)}(x) \|_2^2 \big]  \\
  & \qquad \le \sum_{k \ge 1} \sum_{|J|=n} \big\| \nabla_{\Sc^2}  \big( b_{n,J}(x)  \phi_r \big( d_{\Sc^2}(x, v_{n,J}) \big)  \big) \big\|_2^2 \id_{d_{\Sc^2}(x, v_{n,J})  \in [kr/4,(k+1)r/4]}(x) \\
 & \qquad \le \sum_{k \ge 1} \sum_{|J|=n}  \Big( \| \nabla_{\Sc^2} b_{n,J}(x)\|_2^2  +  5^2 \| b_{n,J}(x)\|_2^2  \Big) \id_{d_{\Sc^2}(x, v_{n,J})  \in [kr/4,(k+1)r/4]}(x) \\
 & \qquad \le  \sum_{k \ge 1} c_1 k^2 r^2 n^2  \big( c_2   e^{-c_3  k^2 n r^2} \big) \le c_7 n  e^{-c_8  nr^2} ,
 \end{align*}
 which gives \eqref{e:frakesup}.
\end{proof}

It remains to prove the spatial localisation property:

\begin{proof}[Proof of Proposition \ref{p:lock}]
Let $\Uc^+$ be a compact subset of $\Sc^2 \cap \mathcal{O}$ such that $\Uc$ is contained in the interior of $\Uc^+$, and recall the point $v_{n,J} \in  \Sc^2 \cap \overline{\mathcal{O}}$ defined in \eqref{eq:xnJ def}. In the proof $c_i > 0$ will be constants that depend only on $\Uc^+$.

For $v = (v_1,v_2,v_3) \in \mathcal{S}^2 \cap  \overline{\mathcal{O}}$, define the function
\[ f_v(x) =  \log \big( v_1^{v_1^2} v_2^{v_2^2} v_3^{v_3^2} \big) - \log \big( x_1^{v_1^2} x_2^{v_2^2} x_3^{v_3^2} \big) \in [0, \infty] \]
on $\Sc^2 \cap \overline{\mathcal{O}}$, using the convention that $0^0 = 1$. Restricted to a compact subset of $\Sc^2 \cap \mathcal{O}$, as $v$ varies, $(f_v)_{v \in  \mathcal{S}^2 \cap  \overline{\mathcal{O}}}$ is a smooth family of functions.

We first claim that, for every $|J| = n$ and $x \in \Sc^2 \cap \overline{\mathcal{O}}$,
\begin{equation}
\label{e:bbound}
 0 \le b_{n,J}(x) \le
 \begin{cases}
 c_1 n^{-1/2}  e^{-n f_{v_{n,J}}(x)} & \text{if } v_{n,J} \in \Uc^+, \\ e^{-n f_{v_{n,J}}(x)} & \text{else.}
\end{cases}
 \end{equation}
Clearly $b_{n,J}(x)$ is non-negative on $\Sc^2 \cap \overline{\mathcal{O}}$. As for the upper bound, using Stirling's formula $n! \sim n^n e^{-n} \sqrt{2 \pi n}$ we can bound
\begin{align}
\nonumber  {n \choose J}   = \frac{ n! }{ (j_1!)(j_2!)(j_3!) }  & \le \frac{ c_1 n^n  \sqrt{2 \pi n}   }{ j_1^{j_1}  j_2^{j_2}   j_3^{j_3}    (2 \pi)^{3/2}  \sqrt{ j_1 j_2 j_3 } } \\
  \label{e:njbound}  &  \le c_2 n^{-1} n^n j_1^{-j_1}  j_2^{-j_2}   j_3^{-j_3}
  \end{align}
  where the constants $c_1,c_2 > 0$ depend only on $\min_i j_i/n$ and hence are uniform over $v_{n,J} \in \Uc^+$. This gives
\begin{align*}
b_{n,J}(x) & = \sqrt{ {n \choose J} } x^J  \le  \sqrt{c_2} n^{-1/2} n^{-n/2} (j_1^{j_1} j_2^{j_2} j_3^{j_3}   )^{-1/2}  x^J  \\
& = \sqrt{c_2} n^{-1/2}  \big( n^{-1/2} n^{(j_1+j_2+j_3)/(2n)} (j_1/n)^{j_1/(2n)}  (j_2/n)^{j_2/(2n)}  (j_3/n)^{j_3/(2n)}  x^{x_{J,n}} \big)^n  \\
& =  \sqrt{c_2} n^{-1/2}  e^{-n f_{v_{n,J}}(x)} .
\end{align*}
In the case $v_{n,J} \notin \Uc^+$ we replace \eqref{e:njbound} with the weaker bound
\[  n^n = (j_1 + j_2 + j_3)^n \ge \max_{|K|=n} {n\choose K} j_1^{k_1} j_2^{k_2} j_3^{k_3} \ge  {n \choose J} j_1^{j_1} j_2^{j_2} j_3^{j_3}   , \]
and \eqref{e:bbound} follows as in the previous case.

We next prove that, for every $|J| = n$ and $x \in \Uc$,
\begin{equation}
\label{e:bderivbound}
\|\nabla_{\Sc^2} b_{n,J}(x)\|_2 \le  c_3  n d_{\Sc^2}(x,v_{n,J}) \times b_{n,J}(x) .
\end{equation}
To this end let $\lambda \in T_x$ be unit tangent vector at $x \in \Uc$, and write $v_{n,J} = (v_1,v_2,v_3)$. Using that
\[ j_i = n v_i^2  = n \big(x_i - (x_i -v_i) \big)^2 \]
we can bound
\begin{align*}
|\partial_\lambda b_{n,J}(x)| &= \bigg| \sqrt{ {n \choose J} }  \sum_{i=1,2,3} \lambda_i \partial_{x_i}  x^J \bigg| = \Big| \Big( \sum_{i =1,2,3} \lambda_i j_i / x_i  \Big)  b_{n,J}(x) \Big|  \\
& =n \times  \Big|   \sum_{i =1,2,3} \lambda_i x_i  -  \lambda_i 2(x_i-v_i) + \lambda_i (x_i - v_i)^2 / x_i \Big| \times   b_{n,J}(x)   \\
& \le c_4 n  \times  \|x-v_{n,J}\|_{L^1(\R^3)} \times  b_{n,J}(x)   ,
\end{align*}
where the inequality used that $ \sum_i \lambda_i x_i = 0$ for $\lambda \in T_x$, and that $x \in \Uc$. Since $d_{\Sc^2}$ is comparable to the $L^1(\R^3)$ norm up to an absolute constant, this implies \eqref{e:bderivbound}.

We next find lower bounds for the function $f_{v}(x)$. To this end we claim that, for each $v \in \Sc^2 \cap \overline{\mathcal{O}}$, $f_v(x)$ attains a unique minimum at the point $v$ with value $0$. Moreover, if $v \in \Sc^2 \cap \mathcal{O}$, this minimum is non-degenerate. To establish the claim, substitute $x_3 = \sqrt{1 - x_1^2 - x_2^2}$, switch to polar coordinates $(r,\rho)$ for $(x_1,x_2)$, and notice that
\begin{equation}
\label{e:fv}
 f_v(x) =\log  c_v - \log g_v(r) - \log h_v(\rho) ,
 \end{equation}
where $c_v =  v_1^{v_1^2} v_2^{v_2^2} v_3^{v_3^2}$, $ g_v(r) =  r^{v_1^2+v_2^2} (1 - r^2)^{v_3^2/2} $, $h_v(\rho) = (\cos \rho)^{v_1^2} (\sin \rho)^{v_2^2} $, and $g_v(r)$ and $h_v(\rho)$ are defined on $r \in [0,1]$ and $\rho \in [0,\pi/2]$ respectively. A direct computation shows that $g_v(r)$ has a unique maximiser at $r' = \sqrt{v_1^2+v_2^2}$, with
\[  g''_v(r)|_{r = r'} =   - g_v(r') \times 2/v_3^2 < 0 . \]
Similarly $h_v(\rho)$ has a unique maximum at $\rho' = \arccos(v_1/ \sqrt{v_1^2+v_2^2})$, with
\[ h''_v(\rho)|_{\rho = \rho'} = - h_v(\rho') \times  \Big( v_1^2 / \cos(\rho')^2 + v_2^2 / \sin(\rho')^2 \Big) < 0 .  \]
Moreover we can check that the point $(r',\rho')$ coincides with $v$, and also that
\[ g_v(r)|_{r = r'}  \times h_v(\rho)|_{\rho = \rho'}  = c_v . \]
Given the separation of variables of $f_v(x)$ in \eqref{e:fv}, this proves the claim.

We are now ready to prove the proposition. Recall the compact subsets $\Uc \subseteq \Uc^+ \subseteq \Sc^2 \cap \mathcal{O}$. There are two cases to consider:

\textbf{Case 1:}  $v_{n,J} \in \Uc^+$. By smoothness and compactness, the minima of $f_v$ are uniformly non-degenerate over $v \in \Uc^+$. Hence, for all $v_{n,J} \in \Uc^+$ and $x \in \Sc^2 \cap \overline{\mathcal{O}}$,
\[ f_{v_{n,J}}(x) \ge  c_5 d_{\Sc^2}( x, v_{n,J}  )^2 . \]
Putting this into \eqref{e:bbound} and \eqref{e:bderivbound}, for all $x \in \Uc$,
\[ |b_{n,J}(x)| \le  c_1 n^{-1/2} e^{-n f_{v_{n_J}}(x)}  \le  c_1 n^{-1/2} e^{- c_5 n \cdot d_{\Sc^2}( x, v_{n,J} )^2 } \]
and
\[\|\nabla_{\Sc^2} b_{n,J}(x)\|_2  \le  c_1 c_3 n^{1/2}  d_{\Sc^2}(x,v_{n,J})  e^{- c_5 n \cdot d_{\Sc^2}( x, v_{n,J} )^2 } \le  c_6 e^{- c_7 n \cdot d_{\Sc^2}( x, v_{n,J} )^2 } .\]

\textbf{Case 2:}  $v_{n,J} \notin \Uc^+$. By smoothness and compactness, $f_v(x)$ are uniformly bounded below over $v \notin \Uc^+$ and $x \in \Uc$. Hence, for all $v_{n,J} \notin \Uc^+$ and $x \in \Uc$,
\[  f_{v_{n,J}}(x)  \ge  c_8 . \]
Putting this into \eqref{e:bbound} gives that, for all $x \in \Uc$,
\[ |b_{n,J}(x)| \le  e^{-n f_{v_{n,J}}(x)} \le e^{- c_8  n} \le   n^{-1/2 } e^{- c_9 n \cdot d_{\Sc^2}( x, v_{n,J} )^2 } \]
where the final inequality is since $\inf_{x \in \Uc, v \notin \Uc^+}  d_{\Sc^2}(x , v)> 0$, and similarly by  \eqref{e:bderivbound}
\begin{equation*}
\|\nabla_{\Sc^2} b_{n,J}(x)\|_2  \le   c_3 n  d_{\Sc^2}(x,v_{n,J})   e^{-n f_{v_{n,J}}(x)}  \le  c_3 n  d_{\Sc^2}(x,v_{n,J})   e^{- c_8 n }  \le  e^{- c_9 n \cdot d_{\Sc^2}( x, v_{n,J} )^2 }.  \qedhere
\end{equation*}
\end{proof}

We conclude this section with the proof of Lemma \ref{l:sup}, which uses standard techniques:

\begin{proof}[Proof of Lemma \ref{l:sup}]
The second claim follows from the first claim and the Borell--TIS inequality \cite[Theorem 2.1.1]{at07}, so we focus on the first claim. By scaling we may assume that $a = 1$. Let $c_1 > 0$ be a constant such that, for every $u \in [0,\pi]$ and $b > 0$, $\Dc_u$ can be covered with at most $n = 1 + c_1 u^2 b^2$ spherical caps $(\Dc_i)_{i \le n}$ of radius $b^{-1}$. Let $\widetilde{\Dc}$ denote one such spherical cap, and let $P$ be its centre. Recalling that $T_P$ denotes the exponential map, define the rescaled field $G  = f \circ T_P( b \cdot) $ on  $T_P(\Sc^2)$. By the assumptions on $f$, the covariance kernel $K(u,v)$ of $G$ satisfies
\[ \sup_{u,v \in T_p(\Sc^2)} \max_{k_1 + k_2 \le 2} | \partial_u^{k_1} \partial_v^{k_2} K(u,v) | \le c_3 \]
for an absolute constant $c_3 > 0$. Moreover there exists an absolute constant $c_4 > 0$ such that $T_P( b \widetilde{\Dc}) \subseteq B_{c_4}$. Hence by Kolmogorov's theorem \cite[Appendix A.9]{NaSoGen}
\[ \E \left[ \sup\limits_{x \in \widetilde{\Dc}} f(x) \right]  \le  \E \left[ \sup\limits_{u \in B_{c_4}} G(u) \right] \le c_5 \]
for an absolute constant $c_5 > 0$. Applying the Borell--TIS inequality, for every $t \ge 0$,
\[ \prob\left( \sup\limits_{x \in \widetilde{\Dc}} f(x)  \ge  c_5 + t \right)  \le e^{-t^2/2} .\]
Hence $X_i := \sup\limits_{x \in \Dc_i} f(x) - c_5$ are sub-Gaussian random variables, and so by a standard bound for such variables (see \cite[Claim 2.16]{MS23})
\[ \E \left[  \sup\limits_{i \le n}  X_i \right] \le  c_6 \sqrt{ \log (1+n) } \]
for an absolute $c_6 > 0$. Since $$\E \left[ \sup\limits_{x \in \Dc_u} f(x) \right] \le c_5 +  \E \left[  \sup\limits_{i \le n}  X_i \right]  ,$$ we have the result.
\end{proof}

\subsection{Finite-range approximation for general isotropic Gaussian fields}
\label{sec:fin rang app gen}
Next we present a finite-range approximation scheme, suitable for general isotropic Gaussian fields on $\Sc^2$; this is the spherical analogue of the well-known `moving-average' approximation scheme for Euclidean fields, see for instance \cite{MV20}. Recall that $\{Y_{\ell',m}\}_{m=-\ell'}^{\ell'}$ is the standard basis of spherical harmonics, and in particular every $g \in L^2(\Sc^2)$ can be expanded (in $L^2(\Sc^2)$)
as
\begin{equation}
\label{e:gexp}
 g(x) = \sum_{\ell'} \sum_{m = -\ell'}^{\ell'} c_{\ell',m} Y_{\ell',m}(x) ,
 \end{equation}
where $\sum c_{\ell',m}^2 = \|g\|_2^2 < \infty$. A function $g \in L^2(\Sc^2)$ is \textit{zonal} if $g(x)$ depends only on $\theta = d_{\Sc^2}(\eta,x)$, the spherical angle between $\eta$ and $x$, and we often consider a zonal $g$ as a function of $\theta \in [0,\pi]$. A function $g$ is zonal if and only if $c_{\ell',m} = 0$ for every $m \neq 0$. The zonal spherical harmonics $Y_{\ell',0}$ can be expressed as
\[ Y_{\ell',0}(\theta) = \sqrt{N_{\ell'}} P_{\ell'}(\cos \theta),\]
recalling that $N_{\ell'} = (2\ell'+1)/(4\pi)$ and $P_\ell$ are the Legendre polynomials.

Similarly, every isotropic Gaussian field $f$ on $\Sc^2$ can be expanded as
\begin{equation}
\label{e:igf}
f(x) = \sum_{\ell'}  \frac{ c_{\ell'} }{  \sqrt{N_{\ell'}} } \sum_{m=-\ell'}^{\ell'}  a_{\ell',m} Y_{\ell',m}(x) ,
\end{equation}
where $\sum c_{\ell'}^2 < \infty$ and $a_{\ell',m}$ are i.i.d.\ Gaussian random variables. Conversely, every sequence $\{c_{\ell'}\}$ such that $\sum c_{\ell'}^2 < \infty$ defines an isotropic Gaussian field via \eqref{e:igf}. The covariance kernel of $f$ in \eqref{e:igf} is the zonal function
\begin{equation}
\label{e:k}
\begin{split}
 K(x) &= \E[f(\eta)\cdot f(x)] = \sum_{\ell'}  \frac{c^2_{\ell'}}{N_\ell} \sum_m Y_{\ell',m}(\eta) Y_{\ell',m}(x) \\&=  \sum_{\ell'}  \frac{c^2_{\ell'}}{\sqrt{N_{\ell'}}}  Y_{\ell',0}(x) =  \sum_{\ell'} c^2_{\ell'} P_{\ell'}(\cos \theta) ,
 \end{split}
 \end{equation}
where we used the {\em addition formula}
\[ \sum_{m} Y_{\ell',m}(\eta) Y_{\ell',m}(x) = \sqrt{N_{\ell'}} Y_{\ell',0}(x) .\]
If moreover $\sum c^2_{\ell'} (\ell')^2 < \infty$, then the Gaussian field $f$ in \eqref{e:igf} is almost surely continuous and its gradient $\nabla_{\Sc^2} f$ exists in a mean-squared sense.

\smallskip
Our first result provides a general method of constructing a finite-range approximation for isotropic spherical Gaussian fields:

\begin{proposition}[General finite-range approximation]
\label{p:fra}
Let $\{c_{\ell'}\}$ be such that $\sum c_{\ell'}^2 (\ell')^2 < \infty$, $f$ be the isotropic Gaussian field \eqref{e:igf} with covariance $K$ in \eqref{e:k}, and define the zonal function
\begin{equation}
\label{e:q}
q(x) = \sum_{\ell'} c_{\ell'}  Y_{\ell',0}(x) = \sum_{\ell'} c_{\ell'} \sqrt{N_{\ell'}} P_{\ell'}(\cos \theta)  .
\end{equation}
Then, for every $r \in [0, \pi/2]$, there exists a coupling of $f$ with a continuous $r$-range dependent isotropic Gaussian field $f^{(r)}$ on $\Sc^2$ satisfying
\[  \left| \textrm{Cov}\left(  (f - f^{(r)})(\eta)  ,  (f - f^{(r)})(\theta)  \right) \right| \le c  \begin{cases} q_1 &  \theta \le r , \\ q_2(\theta) & \theta \ge r , \end{cases}    \]
 and
\[   \E \big[ \|\nabla_{\Sc^2} (f-f^{(r)})(\eta) \|^2_2 \big]   \le c q_3  , \]
where
\[ q_1 =  \int_{\Sc^2 \setminus \mathcal{D}_{r/4}} q(x)^2 \, dx  \ , \quad  q_2(\theta) = |K(\theta)| +  \int_{\mathcal{D}_{r/2}} |q(x)| \, dx  \sup_{\theta' \in [\theta-r/2,\theta+r/2]} |q(\theta')| \]
and
\[ q_3 =   \int_{\Sc^2 \setminus \mathcal{D}_{r/4}} \| \nabla_{\Sc^2} q(x)\|_2^2 \, dx     +  \sup_{x \in D_{r/2} \setminus D_{r/4} } q(x)^2  , \]
and $c > 0$ is an absolute constant.
\end{proposition}

To give some intuition for Proposition \ref{p:fra}, recall the  \textit{spherical convolution}
\[ (g \star h)(x) = \int_{\omega \in SO(3)} g(\omega \eta) h(\omega^{-1} x) \, d\omega  , \]
where $g,h \in L^2(\Sc^2)$, and the integral is over the set $\omega \in SO(3)$ of rotations of the sphere equipped with the Haar measure $d\omega$. The idea is to represent the field $f$ as a \textit{moving average} $f  = q \star W$, where $W$ is the white noise on $\Sc^2$. We may then couple $f$ with $f^{(r)} = q^{(r)} \star W$, where $q^{(r)}$ is a multiplication of $q$ by a smooth approximation of the indicator of the spherical cap $\Dc_{r/2}$. By construction, $f^{(r)}$ is $r$-range dependent, and the difference $f - f^{(r)} = (q-q^{(r)}) \star W$ can be estimated in terms of $q - q^{(r)}$.

To avoid technicalities, we will implement this construction without reference to the spherical white noise. Instead, we rely on the following two properties of the spherical convolution:
\begin{enumerate}
\item If $g$ and $h$ are the zonal functions
\[   g(x) = \sum_{\ell'} c_{\ell'} Y_{\ell',0}(x)  \quad \text{and} \quad     h(x) = \sum_{\ell'} \tilde{c}_{\ell'} Y_{\ell',0}(x)  \]
then $(g \star h)(x) = (h \star g)(x)$ is the zonal function
\[ (g \star h)(x) = (h \star g)(x) = 2 \pi \sum_{\ell'}  \frac{ c_{\ell'} \tilde{c}_{\ell'} }{\sqrt{N_{\ell'}}}   Y_{\ell',0}(x) . \]
\item If $g$ and $h$ are zonal functions, and $g$ is supported on $\Dc_{r}$, then
\[| (g \star h)(\theta)| \le  2\pi \int_{x \in \Dc_r} |g(x)| \, dx\cdot \sup_{ \theta' \in [\theta-r, \theta+r] } |h(\theta')| .  \]
 In particular, if $g$ is supported on $\Dc_{r}$, then $g \star g$ is supported on $\Dc_{2r}$.
 \end{enumerate}
The first property is
\cite[Theorem 1]{dh94}. The second follows directly from the definition of spherical convolution, once we observe that, denoting $\C_\theta = \{x \in \Sc^2 : d_{\Sc^2}(\eta,x) = \theta\}$, if $\omega \in SO(3)$ is such that $\omega \eta \in \C_r$, and $x \in \C_\theta$, then $\omega x \in  \cup_{\theta' \in [\theta-r,\theta+r]} \C_{\theta'}$.

\smallskip
We will make use of the following lemma, which is proved at the end of section \ref{sec:fin rang app gen}:

\begin{lemma}
\label{l:norms}
Let $\sum c_{\ell'}^2 < \infty$, $f$ be the isotropic Gaussian field \eqref{e:igf}, and let $q$ be the zonal function in \eqref{e:q}. Then
\[  \textrm{Var}\left( f(\eta) \right)  =  \sum_\ell c_\ell^2  = \int q(x)^2 \, dx .\]
If, in addition, $\sum c_{\ell'}^2 (\ell')^2 < \infty$, then
\[  \E \big[ \|\nabla_{\Sc^2} f(\eta) \|^2_2 \big]  =  \sum_{\ell'} c_{\ell'}^2 \ell' (\ell' + 1)  =  \int \| \nabla_{\Sc^2} q(x) \|_2^2 \, dx  .\]
\end{lemma}

\begin{proof}[Proof of Proposition \ref{p:fra}]
Let $\phi : [0,\infty) \to [0,1]$ be a smooth function with the properties that $\phi(x) = 0$ for $x \in [0,1/4]$, $\phi(x) = 1$ for $x \ge 1/2$, and $\|\phi'\|_\infty \le 5$. Then for $r \in [0,\pi]$ define the zonal function $\phi_r(x) = \phi(\theta/r)$ on $\Sc^2$; this satisfies $\phi_r = 0$ on $\Dc_{r/4}$ and $\phi_r = 1$ on $\Dc_{r/2}^c$.
Define $q^{(r)}(x) = q(x)\cdot (1-\phi_r(x))$, which is supported on $\Dc_{r/2}$, and consider the expansion
\[ q^{(r)}(x) = \sum_{\ell'} \tilde{c}_{\ell'}  Y_{\ell',0}(x)  \]
which is well-defined since $\|q^{(r)}\|_2 \le \|q\|_2  < \infty$. Recalling that $f$ is as in \eqref{e:igf}, define using the same coefficients $a_{\ell',m}$ as in \eqref{e:igf}, the isotropic field
\[ f^{(r)}(x) =  \sum_{\ell'}   \frac{\tilde{c}_{\ell'}}{\sqrt{ N_{\ell'}}  } \sum_{m=-\ell'}^{\ell'}  a_{\ell',m} Y_{\ell',m}(x)  \]
so that it holds that
\[ (f - f^{(r)})(x) =  \sum_{\ell'} \frac{ (c_{\ell'} - \tilde{c}_{\ell'} )}{\sqrt{N_{\ell'}} } \sum_{m=-\ell'}^{\ell'}  a_{\ell',m} Y_{\ell',m}(x)  .\]

We first claim that $f^{(r)}$ is $r$-range dependent. To this end we evaluate the covariance kernel of $f^{(r)}$ to be
\[  \tilde{K}(x) =   \sum_{\ell'} \frac{ \tilde{c}_{\ell'}^2}{\sqrt{N_{\ell'}} } Y_{\ell',0}(x)  .\]
By the properties of the self-convolution listed above, this is equal to $(2\pi)^{-1} q^{(r)} \star q^{(r)}$, which is supported on $\Dc_{r}$, and hence $f^{(r)}$ is $r$-range dependent as claimed. Moreover, by Lemma \ref{l:norms},
\[   \textrm{Var} \left( (f - f^{(r)})(\eta) \right)  =  \| q - q^{(r)} \|_2^2   = \| q \phi_r \|_2^2 \le   \int_{\Sc^2 \setminus \mathcal{D}_{r/4}} q(x)^2 dx = q_1  . \]
Similarly, by Lemma \ref{l:norms} and recalling that $q$ is zonal,
\begin{align*}
& \E \big[  \|\nabla_{\Sc^2} (f - f^{(r)})(\eta)  \|^2_2 \big]  =   \| \nabla_{\Sc^2} (q \phi_r)  \|_2^2   =    \int_{\Sc^2}  (\partial_\theta q(\theta) \phi_r (\theta))^2 \, dx    \\
& \qquad =  \int_{\Sc^2}  \big( (  \partial_\theta q(\theta) ) \phi_r(\theta) + q(\theta) (r^{-1} \phi'(\theta/r))  \big)^2  \, dx \\
& \qquad \le  2 \int_{\Sc^2 \setminus \mathcal{D}_{r/4}} (\partial_\theta q(\theta))^2  \, dx +  (50/r^2)  \textrm{Area}( \Dc_{r/2} \setminus \Dc_{r/4})  \sup_{ x \in \mathcal{D}_{r/2} \setminus \mathcal{D}_{r/4} } q(x)^2  \\
& \qquad \le  2 \int_{\Sc^2 \setminus \mathcal{D}_{r/4}} \|\nabla_{\Sc^2} q\|_2^2 \,  dx +  75 \pi  \sup_{ x \in \mathcal{D}_{r/2} \setminus \mathcal{D}_{r/4} } q(x)^2   = c q_3,
\end{align*}
where in the second step we used that $(a+b)^2 \le 2a^2 + 2b^2$, $\|\phi\|_\infty = 1$, $\|\phi'\|_\infty \le 5$, and $\phi'$ is supported on $ \mathcal{D}_{r/2} \setminus \mathcal{D}_{r/4} $, and in the final step we used  that
\[ \textrm{Area}( \Dc_{r'} \setminus \Dc_{r} ) = 2\pi( \cos(r) - \cos(r') )   \le \pi (r'-r)(r'+r) ,\]
 valid for all $0 \le r \le r' \le \pi/2$.

 To conclude, suppose that $\theta \ge r$. Then since $f^{(r)}$ is $r$-range dependent,
 \begin{align*}
  \textrm{Cov}[  (f - f^{(r)})(\eta)  ,  (f - f^{(r)})(\theta)  ]  &= K(\theta) + 2 \textrm{Cov}[ f(\eta),  (f - f^{(r)})(\theta)  ]  \\
  &  = K(\theta) +2   \sum_{\ell'} \frac{ (c_{\ell'} - \tilde{c}_{\ell'}) \tilde{c}_{\ell'}}{\sqrt{N_{\ell'}} } Y_{\ell',0}(\theta)  \\
  & = K(\theta) + 2 (2 \pi)^{-1} q \star q^{(r)}(\theta) .
  \end{align*}
  The absolute value of this expression is at most
  \[ |K(\theta)| +   2 \int_{x \in \Dc_{r/2}} |q^{(r)}(x)| \, dx \sup_{\theta' \in [\theta-r/2,\theta+r/2]} | q(\theta')| \le 2 q_2(\theta) , \]
  which completes the proof of Proposition \ref{p:fra}.
\end{proof}

\begin{proof}[Proof of Lemma \ref{l:norms}]
For the first claim, by \eqref{e:k} we have
\[ \textrm{Var}[f(\eta)] =  K(\eta) =  \sum_{\ell'}   c^2_{\ell'}  P_{\ell',0}(1)  = \sum_{\ell'} c_{\ell'}^2  = \|g\|_2^2 \]
where we used that $P_{\ell',0}(1) = 1$. For the second claim, recall that $Y_{\ell',0}$ is an eigenfunction of the spherical Laplacian with eigenvalue $-\ell'(\ell'+1)$. Hence we have
\begin{align*}
  \E[ \|\nabla_{\Sc^2} f(\eta) \|^2_2 ]  &=  - \Delta K(\eta) =  \sum_{\ell'}  \frac{ c^2_{\ell'} }{\sqrt{N_{\ell'}}}  \big( - \Delta Y_{\ell',0}(\eta) \big)  \\
  &=   \sum_{\ell'}  \frac{ c^2_{\ell'} \ell' (\ell'+1) }{\sqrt{N_{\ell'}}}  Y_{\ell',0}(\eta)  = \sum_{\ell'} c^2_{\ell'} \ell' (\ell'+1),
  \end{align*}
where the final step used again that $Y_{\ell',0}(\eta) = \sqrt{N_{\ell'}}P_{\ell'}(1) = \sqrt{N_{\ell'}}$. Moreover, using that $q$ is zonal, the expression $Y_{\ell',0}(\theta) =  \sqrt{N_{\ell'}} P_{\ell'}(\cos \theta)$, and changing coordinates into $u = \cos(\theta)$, we have
 \begin{align*}
   \| \nabla_{\Sc^2} q(x) \|_2^2 &= \int_{\Sc^2}   \Big( \sum_{\ell'} c_{\ell'}  \sqrt{N_{\ell'}}   \partial_\theta P_{\ell'}(\cos \theta) \Big)^2  \, dx   \\
   & =  2\pi    \sum_{\ell',\ell''} c_{\ell'} c_{\ell''}  \sqrt{N_{\ell'}}   \sqrt{N_{\ell''} }  \int_{u \in [-1,1]}  (1 - u^2) P'_{\ell'}(u)  P'_{\ell''}(u)   \, du  \\
 &  = \sum_{\ell'} c^2_{\ell'} \ell' (\ell'+1) ,
 \end{align*}
 where the final step used the weighted orthogonality relation for the Legendre polynomials
 \begin{equation*}
  \int_{u \in [-1,1]}  (1 - u^2) P'_{\ell'}(u)  P'_{\ell''}(u)   \, du  =  \frac{2 \ell'(\ell'+1) }{2\ell'+1} \delta_{\ell',\ell''} ,
  \end{equation*}
  which is a by-product of the standard orthogonality relations for $P_\ell$, via integration by parts.
\end{proof}

We now show how Proposition \ref{p:fra} can be applied to band-limited ensembles:

\begin{proposition}[Finite-range approximation for band-limited ensembles]
\label{p:frable}
Let $g_\ell$ be the band-limited ensemble \eqref{eq:gell band-limited sum}. Then for every $\ell$ there exists a choice of $c_{\ell'}$ as in Proposition \ref{p:fra} such that $f = g_\ell$. Moreover the quantities  $q_1$, $q_2(\theta)$, and $q_3$ in Proposition \ref{p:fra} satisfy the following bounds:

\begin{enumerate}[i.]
\item \textbf{Case $\alpha \in [0,1)$}. There exists an absolute constant $c > 0$ such that if $r \ge 1/\ell$ and $\theta \in [r,3\pi/2-r/2]$ then
\[ q_1 \le c \ell^{-1}  r^{-1}  , \quad q_2(\theta) \le  c \ell^{-1} r^{1/2} \theta^{-3/2}  \qquad \text{and} \qquad q_3 \le c  \ell r^{-1}    . \]

\item \textbf{Case $\alpha=1$, $\beta \in (0,1)$.}
There exists an absolute constant $c > 0$ such that if $r \ge \ell^{-\beta}$ and $\theta \in [r,3\pi/2-r/2]$ then
\[ q_1 \le c \ell^{-\beta}  r^{-1}  , \quad q_2(\theta) \le  c \ell^{-\beta} r^{1/2} \theta^{-3/2}   \qquad \text{and} \qquad q_3 \le c  \ell^{2-\beta} r^{-1}  . \]

\end{enumerate}
\end{proposition}

\begin{proof}
First suppose that $\alpha \in [0,1)$. We may apply Proposition \ref{p:fra} with the choice $c_{\ell'}  = C_\ell  \sqrt{N_{\ell'}}$ if $\ell' \in [\lfloor \alpha \ell \rfloor, \ell]$, and $c'_\ell = 0$ else, so that the field $f$ in \eqref{e:igf} of Proposition \ref{p:fra} is distributed as $g_\ell$. In this case
\[ q(x) = C_\ell \sum_{\ell' = \lfloor \alpha \ell \rfloor}^\ell \sqrt{N_{\ell'} } Y_{\ell'}(x)   = C_\ell  \sum_{\ell' = \lfloor \alpha \ell \rfloor}^\ell   N_{\ell'} P_{\ell'}(\cos \theta)  =  C_\ell^{-1} \Gamma_\ell( \theta)  . \]
Let us now control $q_1$, $q_2(\theta)$ and $q_3$. Using Lemma \ref{l:ubble} and \eqref{e:cbounds} we have
\begin{align*}
  q_1  & =  C_\ell^{-2}   \Big( \int_{\mathcal{D}_{\pi/2} \setminus \mathcal{D}_{r/4}}   \Gamma_\ell(x)^2  \,dx  + \int_{\Sc^2 \setminus \mathcal{D}_{\pi/2}}   \Gamma_\ell(\theta)^2 \, dx  \Big)  \\
&  \le  c_1 C_\ell^{-2} \ell^{-3} \Big( \int_{\theta \in [r/4,\pi/2]} \theta \times  \theta^{-3} \, d\theta +  \int_{\theta' \in [0,\pi/2]} \theta' \times  (\theta')^{-1} \, d\theta'   \Big) \\& \le c_2 \ell^{-1} r^{-1} ,
\end{align*}
and, since $r \ge 1/\ell$ and $r \le \theta \le 3\pi/2-r/2$,
\begin{align*}
  q_2(\theta)  & = \Big( \Gamma_\ell(\theta) +  C_\ell^{-2}  \int_{\Dc_{r/2}} \Gamma_\ell(x) \, dx \sup_{\theta' \in [\theta-r/2, \theta+r/2]} \Gamma_\ell(\theta')  \Big) \\
&  \le  c_3   \Big( (\theta \ell)^{-3/2}  +  C_\ell^{-2} \ell^{-3} \Big( \int_{\theta \in [0,r/2]} \theta \times  \theta^{-3/2} \, d\theta \Big)  (\theta-r/2)^{-3/2} \Big)
\\& \le c_4  \ell^{-1}  r^{1/2} \theta^{-3/2} .
\end{align*}
Finally, since $r \ge 1/\ell$,
\begin{align*}
  q_3  & =  C_\ell^{-2}   \Big( \int_{\mathcal{D}_{\pi/2} \setminus \mathcal{D}_{r/4}}   \Gamma'_\ell(x)^2  \,dx  + \int_{\Sc^2 \setminus \mathcal{D}_{\pi/2}}   \Gamma'_\ell(\theta)^2 \, dx  + \sup_{\theta \in [r/2,r/4] }  \Gamma_\ell(\theta) \Big)  \\
&  \le  c_5  C_\ell^{-2}  \Big( \ell^{-1} \Big[ \int_{\theta \in [r/4,\pi/2]} \theta \times  \theta^{-3} \, d\theta +  \int_{\theta' \in [0,\pi/2]} \theta' \times  (\theta')^{-1} \, d\theta'  \Big]  + \ell^{-3} r^{-3}   \Big) \\
& \le c_6  \ell r^{-1}     .
\end{align*}
Similarly, if $\alpha = 1$ and $\beta \in (0,1)$ we take $c_{\ell'}  = C_\ell  \sqrt{N_{\ell'}}$ if $\ell' \in [ n - \lfloor n^\beta \rfloor, \ell]$, and $c'_\ell = 0$ otherwise. The computations bounding $q_1$, $q_2(\theta)$, and $q_3$ in this case are similar to the above, and are thereby conveniently omitted.
\end{proof}

Let us now comment on the difficulties that arise when applying Proposition \ref{p:fra} in the case of Kostlan's ensemble and the random spherical harmonics:

\begin{itemize}
\item (Kostlan's ensemble) The covariance kernel of Kostlan's ensemble $f_n$ has the harmonic decomposition
\[ K(x) =  (\cos \theta)^n = \sum_{\ell'} \tilde{c}_{\ell'} P_{\ell'}(\cos \theta)  , \]
where (see \cite[Section 3]{dl22} for instance)
\[ \tilde{c}_{\ell'} = \frac{(2\ell'+1) n! }{ 2^{(n-\ell')/2} ( (n-\ell')/2) ! (\ell'+n+1)!!} \id_{  n - \ell' \in 2 \mathbb{N} } \]
 and so taking $c_{\ell'} = \sqrt{\tilde{c}_{\ell'}}$ the field $f$ in \eqref{e:igf} is distributed as $f_n$. However to apply Proposition \ref{p:fra} one needs control on the kernel $q$ defined in \eqref{e:q} in terms of $c_{\ell'} = \sqrt{\tilde{c}_{\ell'}}$, which seems difficult to extract. To bypass this, we use the approach from section \ref{s:frak}.

\item (Random spherical harmonics) For the random spherical harmonics we can simply take $c_{\ell'}  = 1$ if $\ell' = \ell$, and $c'_\ell = 0$ else. Although the kernel $q(x) = \sqrt{N_\ell} P_\ell(\cos \theta)$ in \eqref{e:q} is explicit in this case, its decay is too slow to get any useful conclusion from Proposition~\ref{p:fra}. In particular, since $P_\ell(\cos \theta)$ is typically on the order $ \theta^{-1/2} \ell^{-1/2}$, it produces a bound on $q_1$ which is of unit order for any choice of $r \ge 1/\ell$.
\end{itemize}

\subsection{Sprinkled decoupling for Kostlan's ensemble and band-limited ensembles}
\label{s:sdr}

We now apply the results from the previous sections to establish the sprinkled decoupling inequalities for Kostlan's ensemble and band-limited ensembles.

Let $f$ be a continuous Gaussian field on $\Sc^2$, $\Uc \subseteq \Sc^2$ be a subset, and let $r,u \in [0,\pi]$ be constants. Let $(A_i)_{i \le k}$ be a collection of subsets of continuous functions on $\Sc^2$. The collection of events $(\{f \in A_i\})_{i \le k}$ is in class $E(\Uc,r,u)$ if:
\begin{itemize}
\item For each $i$ there exists a compact subset $U_i \subseteq \Uc$ such that $\{f \in A_i\}$ is measurable with respect to $f$ restricted to $U_i$, and each $U_i$ is contained in some spherical cap $\Dc_u(x)$. Moreover, for distinct $i,j$, the subsets $U_i$ and $U_j$ are separated by distance at least $r$.
\item Each $A_i$ is increasing in the sense that $\{f \in A_i\} \subseteq \{f + t \in A_i \}$ for all $t \ge 0$.
\end{itemize}

We begin with a decoupling estimate for Kostlan's ensemble which we deduce from Proposition \ref{p:frak}. Recall that $f_{n}$ is Kostlan's ensemble, and that $\mathcal{O}$ is the strictly positive orthant.

\begin{proposition}[Sprinkled decoupling for Kostlan's ensemble]
\label{p:sdk}
For every compact subset $\Uc \subseteq \Sc^2 \cap \mathcal{O}$ there exist $c_1,c_2 > 0$ satisfying the following: Let $n \ge 1$, $r  \ge 1/\sqrt{n}$, $u \ge 1/\sqrt{n}$, a collection $(\{f_n \in A_i\})_{i \le k}$ in class $E(\Uc,r,u)$, and $t \ge  c_1 \sqrt{\log (2 u \sqrt{n} ) } e^{-c_2 r^2 n} $. Then we have the inequalities
\begin{equation}
\label{e:sdk1}
 \prob \Big[ \bigcap_{i \le k} \{f_n \in A_i \} \Big] \le  \prod_{i \le k} \prob \big[ \{ f_n + t  \in A_i \} ] +  c_1 k e^{ - c_2 t^2 e^{c_2 r^2 n}   }
 \end{equation}
  and
   \begin{equation}
\label{e:sdk2}
 \prob \Big[ \bigcap_{i \le k} \{f_n \in A_i \} \Big] \le  \Big( c_1  \sup_{i \le k}  \prob \big[ \{ f_n + t \in A_i \} ] \Big)^{c_2 k} +   c_1 e^{ - c_2 k t^2 e^{c_2 r^2 n}   }   .
 \end{equation}
\end{proposition}

\begin{remark}
$\,$
\begin{enumerate}[i.]
\item By taking complements, if each $A_i$ is decreasing rather than increasing, then the same bounds hold after replacing $f_n + t$
with $f_n - t$.
\item Notice that, compared to \eqref{e:sdk1}, \eqref{e:sdk2} has an extra factor of $k$ in the exponent of the second term on the right-hand side, at the expense of a constant factor in the exponent in the first term. This is a major advantage when \eqref{e:sdk2} is applied with $k$ large.
\end{enumerate}
\end{remark}

\begin{proof}
We begin by proving \eqref{e:sdk1}. Let $f_n^{(r)}$ be the $r$-range dependent field of~Proposition \ref{p:frak}. We first observe that, by Proposition \ref{p:frak} and Lemma \ref{l:sup},
\begin{equation}
\label{e:kostsup}
   \sup_{x \in \Uc} \E \Big[ \sup_{y \in \Dc_u(x) \cap \Uc} (f_n - f_n^{(r)})(y)  \Big]  \le c_1 \sqrt{ \log (u \sqrt{n} ) } e^{-c_2 r^2 n}
   \end{equation}
and, for every $t \ge 2 c_1 \sqrt{ \log (2 u \sqrt{n} )  } e^{-c_2 r^2 n} $,
  \begin{equation}
  \label{e:ebound}
 e_t = e_{t,u} := \sup_{x \in \Uc}  \prob \Big[\sup_{y \in \Dc_u (x) \cap \Uc} |f_n - f_n^{(r)}|(y) \ge t \Big]
 \le  e^{ - (1/8c_1) t^2 e^{c_2 r^2 n}  } .
  \end{equation}
  We next observe that, since the $A_i$ are increasing, for every $t \ge 0$,
\begin{equation}
\label{e:decomp}
 \bigcap\limits_{i\le k} \{f_n \in A_i\} \  \subseteq \   \bigcap\limits_{i\le k}  \{f_n^{(r)} + t \in A_i  \} \cup
 \bigcup\limits_{i\le k} \Big\{\sup_{x \in U_i} (f_n - f_n^{(r)})(x) \ge t \Big\} .
 \end{equation}
 Hence,
\begin{align*}
& \prob \Big[ \bigcap_{i \le k} \{f_n \in A_i\} \Big]  \le \prob \Big[ \bigcap_{i \le k} \{f_n^{(r)} + t/2 \in A_i  \} \Big] +   k  \sup_{i \le k} \prob \big[\sup_{x \in U_i} (f_n - f_n^{(r)})(x) \ge t/2\big]\\
 & \quad =  \prod_{i \le k}  \prob [  f_n^{(r)} + t/2 \in A_i  ]   + k   \sup_{i \le k} \prob \left[\sup_{x \in U_i} (f_n - f_n^{(r)})(x) \ge t/2 \right]\\\
 & \quad \le \prod_{i \le k} \min\left\{\prob[  f_n + t \in A_i   ]  +e_{t/2},\, 1\right\}  +  k e_{t/2} \\
 & \quad \le  \prod_{i \le k}    \prob[  f_n + t \in A_i   ]    + 2  k e_{t/2},
   \end{align*}
  where in the first step we used \eqref{e:decomp}, the union bound, and the rotational invariance of $f_{n}$, in the second step that $f_n^{(r)}$ is $r$-range dependent, in the third step \eqref{e:decomp} again (with $f_n$ and $f_n^{(r)}$ interchanged, and restricted to a single $U_i$), and in the final step we bounded  the contribution of each of the $e_{t/2}$ to the product, when multiplied by the other factors, by $e_{t/2}$, since each of the factors is bounded by $1$. Bounding $e_{t/2}$ with \eqref{e:ebound} completes the proof of \eqref{e:sdk1}.

Let us prove \eqref{e:sdk2}. The idea is to `upgrade' \eqref{e:ebound} by bounding the event that $f_n-f_n^{(r)}$ has exceedances in many of the $U_i$ simultaneously. More precisely, for $m\le k$ we denote $\mathcal{I}_m$ to be the collection of all subsets of $\{1,\ldots,k\}$ of cardinality $m$, and seek to bound
\begin{equation}
\label{e:hbound}
h_{m,t} =  \sup_{I \in \mathcal{I}_m }  \prob \Big[ \bigcap_{i \in I}  \Big\{ \sup_{x \in U_i} (f_n - f_n^{(r)})(x) \ge t \Big\} \Big].
\end{equation}
Without loss of generality we suppose $I = \{1,\ldots,m\} \in \mathcal{I}_m$, and define the auxiliary continuous Gaussian process $Z$ on the state space $V = U_1 \times \cdots \times U_m$
 \[  Z_{x_1,\ldots,x_m} = \frac{1}{m} \sum_{i \le m}  (f_n - f_n^{(r)})(x_i)   , \]
 so that $\bigcap\limits_{i \le m} \left\{ \sup_{x \in U_i} (f_n - f_n^{(r)})(x) \ge t\right\}$ implies that $\sup\limits_{x \in V} Z_x \ge t$. Since the $U_i$ are disjoint, we observe that
 \begin{align*}
  \E \left[ \sup_{x \in V} Z_x\right] = & \frac{1}{m} \E \left[ \sup_{x_1 \in U_i,\ldots,x_m \in U_m} \sum_{i \le m}  (f_n - f_n^{(r)})(x_i) \right]   \\
  & = \frac{1}{m}  \sum_{i \le m}  \E \Big[ \sup_{x_i \in U_i }  (f_n - f_n^{(r)})(x_i)  \Big]  \\
  & \le  \sup_{x \in \Uc} \E \Big[ \sup_{y \in \Dc_u(x) \cap \Uc} (f_n - f_n^{(r)})(y)  \Big]  \le c_1\sqrt{ \log (2u \sqrt{n} )} e^{-c_2 r^2 n}
   \end{align*}
   where the final bound uses \eqref{e:kostsup}. Further, by \eqref{e:frakvar}
\begin{align}
\nonumber \label{e:varbound}
 \sup_{x \in V} \var(Z_x)  &=  \frac{1}{m^2} \left( \sup_{x_1 \in U_1, \ldots , x_m \in U_m}  \sum_{1 \le i,j \le m}  \cov\left( (f_n - f_n^{(r)})(x_i), (f_n - f_n^{(r)})(x_j)\right) \right)  \\
 & \le   \frac{1}{m} \sup_{x_1 \in U_1, \ldots , x_m \in U_m}  \sup_{1 \le i \le m}  \sum_{1 \le j \le m} e^{-  c_2 \max\{ r^2 , d_{\Sc^2}(x_i,x_j)^2\}  n }.
 \end{align}
Since there are at most $c_3 j$ points in $\Dc_{(j+1)r} \setminus \Dc_{jr}$ at mutual distance $r$ apart, \eqref{e:varbound} is at most
 \[ \sup_{x \in V} \var(Z_x) \le  \frac{c_3 }{m} \sum_{j \ge 1}  j e^{- c_2 (rj)^2 n }  \le  \frac{c_4}{m} e^{ - c_2 r^2 n }  , \]
 where we used that $r \ge 1/\sqrt{n}$ in the final step.  Assuming that $t \ge 2 c_1  \sqrt{\log (2 u \sqrt{n} )  }  e^{-c_2 r^2 n} $, by the Borell--TIS inequality applied to the process $Z$ we therefore have that
 \begin{equation}
 \label{e:hmt bound}
 h_{m,t} \le  \prob \Big( \sup_{x \in V} Z \ge t \Big) \le  e^{-  c_5 m t^2 e^{c_2 r n^2} }.
 \end{equation}

To conclude the proof of \eqref{e:sdk2} we observe that, for every $t \ge 0$ and $m \in \{ 1,\ldots,k \}$,
\begin{equation}
\label{e:decomp2}
 \bigcap\limits_{i \le k} \{f_n \in A_i\} \ \subseteq \  \bigcup\limits_{I \in \mathcal{I}_m}   \bigcap\limits_{i \in I} \{f_n^{(r)} + t \in A_i  \}   \cup     \bigcup\limits_{I \in \mathcal{I}_{k-m+1} }
 \bigcap\limits_{i \in I}  \Big\{\sup_{x \in U_i} (f_n - f_n^{(r)})(x) \ge t \Big\} ,
 \end{equation}
 where, as above, $\mathcal{I}_m$ denotes the subsets of $\{1,\ldots,k\}$ of cardinality $m$. Similarly to the proof of \eqref{e:sdk1} above, we then have
 \begin{align*}
   \prob \left( \bigcap\limits_{i \le k} \{f_n \in A_i\} \right)  & \le  2^k  \Big( \Big( \sup_{i \le k}    \prob\left(  f_n + t \in A_i   \right)  +e_{t/2} \Big)^{m}   + h_{k-m+1,t/2}   \Big) \\
   &\le  2^{2k}   \left( \left( \sup_{i \le k}   \prob\left(\  f_n + t \in A_i   \right) \right)^m + e_{t/2}^m  + h_{k-m+1,t/2}\right) ,
   \end{align*}
 where $e_t$ is as in \eqref{e:ebound}, and in the last step we used the trivial bound $(a+b)^m \le 2^m (a^m + b^m)$ for $a,b>0$. Choosing $m = \lfloor k/2 \rfloor$, and upon using the bounds \eqref{e:ebound} and \eqref{e:hmt bound}, we have proven that
 \[ \prob \left( \bigcap\limits_{i\le k} \{f_n \in A_i\} \right)  \le 2^{2k}    \left( \left( \sup_{i \le k}   \prob \big(  f_n + t \in A_i   \big) \right)^{\lfloor k/2 \rfloor  } +
 e^{-  c_6 t^2 k  e^{c_2 r n^2} }   \right) . \]
 and the result follows by absorbing the $2^{2k}$ factor into the other terms.
\end{proof}

\begin{proposition}[Sprinkled decoupling for band-limited ensembles]
\label{p:sdble}
Let $g_\ell$ be the band-limited ensemble \eqref{eq:gell band-limited sum}, and let $p$ be as in \eqref{e:p}. Then there exist constants $c_1,c_2  > 0$ satisfying the following. Let $\ell\ge 1$, $\Uc \subseteq \Sc^2$, $r \ge 1/\ell^p$, $u \ge 1/\ell$, $k \ge 2$, a collection $(\{g_\ell \in A_i\})_{i \le k}$ in class $E(\Uc,r,u)$, and $t  \ge c_1  \sqrt{ \log (2 u \ell) } (r \ell^p)^{-1/2}$.
\begin{enumerate}[i.]
\item We have the inequality
\begin{equation}
\label{e:sdbl1}
 \prob \left( \bigcap\limits_{i \le k} \{g_\ell \in A_i \} \right) \le  \prod\limits_{i \le k} \prob \left( \{ g_\ell + t \in A_i \} \right) +   c_1 k e^{-c_2 t^2 r \ell^p }.
 \end{equation}
\item
If, in addition, $\Uc$ is contained in a spherical cap $\Dc_{r'}$ for some $r' + r/2 \le 3\pi/2$, then
 \begin{equation}
\label{e:sdbl2}
 \prob \left( \bigcap_{i \le k} \{g_\ell \in A_i \} \right) \le   \Big( c_1 \sup_i \prob \left( \{ g_\ell + t \in A_i \} \right)    \Big)^{c_2 k}   + c_1  e^{c_1k - c_2 k^{3/4} t^2 (r \ell^p)  }    .
 \end{equation}
\end{enumerate}
\end{proposition}

\begin{proof}
The proof of \eqref{e:sdbl1} is nearly identical to the proof of \eqref{e:sdk1}, except that we first use Proposition \ref{p:frable} and Lemma \ref{l:sup} to deduce that
\begin{equation}
\label{e:glsup}
    \E \left[ \sup_{x \in \Dc_u} (g_\ell - g_\ell^{(r)})(x)  \right]  \le c_1  \sqrt{ \log (2u \ell) }  (r \ell^p)^{-1/2}  ,
    \end{equation}
and then exploit the rotational invariance of $g_\ell - g_\ell^{(r)}$, Proposition \ref{p:frable}, and Lemma \ref{l:sup}, to bound
\begin{align*}
 e_t  :=   \sup_{x \in \Uc}  \prob \left(\sup_{y \in \Dc_u (x) \cap \Uc} |g_\ell - g_\ell^{(r)}|(y) \ge t \right)   \le \prob \left(\sup_{x \in \Dc_u} |g_\ell - g_\ell^{(r)}|(x) \ge t\right)  \le c_1 e^{-c_2 t^2 r \ell^p} .
 \end{align*}
 The proof of \eqref{e:sdbl2} is similar to the proof of \eqref{e:sdk2}, except the bound on $h_{t,m}$ is slightly modified. Let us first assume that $\alpha \in [0,1)$. Define the Gaussian process $Z$ on $V = U_1 \times \cdots \times U_m$
 \[  Z_{x_1,\ldots,x_m} := \frac{1}{m} \sum_{i \le m}  (g_\ell - g_\ell^{(r)})(x)  . \]
 Using the rotational invariance of $g_\ell - g_\ell^{(r)}$ and \eqref{e:glsup} we have
 \[   \E \left[ \sup_{x \in V} Z_x \right] \le   \E \left[ \sup_{x \in \Dc_u} (g_\ell - g_\ell^{(r)})(x)  \right]  \le  c_1 \sqrt{ \log (2u \ell) }   (r \ell)^{-1/2}    .\]
Moreover using Lemma \ref{l:ubble} and Proposition \ref{p:frable}, we have
\begin{align*}
 \sup_{x \in V} \var(Z_x)  &=  \frac{1}{m^2} \left( \sup_{x_1 \in U_1, \ldots , x_m \in U_m}  \sum_{1 \le i,j \le m}  \cov\left( (g_\ell - g_\ell^{(r)})(x_i), (g_\ell - g_\ell^{(r)})(x_j)\right) \right)  \\
 & \le \frac{q_1}{m} + \frac{c_3}{m^2} \sup_{x_1 \in U_1, \ldots , x_m \in U_m}  \sum_{1 \le i \neq j \le m} q_2(d_{\Sc^2}(x_i,x_j) )  \\
 & \le  \frac{c_4}{m}   \Big( (r \ell)^{-1} +  \sum_{i = 1}^{\lceil \sqrt{m} \rceil}  k \ell^{-1}  (k r)^{-3/2} r^{1/2}  \Big)  \\
 & \le   c_5 m^{-3/4} (r  \ell)^{-1}  ,
 \end{align*}
where in order to apply the bound on $q_2$ we used that $d_{\Sc^2}(x_i,x_j) + r/2 \le 3\pi/2$ because of our assumption on $\Uc$.
Using this in place of \eqref{e:hbound} in the course of the proof of \eqref{e:sdk2} gives
 \[ \prob \left[ \bigcap_{i \le m} \{g_\ell \in A_i\} \right]  \le 2^{2k}    \left( \left( \sup_{i \le m}   \prob[  g_\ell + t \in A_i   ] \right)^{\lfloor k/2 \rfloor  } + c_1 e^{-  c_6 t^2  k^{3/4} r \ell }   \right) , \]
 which proves \eqref{e:sdbl2} (note that we cannot absorb the factor $2^{2k}$ in the second term in this case). The proof in the case $\alpha = 1$ is similar, and we omit it.
 \end{proof}

\subsection{A general sprinkled decoupling estimate for spherical ensembles}
\label{s:sdr2}

To conclude the section we give a different, more general, sprinkled decoupling estimate due to \cite{m23}. Unlike the strong sprinkled decoupling results stated above, this estimate only applies to \textit{pairs} of events, and gives a weaker quantitative conclusion on large scales. However, it has two key advantages:
\begin{itemize}
\item It gives useful output even when applied to the random spherical harmonics;
\item When applied to monochromatic band-limited ensembles ($g_\ell$ for $\alpha = 1$ and $\beta \in (0,1)$), it may be applied at the scales less than $\ell^{-\beta}$ (c.f.\ Proposition \ref{p:sdble} which requires $r \ge c \ell^{-\beta}$); this can be crucial when initialising renormalisation arguments (see e.g.\ the proof of Proposition \ref{p:ce} in the case $\alpha = 1$).
\end{itemize}

\begin{proposition}[General sprinkled decoupling estimate]
\label{p:gsd}
Let $f$ be a continuous centred Gaussian field on $\Sc^2$. Then, for every $\Uc \subseteq \Sc^2$, $r, u > 0$, pair of events $(\{f_n \in A_i\})_{i = 1,2}$ in class $E(\Uc,r,u)$, and $t > 0$, we have
\begin{equation}
\label{e:sdgen1}
 \prob \left( \bigcap_{i = 1,2} \{f_n \in A_i \} \right) \le  \prod_{i=1,2} \prob \left( \{ f_n + t  \in A_i \} \right) +  c  t^{-2} \sup_{x \in U_1, y \in U_2} | \E[f(x)\cdot f(y)]  |    ,
 \end{equation}
 where $U_i \subseteq \Uc$, $i=1,2$, are the disjoint subsets in the definition of the pair $(\{f_n \in A_i\})_{i = 1,2} \in E(\Uc,r,u)$, and $c > 0$ is a universal constant.
\end{proposition}
\begin{proof}
For the Euclidean Gaussian fields this is \cite[Theorem 1.1]{m23}, and the result follows for fields on $\Sc^2$ by a suitable mapping from $U_1 \cup U_2$ to $\R^2$.
\end{proof}


\medskip
\section{Local uniqueness}
\label{s:lu}

In this section we study the `local uniqueness' properties of the giant component of spherical ensembles in the supercritical regime. Since we work in the two-dimensional setting, this can be reduced to estimates on \textit{crossing events}, for which we apply a version of Kesten's classical renormalisation scheme. In this section we work mainly with the diametric giant $\Vc^d(t)$, although we also consider implications of local uniqueness for the volumetric giant $\Vc^a(t)$.

\subsection{Definition and basic properties}
\label{s:dfp}
Let $f$ be a continuous isotropic Gaussian field on $\Sc^2$, and let $t  \in \R$ be a level. Recall that $\Uc(t) = \{ x \in \Sc^2 : f(x) \le t\}$ and $ \Vc^d(t)$ denotes the component of $ \Sc^2 \cap \Uc(t)$ of largest diameter; we shall often abbreviate these as $\Uc$ and $\Vc^d$ when $t$ is fixed. For a compact subset $U \subseteq \Sc^2$ let $ \Vc^d(U) = \Vc^d(U;t)$ be the component of $\Uc^d \cap U$ of largest diameter; this is a `local proxy' for $\Vc^d \cap U$, in the sense that $\Vc^d(U)$ depends only on $f|_U$, whereas $\Vc^d \cap U$ does not.

Recall that $\Dc_r(x)$ denotes a spherical cap of radius $r > 0$ centred at $x \in \Sc^2$, and we abbreviate $\Dc_r = \Dc_r(\eta)$ where $\eta$ is the north pole. If $0 < r \le \pi/2$, the \textit{square} $\Sc_r$ of side-length $2r$ is the image of $[-r,r]^2 \subseteq \R^2$ onto $\Sc^2$ under the exponential map $\exp_{\eta}$; its \textit{centre} is the image of the origin. If $r \ge \pi/2$ then by convention we define $\Sc_r = \Sc_{\pi/2}$. For a rotation $\omega \in SO(3)$, $\Sc_r(\omega)$ is the rotation of $\Sc_r$ by $\omega$.

\smallskip
We wish to consider the event that, within the spherical cap $\Dc_r(x)$, a diametric giant component \textit{exists} and is \textit{unique}:
\[ \textrm{EU}_{r,\delta}(x) = \textrm{EU}_{r,\delta}(x;t) = \{  \Uc \cap \Dc_r(x) \textrm{ contains exactly one component of diameter} \ge  \delta r \} .  \]
In what follows we introduce a variant $\widetilde{\textrm{EU}}_{r,\delta}(x)$ of $\textrm{EU}_{r,\delta}(x)$, that is more susceptible to analysis:

\begin{definition}[Local existence and uniqueness]
\label{def:EU tilde def}
$\,$
\begin{enumerate}[i.]

\item For $x\in\Sc^{2}$ and $r>0$, let $\Nc(x,r)$ be the family
\begin{equation*}
\begin{split}
\Nc(x,r)&:= \left\{ \Dc_{r'}(y)  :\: y\in\Sc^{2},\, r' \in [r,2r],\, \Dc_{r'}(y)\subseteq\Dc_{3r}(x)\right\} \\& \qquad \cup
\left\{ \Sc_{r'}(\omega)  :\: \omega\in SO(3),\, r' \in [r,2r],\,\Sc_{r'}(\omega)\subseteq\Dc_{3r}(x)\right\},
\end{split}
\end{equation*}
i.e.\ those spherical caps of radius $r'\in [r,2r]$ that entirely lie in $\Dc_{3r}(x)$, along with those squares of side length
$2r'$ with $r'\in [r,2r]$, also contained in $\Dc_{3r}(x)$.

\item For $x\in\Sc^{2}$, $r > 0$, and $\delta \in (0,1/100]$, define the event
\begin{equation}
\label{eq:EU tilde def}
\begin{split}
&\widetilde{\textrm{EU}}_{r,\delta}(x) = \widetilde{\textrm{EU}}_{r,\delta}(x; t)  \\& :=
\! \! \! \bigcap_{U\in \Nc(x,r)} \! \{ \textrm{every closed connected subset of $U$ of diameter $\ge  \delta r$ intersects $\Vc^d(U;t)$} \} .
\end{split}
\end{equation}
\end{enumerate}

\end{definition}

The event $\widetilde{\textrm{EU}}_{r,\delta}(x) $ enjoys the following important properties, with the first being a key advantage of $\widetilde{\textrm{EU}}_{r,\delta}(x)$ over $\textrm{EU}_{r,\delta}(x)$.
\begin{enumerate}[1.]
\item It is \textit{increasing} in the level $t$: $\widetilde{\textrm{EU}}_{r,\delta}(x;t) $ implies $\widetilde{\textrm{EU}}_{r,\delta}(x;t') $ for every $t' \ge t$;
\item It is \textit{local}:  $\widetilde{\textrm{EU}}_{r,\delta}(x) $ depends only on the restriction of $f$ to $\Dc_{3r}(x)$;
\item It is \textit{dominating}: $\widetilde{\textrm{EU}}_{r,\delta}(x)$ implies $\textrm{EU}_{r',\delta}(y)$ for every $r' \in [r,2r]$ and $y  \in \Dc_r(x)$.
\end{enumerate}

To see the monotonicity, observe that if $\widetilde{\textrm{EU}}_{r,\delta}(x;t) $ holds, then for every $U\in \Nc(x,r)$, $\Vc^d(U;t)$ has diameter $ \ge \delta r$ (since it must intersect two spherical caps of radius $\delta r$ positioned at distance $\ge \delta r$ inside $U$) while the components of $U \setminus \Vc^d(U;t)$ have diameter at most $\delta r$, so that $\Vc^d(U;t) \subseteq \Vc^d(U;t')$ for every $t' \ge t$.

\vspace{2mm}
We next develop some further consequences of the event $\widetilde{\textrm{EU}}_{r,\delta}(x) $. First we state an important {\em inheritability} property:

\begin{lemma}[Inheritability]
\label{l:inherit}
Let $r > 0$, $\delta \in (0,1/100]$, suppose $\widetilde{\textrm{EU}}_{r,\delta}(x)$ holds, and assume that
$U_1,U_2 \in  \Nc(x,r)$, where $\Nc(x,r)$ is as in Definition \ref{def:EU tilde def}(i).
One has:
\begin{enumerate}
\item[(I1)]  If $U_1 \cap U_2$ contains a spherical cap of radius $r/20$, then $\Vc^d(U_1)$ and $\Vc^d(U_2)$ have non-empty intersection.
\item[(I2)] If $U_1 \subseteq U_2$, then $\Vc^d(U_1) \subseteq \Vc^d(U_2)$.
\end{enumerate}
\end{lemma}
\begin{proof}
$\,$
\textrm{(I1)} Consider two spherical caps $\Dc_1, \Dc_2$ inside $U_1 \cap U_2$ of radius $r/50$ separated by a distance $r/100$ and at distance at least $r/100$ of the boundary of $U_1 \cap U_2$. Since $\widetilde{\textrm{EU}}_{r,\delta}(x)$ holds, each of $\Dc_1$ and $\Dc_2$ intersect  $\Vc(U_1)$, and so in particular $\Vc(U_1)$ contains a path inside $U_1 \cap U_2$ of diameter at least $r/100$. This path must intersect $\Vc(U_2)$.

\textrm{(I2)} This follows from (I1) by the following observation: since $U_1 \subseteq U_2$, every component of $\Uc \cap U_1$ that intersects $\Vc(U_2)$ must be contained in $\Vc(U_2)$.
 \end{proof}

We shall call the claims (I1) and (I2) of Lemma \ref{l:inherit} the `first' and `second' inheritability properties respectively. The following lemma develops further consequences.  Denote by
\begin{equation}
\label{d:e}
 e_{r,\delta} = e_{r,\delta;t} := 1 - \prob(\widetilde{\textrm{EU}}_{r,\delta}(x))
 \end{equation}
the probability that the local existence and uniqueness fails at the scale $r$ (by the rotational invariance this probability is independent of $x$), and abbreviate $e_{r} = e_{r,1/100}$.

\begin{lemma}
\label{l:lu}
Let $r > 0$, and let $e_{r} = e_{r,1/100}$ be as in \eqref{d:e}. Then one has:
\begin{enumerate}[i.]
\item  For every $\eps > 0$ there exists a constant $c > 0$, depending only on $\eps$, such that
\[ \prob \big( \text{$\Uc$ has at most one component of diameter $\ge \eps$} \big) \ge 1 -  c \sup_{r' \ge 1/c} e_{r'} . \]

\item Let $E_r$ be the event that,
for every spherical cap $\Dc$ of radius $r' \ge r$, $\Uc \cap \Dc$ contains exactly one component of diameter $\ge r'/ 100$, and that component is contained in $\Vc^d$. There exists an absolute constant $c > 0$ such, for every $r> 0$, one has
\[  \prob(E_r ) \ge 1 - c  r^{-2} (\log 1/r) \sup_{r' \ge r} e_{r'}.\]

\item There exists an absolute constant $c > 0$ such that, for every $x \in \Sc^2$, $r, r' > 0$ and
$U \in  \{ \Dc_{r'}(y), \Sc_{r'}(\omega)\} $ so that $\Dc_r(x) \subseteq U$, one has
\[   \prob\big( \text{$\Vc^d( \Dc_r(x))$ is contained in $\Vc^d( U) $} \big) \ge 1 -   \frac{c r'}{ d_{\Sc^2}(x,\partial U ) }   \sum_{k = 0}^{\lfloor \log_2(\pi/r') \rfloor} e_{2^k  r} . \]
\end{enumerate}
\end{lemma}
\begin{proof}
(i). Let $\eps > 0$ be given, and set $r = 100 \eps$. Fix a finite set of points $(x_i)_{i \ge 1}$ on the sphere, with cardinality depending only on $\eps$, such that, for all $i \ge 1$, $d_{\Sc^2}(x_{i+1},x_i) \le r$, and $\Sc^2 \subseteq \cup_i \Dc_{r}(x_i)$. Assume that the event
\[ \bigcap\limits_{0 \le k \le   \lfloor \log_2(\pi/r) \rfloor} \widetilde{\textrm{EU}}_{ 2^k r}(x_1) \ \cap  \bigcap\limits_{i}  \widetilde{\textrm{EU}}_{r}(x_i)  \]
holds; we will argue that $\Uc$ has at most one component of diameter $\ge \eps = r/100$, so that, bearing in mind that the number of the events in the intersection only depends on $\epsilon$, the claim follows from the union bound. By (I2) of Lemma \ref{l:inherit}, applied iteratively to $\Vc^d(\Dc_{2^k r}(x_1))$, $0 \le k \le   \lfloor \log_2(\pi/r) \rfloor$, $\Vc^d(\Dc_{r}(x_1))$ is contained in $\Vc^d$. Then by (I1) applied iteratively to $\Vc^d( \Dc_{r}(x_i))$ and $\Vc^d( \Dc_{r}(x_{i+1}))$, it follows that $\Vc^d(\Dc_{r}(x_i))$ is contained in $\Vc^d$ for each $i$.  Now suppose by contradiction that $\Uc \setminus \Vc^d$ has a component $\mathcal{W}$ of diameter $\ge r/100$. Then we can find $y \in \mathcal{W}$ such that $\mathcal{W} \cap \Dc_{r}(y)$ has a component $W$ of diameter
$\ge r/100$. Then there exists $x_{i}$ such that $y \in \Dc_{r}(x_{i})$, and so by (I1) again, $\Vc^d(\Dc_{r }(y))$ is contained in $\Vc^d$. But then $\Uc \cap \Dc_{r }(y)$ contains two distinct components, $W$ and $\Vc^d(\Dc_{r }(y))$, of diameter $\ge \eps =r/100$, which contradicts $\widetilde{\textrm{EU}}_{r}(x_i) $.

(ii). Let $r > 0$ be given, and fix a finite set of points  $(x_i)_{i \ge 1}$ on the sphere, with cardinality $c r^{-2}$ for an absolute $c > 0$, such that, for all $i \ge 1$, $d_{\Sc^2}(x_{i+1},x_i) \le r$, and $\Sc^2 \subseteq \bigcup\limits_{i} \Dc_{r}(x_i)$. Assume that the event
\[\bigcap\limits_{0 \le k  \le \lfloor \log_2(\pi/r) \rfloor} \bigcap\limits_{i} \widetilde{\textrm{EU}}_{2^k r}(x_i) \]
holds; we argue that this implies $E_r$, so that the statement follows from the union bound. Let $\Dc_{r'}(x)$ be a spherical cap of radius $2^k r \le r' \le 2^{k+1} r$, with some $k\ge 0$. One may find $x_i$ such that $x \in \Dc_{2^{k+1} r}(x_i)$. By (I1) and (I2), $\Vc^d(\Dc_{r'}(x))$ is contained in  $\Vc^d(\Dc_{2^{k+1}r}(x))$, which intersects $\Vc^d(\Dc_{2^{k+1}r}(x_i))$, which is contained in $\Vc^d(\Dc_{2^{k'}r}(x_i) )  $ for $k' \ge k+1$. In particular taking $k' = \lfloor \log_2(\pi/r) \rfloor$, $\Vc^d( \Dc_{r'}(x))$ is contained in $\Vc^d$. Since the event $ \widetilde{\textrm{EU}}_{2^k r}(x_i) $ also implies that $\Vc^d( \Dc_{r'}(x))$ can also have at most one component with diameter $\ge r'/100$, the claim follows.

(iii). Let $x \in \Sc^2$,  $r > 0$ and $U$ be given, abbreviate $d = d_{\Sc^2}(x,\partial U)$,  and let $k' = \lfloor \log_2(d/r') \rfloor$. Fix a finite set of points  $(x_i)_{i \le 1}$ with cardinality bounded by $2r' / d$, such that, for all $i \ge 1$, $d_{\Sc^2}(x_{i+1},x_i) \le d$, $x_1 = x$, and $x_n$ is the centre of $U$. Assume that
\[ \bigcap\limits_{0 \le k \le  k' }  \widetilde{\textrm{EU}}_{ 2^k r}(y) \  \cap \ \bigcap\limits_{i}  \widetilde{\textrm{EU}}_{ 2^{k'} r}(x_i) \ \cap \ \bigcap\limits_{k' \le k \le \lfloor \log_2(\pi/r') \rfloor}  \widetilde{\textrm{EU}}_{2^{k} r}(x) \]
holds; we argue that this implies that $\Vc^d(\Dc_r(x))$ is contained in $\Vc^d(U)$, so that the statement follows from the union bound. By (I1) and (I2) applied iteratively,  $\Vc^d( \Dc_{r}(x))$  is contained in $\Vc^d( \Dc_{2^{k'} r}(x))$, which intersects $\Vc^d( \Dc_{2^{k'} r}(x_i))$ for each $i$, and therefore intersects $\Vc^d( \Dc_{2^{k'} r}(x_n))$. Let $k''$ be the largest integer such that $\Dc_{2^{k''} r}(x_n)$ is contained in $U$. Then by (I2) again applied iteratively we find that $\Vc^d( \Dc_{2^{k'} r}(x_n))$ is contained in $\Vc^d(\Dc_{2^{k''} r}(x_n))$. Applying (I2) one final time, $\Vc^d(\Dc_{2^{k''} r}(x_n))$ is contained in $\Vc^d(U)$, which completes the proof.
\end{proof}

\subsection{Consequences for the volumetric giant}
We now consider some implications of local uniqueness for the volumetric giant $\Vc^a = \Vc^a(t)$, defined to be the component of $\Sc^2 \cap \Uc(t)$ of largest area. Similarly, for a compact subset $U \subseteq \Sc^2$, let $ \Vc^a(U) = \Vc^a(U;t)$ be the component of $\Uc \cap U$ of largest area.

\vspace{2mm}
We first argue that local uniqueness allows us to compare $\textrm{Area}(\Vc^a(U))$ with $\textrm{Area}(\Vc^a \cap U)$:

\begin{lemma}[Local volume bound]
\label{l:volbound}
There exists $\rho > 0$ such that the following holds. Let $r \in (0,\rho)$, $\delta \in (0,1/100]$, suppose $\widetilde{\textrm{EU}}_{r,\delta}(x)$ holds, and assume that
\[ U_1 \in \Nc(x,r),    \]
with $\Nc(x,r)$ as in Definition \ref{def:EU tilde def}(i). Then for every $U_2 \subseteq \Sc^2$ such that $U_1 \subseteq U_2$, one has
\[ \textrm{Area}(  \Vc^a(U_2) \cap  U_1 ) \le \textrm{Area}( \Vc^a(U_1)  ) + 3 \delta \textrm{Area}(U_1) . \]
\end{lemma}
\begin{proof}
Since $\widetilde{\textrm{EU}}_{r,\delta}(x)$ holds, its definition \eqref{eq:EU tilde def} implies that at most one component of $\Vc^a(U_2) \cap  U_1$ is not contained in
\[  \partial^- U_1 := \{x \in U_1 : d_{\Sc^2}(x, \partial U_1) \le \delta r \} . \]
 Let $W$ denote this component if it exists. Observe that if $W$ is not equal to $\Vc^a(U_1)$, then the diameter of $\Vc^a(U_1)$ is at most $\delta r$, so that
\[ \textrm{Area}(W) \le \textrm{Area}(\Vc^a(U_1)) \le  \area(\Dc_{\delta r} ) \le  (\pi/4) \delta^2 r^2 .\]
This shows that
\[ \textrm{Area}(  \Vc^a(U_2) \cap  U_1 ) \le \textrm{Area}( \Vc^a(U_1)  )  + \textrm{Area}(  \partial^- U_1 )  +  ( \pi/4) \delta^2 r^2. \]
To conclude the proof we claim that
\[  \frac{\textrm{Area}(  \partial^- U_1 ) }{ \textrm{Area}(U_1) }   \le \frac{5 \delta}{2}  \quad \text{and} \quad  \frac{ \delta \textrm{Area}(U_1) }{2}  \ge  \frac{\pi r^2 \delta}{400} \ge (\pi /4)\delta^2 r^ 2 .  \]
To see this, note that if $U \subseteq \R^2$ is either the square $[-r,r]^2$ or the ball $B_r$ and $\delta \le 1/100$ then
\[ \frac{\textrm{Area}(  \{x \in U : d(x, \partial U) \le \delta r \} )}{ \textrm{Area}(U) }   \le (1+\delta)^2 - 1 = 2\delta + \delta^2 < (2 + 1/100)\delta \]
and $\textrm{Area}(U)  \ge \pi r^2$. Then the claim follows from applying the exponential map, which locally preserves area up to an arbitrarily small error.
\end{proof}

We next observe that under the local uniqueness either $\Vc^d(U) = \Vc^a(U)$, or $\textrm{Area}(\Vc^a(U))$ is small:

\begin{lemma}
\label{l:smallgiant}
There exists $\rho > 0$ such that the following holds. Let $r \in (0,\rho)$, $\delta \in (0,1/100]$, suppose $\widetilde{\textrm{EU}}_{r,\delta}(x)$ holds, and assume that
\[ U  \in \Nc(x,r),    \]
with $\Nc(x,r)$ as in Definition \ref{def:EU tilde def}(i). Then either $\Vc^d(U) = \Vc^a(U)$ or $\textrm{Area}(\Vc^a(U)) \le \delta \area(U)$.
\end{lemma}
\begin{proof}
Since $\widetilde{\textrm{EU}}_{r,\delta}(x)$ holds, if $\Vc^d(U) \neq \Vc^a(U)$ then $\Vc^a(U)$ has diameter at most $\delta r$, so that, as in the proof of Lemma \ref{l:volbound},
\begin{equation*}
  \textrm{Area}(\Vc^a(U)) \le (\pi/4) \delta^2 r^2 \le \delta \area(U) . \qedhere
  \end{equation*}
\end{proof}

\subsection{On the failure of local uniqueness}
In this section we derive some estimates on the quantity $e_{r,\delta}$ in \eqref{d:e} for the spherical ensembles under consideration. These estimates, combined with Lemma \ref{l:lu}, are readily sufficient to yield the ubiquity of the giant component and their local uniqueness asserted as part of Theorems \ref{thm:unique giant Kostlan} and~\ref{thm:unique giant band-lim}.

\begin{proposition}
\label{p:lu}
Let $t',\delta' > 0$ and $\delta \in (0,1/100]$ be given. Let either $e_{r,\delta;t}=e_{r,\delta;t,n}$ be the probability \eqref{d:e} associated to Kostlan's ensemble $f_n$ in \eqref{eq:fn Kostlan}, or $e_{r,\delta;t}=e_{r,\delta;t,\ell}$ associated to either the band limited ensembles $g_\ell$ in \eqref{eq:gell band-limited sum}
 or the random spherical harmonics $T_\ell$ in \eqref{eq:Tl spher harm}.

\begin{enumerate}[i.]
\item For Kostlan's ensemble, there exist $c_1,c_2 > 0$ such that for all $t \ge t'$ and $r \in [1/\sqrt{n}, \pi]$,
\[ e_{r,\delta;t}  \le c_1 e^{- c_2  r \sqrt{n}  }  .  \]

\item For the band-limited ensembles, there exists a constant $c_1 > 0$, possibly depending on the ensemble, such that for all  $t \ge t'$ and $r \in [1/\ell,\pi]$,
\[ e_{r,\delta;t}  \le  e^{-  c_1 r \ell^{p-\delta'} }   ,\]
with $p$ as in \eqref{e:p}.

\item For the band-limited ensembles and the random spherical harmonics, there exists a constant $c_1 > 0$, depending only on $t',\delta,\delta'$, such that for all $t \ge t'$ and $r \in [1/\ell,\pi]$,
\[ e_{r,\delta;t}  \le    c_1 (r \ell)^{-1/2}  (\log (1+ r\ell))^{2+\delta'} .\]
\end{enumerate}
\end{proposition}

We shall deduce Proposition \ref{p:lu} from an analogous result for \textit{crossing events}, which are amenable to renormalisation. Let us state this result first. For $r \in [0,\pi/2]$, recall the `square' $\Sc_r \subseteq \Sc^2$, and let $\textrm{AnnCross}(t,r)$ denote the \textit{annulus crossing} event that $\Sc_r$ and $\partial \Sc_{2r}$ are connected in $\Uc(t)$, in the context of either of the three spherical ensembles, see Figure \ref{f:renorm} (left). Note that, unlike the events $\widetilde{\textrm{EU}}_{r,\delta}(x)$, $\textrm{AnnCross}(t,r)$ does not depend on $\delta$.

\begin{proposition}
\label{p:ce}
Let $t',\delta' > 0$ be given. The assertions of Proposition \ref{p:lu} hold with $e_{r,\delta;t}$ replaced by $\prob(\textrm{AnnCross}(-t,r))$.
\end{proposition}

\begin{proof}[Proof of Proposition \ref{p:lu} assuming Proposition \ref{p:ce}]
Let $\textrm{AnnCirc}(t,r)$ denote the \textit{annulus circuit} event that $\Uc(t)$ contains a loop inside
$\Sc_{2r} \setminus \Sc_r$ that encloses $\Sc_r$, see Figure \ref{f:lu} (left). Hence, by its definition, $\textrm{AnnCirc}(t,r)$ is complementary to the annulus crossing event $\textrm{AnnCross}^\ast(t,r)$ that $\Sc_r$ and $\partial \Sc_{2r}$ are connected in the complement excursion set $f^{-1}((t,\infty)) = \Sc^{2}\setminus \Uc(t)$. By the equivalence in law of $f$ and $-f$, and the stability of Proposition \ref{p:stab}, we deduce that
\[ \prob(\textrm{AnnCirc}(t,r)) = 1 -  \prob( \textrm{AnnCross}^\ast(t,r)) =   1 -  \prob( \textrm{AnnCross}(-t,r)) . \]
Fix $\delta \in (0,1/100]$. Then moreover there exists a constant $c = c(\delta) \ge 1$ such that the following holds: for every $r \in [1/\sqrt{n},\pi]$ there exists a collection of $< c$ rotated copies of $\textrm{AnnCirc}(t,r/c)$ such that $\widetilde{\textrm{EU}}_{r,\delta}(\eta)$ is implied by their intersection, see Figure \ref{f:lu} (right). The result follows from Proposition \ref{p:ce} and the union bound.
\end{proof}

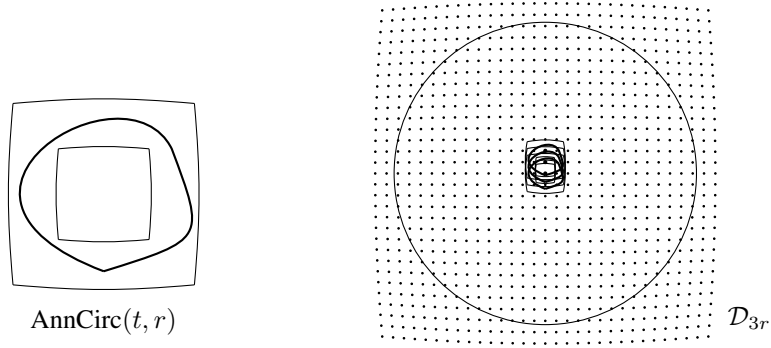
\begin{figure}
\begin{tikzpicture}
      \node[above] at (-4.5,-2) {\footnotesize{$\textrm{AnnCirc}(t,r)$}};
  \draw[scale=3,shift={(-1.5,0)},line width=0.8pt] (0,-0.34) .. controls (0.45,-0.2) .. (0.3,0.2)
                               .. controls (0.1,0.6) and (-0.9,0) .. (0,-0.34);
     \draw[scale=3,shift={(-1.5,0)}] (-0.4,0.4) to[out=5,in=175] (0.4,0.4);
          \draw[scale=3,shift={(-1.5,0)}] (0.4,0.4) to[out=275,in=85] (0.4,-0.4);
               \draw[scale=3,shift={(-1.5,0)}] (0.4,-0.4) to[out=185,in=-5] (-0.4,-0.4);
                    \draw[scale=3,shift={(-1.5,0)}] (-0.4,-0.4) to[out=95,in=265] (-0.4,0.4);
         \draw[scale=1.5,shift={(-3,0)}] (-0.4,0.4) to[out=5,in=175] (0.4,0.4);
          \draw[scale=1.5,shift={(-3,0)}] (0.4,0.4) to[out=275,in=85] (0.4,-0.4);
               \draw[scale=1.5,shift={(-3,0)}] (0.4,-0.4) to[out=185,in=-5] (-0.4,-0.4);
                    \draw[scale=1.5,shift={(-3,0)}] (-0.4,-0.4) to[out=95,in=265] (-0.4,0.4);
                    \end{tikzpicture}
                    \hspace{2cm}
                    \begin{tikzpicture}
  \draw[scale=0.6,shift={(0,0)},line width=0.8pt] (0,-0.34) .. controls (0.45,-0.2) .. (0.3,0.2)
                               .. controls (0.1,0.6) and (-0.9,0) .. (0,-0.34);
     \draw[scale=0.6,shift={(0,0)}] (-0.4,0.4) to[out=10,in=170] (0.4,0.4);
          \draw[scale=0.6,shift={(0,0)}] (0.4,0.4) to[out=280,in=80] (0.4,-0.4);
               \draw[scale=0.6,shift={(0,0)}] (0.4,-0.4) to[out=190,in=-10] (-0.4,-0.4);
                    \draw[scale=0.6,shift={(0,0)}] (-0.4,-0.4) to[out=100,in=260] (-0.4,0.4);
         \draw[scale=0.3,shift={(0,0)}] (-0.4,0.4) to[out=10,in=170] (0.4,0.4);
          \draw[scale=0.3,shift={(0,0)}] (0.4,0.4) to[out=280,in=80] (0.4,-0.4);
               \draw[scale=0.3,shift={(0,0)}] (0.4,-0.4) to[out=190,in=-10] (-0.4,-0.4);
                    \draw[scale=0.3,shift={(0,0)}] (-0.4,-0.4) to[out=100,in=260] (-0.4,0.4);
  \draw[scale=0.6,shift={(0,0.15)},line width=0.8pt] (0,-0.34) .. controls (0.45,-0.2) .. (0.3,0.2)
                               .. controls (0.1,0.6) and (-0.9,0) .. (0,-0.34);
     \draw[scale=0.6,shift={(0,0.15)}] (-0.4,0.4) to[out=10,in=170] (0.4,0.4);
          \draw[scale=0.6,shift={(0,0.15)}] (0.4,0.4) to[out=280,in=80] (0.4,-0.4);
               \draw[scale=0.6,shift={(0,0.15)}] (0.4,-0.4) to[out=190,in=-10] (-0.4,-0.4);
                    \draw[scale=0.6,shift={(0,0.15)}] (-0.4,-0.4) to[out=100,in=260] (-0.4,0.4);
         \draw[scale=0.3,shift={(0,0.15)}] (-0.4,0.4) to[out=10,in=170] (0.4,0.4);
          \draw[scale=0.3,shift={(0,0.15)}] (0.4,0.4) to[out=280,in=80] (0.4,-0.4);
               \draw[scale=0.3,shift={(0,0.15)}] (0.4,-0.4) to[out=190,in=-10] (-0.4,-0.4);
                    \draw[scale=0.3,shift={(0,0.15)}] (-0.4,-0.4) to[out=100,in=260] (-0.4,0.4);
        \draw[scale=0.6,shift={(0,0.3)},line width=0.8pt] (0,-0.34) .. controls (0.45,-0.2) .. (0.3,0.2)
                               .. controls (0.1,0.6) and (-0.9,0) .. (0,-0.34);
     \draw[scale=0.6,shift={(0,0.3)}] (-0.4,0.4) to[out=10,in=170] (0.4,0.4);
          \draw[scale=0.6,shift={(0,0.3)}] (0.4,0.4) to[out=280,in=80] (0.4,-0.4);
               \draw[scale=0.6,shift={(0,0.3)}] (0.4,-0.4) to[out=190,in=-10] (-0.4,-0.4);
                    \draw[scale=0.6,shift={(0,0.3)}] (-0.4,-0.4) to[out=100,in=260] (-0.4,0.4);
         \draw[scale=0.3,shift={(0,0.3)}] (-0.4,0.4) to[out=10,in=170] (0.4,0.4);
          \draw[scale=0.3,shift={(0,0.3)}] (0.4,0.4) to[out=280,in=80] (0.4,-0.4);
               \draw[scale=0.3,shift={(0,0.3)}] (0.4,-0.4) to[out=190,in=-10] (-0.4,-0.4);
                    \draw[scale=0.3,shift={(0,0.3)}] (-0.4,-0.4) to[out=100,in=260] (-0.4,0.4);
\filldraw[black] (0,0) circle (0.02cm);
\filldraw[black] (0,0.15) circle (0.02cm);
\filldraw[black] (0,0.3) circle (0.02cm);
 \foreach \x in {-15,...,15}
    \foreach \y in {-15,...,15}
          {  \filldraw[black] ({0.15*\x*cos(1.1*\y)},{0.15*\y*cos(1.1*\x)}) circle (0.01cm) ; } ;
\draw (0,0) circle (2cm);
\node[above] at (2.7,-2.2) {\footnotesize{$\Dc_{3r}$}};
\end{tikzpicture}
\caption{Left: The event $\textrm{AnnCirc}(t,r)$. Right: The event $\widetilde{\textrm{EU}}_{r,\delta}(\eta)$ is implied by the intersection of $<c_\delta$ rotations of the $\textrm{AnnCirc}(t, \delta r)$ mapping the centre of $\Sc_{r}$ to the points of a grid that covers $\Dc_{3r}$ (only some of the events are depicted). The grid is obtained by mapping a Euclidean grid with spacing $(\delta/100) r$ by the exponential map $T_\eta$; by taking $r$ sufficiently small the geometry of the grid is approximately Euclidean.}
\label{f:lu}
\end{figure}

We prove Proposition \ref{p:ce} via a well-known renormalisation scheme (see \cite[Section 5]{kes82} for a version for Bernoulli percolation, and \cite{MV20,MS23,m23} for recent applications in the context of Gaussian fields), except that the details differ slightly in our case because we work with spherical ensembles:

\begin{figure}
\begin{tikzpicture}
\begin{scope}[shift={(0,2)},scale=4]
     \draw(-0.4,0.4) to[out=5,in=175] (0.4,0.4);
          \draw (0.4,0.4) to[out=275,in=85] (0.4,-0.4);
               \draw(0.4,-0.4) to[out=185,in=-5] (-0.4,-0.4);
                    \draw (-0.4,-0.4) to[out=95,in=265] (-0.4,0.4);
                    \end{scope}
       \begin{scope}[shift={(0,2)},scale=2]
     \draw(-0.4,0.4) to[out=5,in=175] (0.4,0.4);
          \draw (0.4,0.4) to[out=275,in=85] (0.4,-0.4);
               \draw(0.4,-0.4) to[out=185,in=-5] (-0.4,-0.4);
                    \draw (-0.4,-0.4) to[out=95,in=265] (-0.4,0.4);
                    \end{scope}
                        \node[above] at (0,-0.5) {\footnotesize{$\textrm{AnnCross}(t,r)$}};
        \begin{scope}[shift={(0,2)},scale=0.5]
                  \draw[line width=0.8pt] (0.1,1.65) .. controls (0.3,2.1) and (0.2,2.2) .. (0.2,2.5)
                               .. controls (0.2,2.8) and (0.1,3.2) .. (0.3,3.37);
      \end{scope}
\end{tikzpicture}
\hspace{2cm}
\begin{tikzpicture}
\begin{scope}[scale=8]
     \draw(-0.4,0.4) to[out=5,in=175] (0.4,0.4);
          \draw (0.4,0.4) to[out=275,in=90] (0.42,0);
                    \draw (-0.42,0) to[out=90,in=265] (-0.4,0.4);
                    \end{scope}
         \begin{scope}[scale=4]
     \draw(-0.4,0.4) to[out=5,in=175] (0.4,0.4);
          \draw (0.4,0.4) to[out=275,in=90] (0.42,0);
                    \draw (-0.42,0) to[out=90,in=265] (-0.4,0.4);
                    \end{scope}
     \begin{scope}[shift={(0,1.7)},scale=1.4]
     \draw(-0.4,0.4) to[out=5,in=175] (0.4,0.4);
          \draw (0.4,0.4) to[out=275,in=85] (0.4,-0.4);
               \draw(0.4,-0.4) to[out=185,in=-5] (-0.4,-0.4);
                    \draw (-0.4,-0.4) to[out=95,in=265] (-0.4,0.4);
                    \end{scope}
    \begin{scope}[shift={(0,1.7)},scale=0.7]
     \draw(-0.4,0.4) to[out=5,in=175] (0.4,0.4);
          \draw (0.4,0.4) to[out=275,in=85] (0.4,-0.4);
               \draw(0.4,-0.4) to[out=185,in=-5] (-0.4,-0.4);
                    \draw (-0.4,-0.4) to[out=95,in=265] (-0.4,0.4);
                    \end{scope}
           \begin{scope}[shift={(0.3,3.37)},scale=1.4]
     \draw(-0.4,0.4) to[out=5,in=175] (0.4,0.4);
          \draw (0.4,0.4) to[out=275,in=85] (0.4,-0.4);
               \draw(0.4,-0.4) to[out=185,in=-5] (-0.4,-0.4);
                    \draw (-0.4,-0.4) to[out=95,in=265] (-0.4,0.4);
                    \end{scope}
    \begin{scope}[shift={(0.3,3.37)},scale=0.7]
     \draw(-0.4,0.4) to[out=5,in=175] (0.4,0.4);
          \draw (0.4,0.4) to[out=275,in=85] (0.4,-0.4);
               \draw(0.4,-0.4) to[out=185,in=-5] (-0.4,-0.4);
                    \draw (-0.4,-0.4) to[out=95,in=265] (-0.4,0.4);
                    \end{scope}
        \node[above] at (2.1,-0.5) {\footnotesize{$\partial \Sc_{R}$}};
                \node[above] at (3.9,-0.5) {\footnotesize{$\partial \Sc_{2R}$}};
                \node[above] at (0.7,0.6) {\footnotesize{$\omega_y \Sc_{2r}$}};
                \node[above] at (1.2,2.3) {\footnotesize{$\omega_{y'} \Sc_{2r}$}};
                  \draw[line width=0.8pt] (0.1,1.65) .. controls (0.3,2.1) and (0.2,2.2) .. (0.2,2.5)
                               .. controls (0.2,2.8) and (0.1,3.2) .. (0.3,3.37);
\end{tikzpicture}
\caption{Left: The event $\textrm{AnnCross}(t,r)$. Right: If $r \in [R/10, R]$, the event $\textrm{AnnCross}(t,R)$ implies the occurrence of two rotated copies of $\textrm{AnnCross}(t,r)$, centred at $y \in \partial \Sc_R$ and $y' \in \partial \Sc_{2R}$, among a bounded collection of such events.}
\label{f:renorm}
\end{figure}
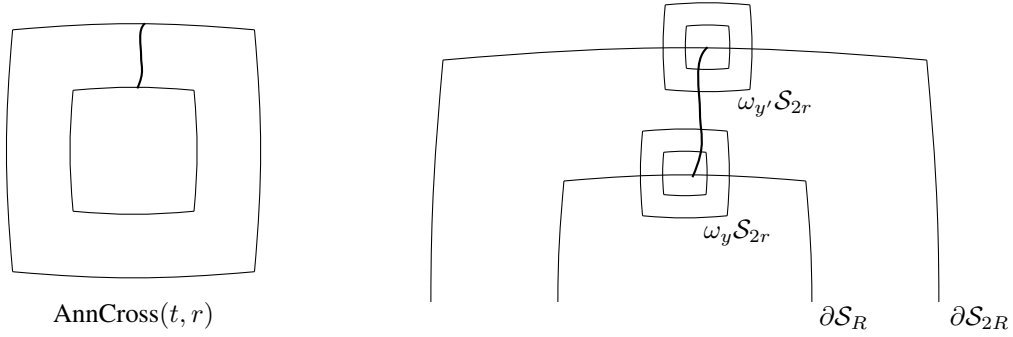

\begin{proof}[Proof of Proposition \ref{p:ce}]
We need to prove each of the three statements of Proposition \ref{p:lu} with $e_{r,\delta;t}$ replaced by $\prob(\textrm{AnnCross}(-t,r))$. Let $\delta \in (0,1)$ be given. Let $F_n$ be a sequence of smooth isotropic Gaussian fields on the sphere $\Sc^2$ that converges locally to a field $F_\infty$ at the scale $s_n \to 0$ in the sense of Definition \ref{d:lsl}. Recall that, when restricted to a small neighbourhood of the origin, the exponential map is preserving distances up to an arbitrarily small error. Then the classical Euclidean renormalisation scheme (see \cite[Lemma 4.3]{MS23}), suitably modified to the spherical setting, yields that, for $\rho > 0$ an absolute constant chosen sufficiently small, and every $0 < R/10 < r < R < \rho$,
\begin{align}
\label{e:renorm}
& \prob[F_n \in \textrm{AnnCross}(-t,R) ] \\
\nonumber & \qquad \le c_1  \sup_{y \in \partial \Sc_R, y' \in \partial \Sc_{2R}} \prob\left( F_n \in \omega_y \textrm{AnnCross}(-t,r) \cap  \omega_{y'}\textrm{AnnCross}_n(-t,r)  \right),
 \end{align}
where $c_1 > 0$ is an absolute constant, and $\omega_y \in SO(3)$ is the (unique) rotation that maps the north pole $\eta$ to $y$ and fixes the unique geodesic between these points. We further observe that, for $y \in \partial \Sc_R, y' \in \partial \Sc_{2R}$, and assuming that $R > 2r$, the pair of events
\begin{equation}
\label{e:rotated AnnCross}
\{\omega_y \textrm{AnnCross}(-t,r) ,  \omega_{y'} \textrm{AnnCross}(-t,r) \}
\end{equation}
are in class $E(\Sc_{c_2 R}, c_2 (R-2r) , c_2 r)$, as defined in section \ref{s:sdr}, with some absolute constant $c_2 > 0$. On the other hand, by the local convergence and the stability in Proposition \ref{p:stab}, for every fixed $R > 0$, as $n \to \infty$,
 \[ \prob[ F \in \textrm{AnnCross}_n(-t, R s_n) ] \to   \prob[ F_\infty  \in \textrm{AnnCross}(-t, R ) ] , \]
 which is used to initialise the renormalisation scheme.

Let us first analyse the renormalisation scheme in the context of Kostlan's ensemble $f_n$, which will prove the first statement. Fix constants $\rho > 0$ and  $R_0 > 1$ to be chosen later, and define the sequences $R_{m+1} = 2 R_m + \sqrt{R_m}$ and $t_{m+1} = t_m +  (\log R_m)^{-2}$, with $t_0 = t/2$. By choosing $R_0$ sufficiently large we can ensure that
\[ t_\infty := \lim_{m \to \infty} t_m =  t_0 + \sum_{m \ge 0} (\log R_m )^{-2}  <  t . \] 
Fix $n$ and define
\[ a_{R_m} = \prob \big(  \textrm{AnnCross}(-t_m,R_m / \sqrt{n} )  \big) . \]
By \eqref{e:renorm} (with $R = R_{m+1} / \sqrt{n}$ and $r = R_m / \sqrt{n}$), and the first statement of Proposition \ref{p:sdk}
(with $r = c_2 (R_{m+1} - 2R_m) / \sqrt{n}  = c_2 \sqrt{R_m} / \sqrt{n}$ and $u = c_2 R_m / \sqrt{n}$, and the corresponding pair of events
\eqref{e:rotated AnnCross}), we derive the renormalisation equation
\[ a_{2R_m+\sqrt{R_m}} = a_{R_{m+1}}  \le c_3  a_{R_m}^2 + c_3 e^{-c_4  (\log R_m)^{-4} e^{c_4 R_m} } \le c_5 a_{R_m}^2 + e^{- R_m }   \]
for constants $c_3,c_4,c_5 > 0$, valid if $R_m < \rho \sqrt{n}$, and if $\rho$ and $R_0$ are chosen sufficiently small and large respectively. Moreover, as $n \to \infty$,
\[ a_{R_0}  =  \prob \big( f_n \in \textrm{AnnCross}(-t/2,R_0 / \sqrt{n} )  \big)  \to  \prob \big( h_{BF} \in  \textrm{AnnCross}_\infty(-t/2, R_0 ) \big)  ,   \]
which by Proposition \ref{p:limitfield} can be made arbitrary small by taking $R_0$ sufficiently large. By induction (see \cite[Lemma 6.3]{MV20})  we deduce that
\[    \prob \big( f_n \in \textrm{AnnCross}(-t,R_m / \sqrt{n} )  \big)  \le a_{R_m} \le e^{-c_6 R_m }  \]
 for some $c_6 > 0$ and all $R_m < \rho \sqrt{n}$. Modifying the final step of the induction by replacing $R_{m+1}$ with $R \in [R_{m+1},R_{m+2}]$ extends this bound (with adjusted constant) to all $R < \rho \sqrt{n}$, which gives
 \begin{equation*}
 \prob(\textrm{AnnCross}(-t,r)) \le c_{1}e^{-c_{2}r\sqrt{n}} .
 \end{equation*}

\vspace{2mm}

The proofs of the other two statements follow along similar lines, but we highlight the important differences in the analysis. We begin with the proof of the third statement, since it is needed as input in the proof of the stronger second statement in the monochromatic case $\alpha = 1$. Let $\delta' > 0$ be given, let $F_\ell \in \{T_\ell,g_\ell \}$, fix constants $\rho < 0$ and $R_0 > 1$ to be chosen later, and define the sequence $R_{m+1} = 3 R_m $ and $t_{m+1} = t_m +  (\log R_m)^{-(1+\delta'/2)} $, with $t_0 = t/2$. Again, we choose $R_0$ sufficiently large so that $t_\infty = \lim_{m \to \infty} t_m =  t_0 + \sum_{m \ge 0} (\log R_m )^{-(1+\delta'/2)}  <  t $. Define
\[ a_{R_m} = \prob \big( F_\ell \in \textrm{AnnCross}(-t_m,R_m /\ell ) \big) . \]
By \eqref{e:renorm} (with $R = R_{m+1} / \ell$ and $r = R_m / \ell$), Proposition \ref{p:gsd} (with $r = c_2 (R_{m+1} - 2R_m) / \ell  = c_2 R_m / \ell$, $u = c_2 R_m / \ell$,  and the corresponding pair of events \eqref{e:rotated AnnCross})), and the uniform bounds in Lemmas \ref{l:ubrsh} and \ref{l:ubble}, we derive the renormalisation equation
\[ a_{3R_m} = a_{R_{m+1}}  \le  a_{R_m}^2 + c_7   R_m^{-1/2} (\log R_m)^{2+\delta'}   \]
where $c_7 > 0$ is uniform over $F_\ell \in \{T_\ell,g_\ell \}$, valid if $R_m < \rho \ell$, and if $\rho$ and $R_0$ are chosen sufficiently small and large respectively. Moreover, by Proposition \ref{p:limitfield}, the number $a_{R_0}$ can be made sufficiently small if $R_0$ and $\ell$ are taken sufficiently large (uniformly over $F_\ell \in \{T_\ell,g_\ell \}$). By induction (see the proof of \cite[Theorem 3.7]{m23}) we may deduce from this that
\[  \prob \big(  \textrm{AnnCross}(-t,R_m /\ell)  \big)  \le a_{R_m}  \le  c_8 R_m^{-1/2}  (\log R_m)^{2+\delta'} \]
for $R_m \le c_8 \ell$, where $c_8 > 0$ is uniform over $F_\ell \in \{T_\ell,g_\ell\}$. Again we extend this bound (with adjusted constant) to all $R < \rho \ell$ by modifying the final step of the induction, thus yielding the third statement.

\vspace{2mm}
Finally we prove the second statement. Again let $\delta' > 0$ be given, fix constants $\rho < 0$ and $R_0 > 1$ to be chosen later, and define the sequence $R_{m+1} = 2 R_m + R_m^{1-\delta'/2}$ and $t_{m+1} = t_m + (\log R_m)^{-2}$, with $t_0 = t/2$. Again we choose $R_0$ sufficiently large so that  $t_\infty = \lim_{m \to \infty} t_m  <  t $. Define
\[ a_{R_m} = \prob \left(  g_\ell \in \textrm{AnnCross}(-t_m, R_m / \ell^p   \big)\right) . \]
By \eqref{e:renorm} (with $R = R_{m+1} / \ell^p$ and $r = R_m / \ell^p$), and the first statement of Proposition \ref{p:sdble} (with $r = c_2 (R_{m+1} - 2R_m) / \ell^p = c_2 R_m^{1-\delta'/2} / \ell^p$ and $u = c_2 R_m  / \ell^p$, and the corresponding pair of events \eqref{e:rotated AnnCross}), we derive a renormalisation equation
\[ a_{2R_m+R_m^{1-\delta'/2}} = a_{R_{m+1}}  \le c_9  a_{R_m}^2 + c_9 e^{-c_{10}   (\log R_m)^{4} R_m^{1-\delta'/2} } \le  c_9  a_{R_m}^2 + c_{11} e^{-  R_m^{1-\delta;} }  \]
for constants $c_9,c_{10} ,c_{11}> 0$, valid if $R_m < \rho \ell$, and if $\rho$ and $R_0$ are chosen sufficiently small and large respectively. Moreover, if $\alpha \in [0,1)$, then as above, by Proposition \ref{p:limitfield}, the number $a_{R_0}$ can be taken
sufficiently small if $R_0$ and $\ell$ are taken sufficiently large. On the other hand, if $\alpha = 1$ then by the above proved third statement of Proposition \ref{p:ce}, one has
\[ a_{R_0}   =    \prob \big( g_\ell \in  \textrm{AnnCross}(-t/2,R_0 \ell^{1-p} / \ell )  \big)   \le c_8 (R_0 \ell^{1-p})^{-1/2} \log( R_0 \ell^{1-p})  ,  \]
which can also be made arbitrary small by taking $\ell$ large. Hence by induction (one can use a trivial modification of \cite[Lemma 6.3]{MV20}) we deduce from this that
\[  \prob \big( g_\ell \in  \textrm{AnnCross}(-t,R_m /\ell)  \big)  \le a_{R_m} \le e^{- c_{12} R_m^{1-\delta} }  \]
 for some $c_{12}> 0$ and all $R_m < \rho \ell$, and we extend this to all $R < R_0 \ell$ as in previous cases, thus concluding the proof of the second statement. Proposition \ref{p:ce} is now proved.
 \end{proof}

\subsection{Further consequences}
We now state further consequences of Proposition \ref{p:ce} for arm events and the non-existence of giant components in the subcritical regime, which follow from the standard arguments in percolation theory.

\subsubsection{Arm events}
 Recall the (planar) arm event $\textrm{Arm}_\infty(t,r)$ as in \eqref{eq:Arm event def}, and by analogy, for $r \in [0,\pi]$, let $\textrm{Arm}(t,r)$ be the event that the north pole $\eta$ is joined to $\partial \Dc_r$ in $\Uc(t) \subseteq \Sc^2$. Define also the (planar) \textit{truncated arm event}
\[ \textrm{TruncArm}_\infty(t,r) = \textrm{Arm}_\infty(t,r)  \cap \{ 0 \text{ is in a bounded component of } \{h \le t\} \} . \]
Recall also the local scaling limits $h \in \{h_{BF}, (h_\alpha)_{\alpha \in [0,1]}\}$ of the spherical ensembles from Example \ref{e:lsl}.

\begin{proposition}
\label{p:ae}
Let $t',\delta' > 0$ be given. The assertions of Proposition \ref{p:lu} hold with $e_{r,\delta;t}$ replaced by $\prob(\textrm{Arm}(-t,r))$ (with constants $c_1,c_2$ independent of $\delta$). In addition:

\begin{enumerate}[i.]
\item There exist $c_1,c_2 > 0$ such that, for all $t \ge t'$ and $r \ge 1$,
\[ \prob \big( h_{BF} \in \textrm{AnnCross}_\infty(-t,r) \cup \textrm{Arm}_\infty(-t,r) \cup \textrm{TruncArm}_\infty(t,r) \big)  \le c_1 e^{- c_2  r  }  .  \]

\item For every $\alpha \in [0,1)$ there exists a constant $c_1 > 0$ such that, for all  $t \ge t'$ and $r \ge 1$,
\[  \prob \big( h_\alpha \in \textrm{AnnCross}_\infty(-t,r) \cup \textrm{Arm}_\infty(-t,r) \cup \textrm{TruncArm}_\infty(t,r) \big)    \le  e^{-  c_1 r^{1-\delta'} }   .\]

\item There exists a constant $c_1 > 0$ such that, for all $\alpha \in [0,1]$, $t \ge t'$ and $r \ge 2$,
\[  \prob \big( h_\alpha \in \textrm{AnnCross}_\infty(-t,r) \cup \textrm{Arm}_\infty(-t,r) \cup \textrm{TruncArm}_\infty(t,r) \big)    \le    c_1 r^{-1/2}  (\log r)^{2+\delta'} .\]
\end{enumerate}

\end{proposition}

\begin{proof}
The first statement follows immediately from Proposition \ref{p:ce} since there exist constants $c, \rho> 0$ such that $\textrm{Arm}(-t,r)$ implies $\textrm{AnnCross}(-t,c r)$ for all $0 < r < \rho$. Turning to the second statement, we bound the probability of the three events separately. For $\textrm{AnnCross}_\infty(-t,r)$ (resp.\ $ \textrm{Arm}_\infty(-t,r) $), we deduce the result from Proposition \ref{p:lu} (resp.\ the first statement of this proposition) by the convergence to the local scaling limit and by Proposition \ref{p:stab}. It remains to prove the result for $\textrm{TruncArm}_\infty(t,r)$. Define the \textit{annulus circuit} event $\textrm{AnnCirc}_\infty(t,r)$ that $\{h \le t\}$ contains a loop inside $B_{2r} \setminus B_r$ that encloses $B_r$. Note that if $\textrm{AnnCirc}_\infty(t,r/2)$ occurs and $B_{r/2}$ is connected to $\infty$ in $\{ h \le t\}$, then  $\textrm{TruncArm}_\infty(t,r)$ cannot occur, see Figure \ref{f:glue} (left). On the other hand one can find translated copies of the events $\textrm{AnnCirc}_\infty(t,  2^k r/2) $, $k \ge 0$, such that if all of them occur then $B_{r/2}$ is connected to $\infty$ in $\{h \le t\}$, see Figure \ref{f:glue} (right). We conclude that
\begin{align*}
  \prob(  \textrm{TruncArm}_\infty(t,r) ) & \le \prob( \textrm{AnnCirc}^c_\infty(t,  r/2)^c )   +  \sum_{k \ge 0} \prob( \textrm{AnnCirc}^c_\infty(t,  2^k r / 2)^c )  \\
  & = \prob( \textrm{AnnCross}_\infty(-t,  r / 2) )  + \sum_{k \ge 0} \prob( \textrm{AnnCross}_\infty(-t,  2^k r / 2) )  .
  \end{align*}
The result then follows from the just established bounds on \[ \prob( \textrm{AnnCross}_\infty(-t,  r)),\]
which are summable over a dyadic sequence.
\end{proof}

\begin{figure}
\begin{tikzpicture}
\draw (0,0) circle (20pt);
\draw (0,0) circle (40pt);
     \node[above] at (0.2,-0.6) {\footnotesize{$B_{r/2}$}};
          \node[above] at (1.4,-1.3) {\footnotesize{$B_{r}$}};
              \node[above] at (1.9,1) {\footnotesize{$\infty$}};
            \draw[scale=2.7,shift={(0,0)},line width=0.8pt] (0,-0.34) .. controls (0.45,-0.2) .. (0.3,0.2)
                               .. controls (0.1,0.6) and (-0.9,0) .. (0,-0.34);
      \draw[line width=0.8pt] (0.5,0.5) .. controls (1,0.5) and (1.5,1.4) .. (2,1.6);
        \draw[scale=0.8,shift={(-0.1,-1.65)},line width=0.8pt,dashed] (0.1,1.65) .. controls (0.3,2.1) and (0.2,2.2) .. (0.2,2.5)
                               .. controls (0.2,2.8) and (0.1,3.2) .. (0.3,3.37);
\end{tikzpicture}
\hspace{1cm}
\begin{tikzpicture}
\draw (0,0) circle (20pt);
\draw (1.04,0) circle (20pt);
\draw (1.04,0) circle (40pt);
                                           \draw[scale=2.9,shift={(0.4,0)},line width=0.8pt] (0,-0.34) .. controls (0.45,-0.2) .. (0.3,0.2)
                               .. controls (0.1,0.6) and (-0.9,0) .. (0,-0.34);
              \draw[scale=5.6,shift={(0.6,0)},line width=0.8pt] (-0.1,-0.34) .. controls (-0.45,-0.2) and (-0.6,0.1)
                                .. (0,0.34);
                           \begin{scope}[shift={(3.1,0)},scale=1.35]
                                  \draw [domain=90:270] plot ({cos(\x)}, {sin(\x)});
                                 \end{scope}
                        \begin{scope}[shift={(3.4,0)},scale=2.7]
                                  \draw [domain=130:230] plot ({cos(\x)}, {sin(\x)});
                                 \end{scope}
                                     \node[above] at (-0.7,-1) {\footnotesize{$B_{r/2}$}};
\end{tikzpicture}
\caption{Left: The intersection of $\textrm{AnnCirc}_\infty(t,r/2)$ and the event $\{B_{r/2}$ is connected to $\infty$ in $\{ h \le t\} \}$ prohibits $\textrm{TruncArm}_\infty(t,r)$. Right: The intersection of translated copies of $ \textrm{AnnCirc}_\infty(t,  2^k r/2)$, $k \ge 0$ (with $k\in\{0,1\}$ depicted), implies the event $\{B_{r/2}$ is connected to $\infty$ in $\{h \le t\}\}$.}
\label{f:glue}
\end{figure}
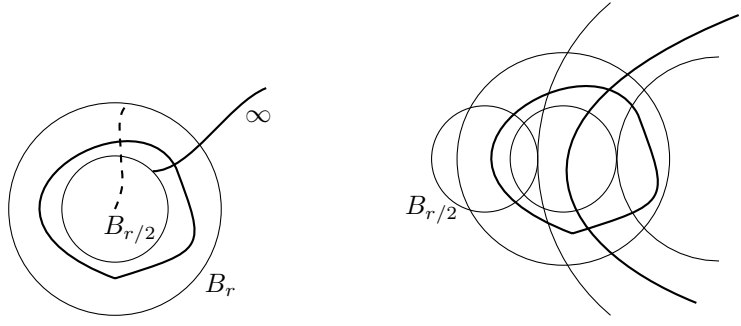

\begin{remark}
The fact that $ \prob(h_\alpha \in \textrm{TruncArm}_\infty(t,r))  \to 0$ as $r \to \infty$ \textit{uniformly over} $\alpha \in [0,1]$ will be important in showing the continuity of $\alpha \mapsto \varphi(t,\alpha)$ in Proposition \ref{prop:phi}.
\end{remark}

\subsubsection{Non-existence of subcritical giant components}

We also derive the bounds on the non-existence of giant components in the subcritical regime. For the random spherical harmonics this establishes the corresponding (third) statement \eqref{eq:spher harm subcrit} of Theorem \ref{thm:spher harm}. However for Kostlan's ensemble and the band-limited ensembles, these are weaker than the corresponding statements in Theorems \ref{thm:unique giant Kostlan} and \ref{thm:unique giant band-lim} respectively.

\begin{proposition}
\label{p:roughgiant}
 Let $t < 0$ and $\eps ,\delta> 0$ be given. Then there exist $c_1,c_2 > 0$, depending on the ensemble, such that
\[ \prob \big( \textrm{Area}( \Vc^a(t) ) \ge \eps \big)  \le \begin{cases} c_1 e^{- c_2  \sqrt{n}  }  &  \textit{ for Kostlan's ensemble,} \\    c_1 e^{-c_2  \ell^{p-\delta} } & \text{ for the band-limited ensembles,}   \\     c_1  \ell^{-1/2}  (\log \ell)^{2+\delta}  & \text{ for the random spherical harmonics.}
\end{cases}\]
\end{proposition}
\begin{proof}
There exists an absolute constant $c_3 > 0$ such that if $\textrm{Area}(\Vc^a(t) ) > \eps$ then there must exist a path in $\Vc^a(t)$ of length at least $c \sqrt{\eps}$. By the same construction to the one in the proof of Proposition \ref{p:lu} (see Figure \ref{f:lu}), one can find a collection of at most $c_4 = c_4(\eps)$ rotated copies of $\textrm{AnnCross}( t,\sqrt{\eps} / c_4)$ such that the existence of such a path implies that one of these events occurs. The result then follows from Proposition \ref{p:ce} and the union bound.
\end{proof}

\medskip
\section{Proof of the main results: Kostlan's and band-limited ensembles}
\label{s:mr}
In this section we prove Theorems \ref{thm:unique giant Kostlan}-\ref{thm:unique giant Kostlan2} on Kostlan's ensemble, and Theorems  \ref{thm:unique giant band-lim}-\ref{thm:unique giant bl2} on the band-limited ensembles. First we prove the qualitative concentration of the volume of the giant for spherical ensembles which converge locally; our argument relies only on the rotational invariance and the ergodicity of the local scaling limit. Then in the case of Kostlan's ensemble and band-limited ensembles, we use the renormalisation scheme described in section \ref{sec:proof outline} to upgrade this to quantitative estimates. A separate argument for the random spherical harmonics is given in the following section. We will work mainly with the volumetric giant $\Vc^a(t)$ since we wish to exploit its monotonicity.

\subsection{Qualitative concentration}
Let $F_n$ be a sequence of smooth isotropic Gaussian fields on $\Sc^2$ converging locally to $F_\infty$ at the scale $s_n \to 0$ in the sense of Definition \ref{d:lsl}. Fix a level $t \in \R$. Recall that $\Dc_r$ and $\Sc_r$ denote the spherical cap of radius $r$ and a `square' of side-length $2r$, both centred at the north pole, respectively. Recall also that if $U \subseteq \Sc^2$ is compact, then $\Vc^d(U)$ and $\Vc^a(U)$ denote the components of $\Uc(t) \cap U$ of largest diameter and area respectively. Finally recall the (planar) arm event $\textrm{Arm}_\infty(r) = \textrm{Arm}_\infty(t,r)$ for the Euclidean field $F_\infty$ from section \ref{ss:stab}, let its probability be denoted
\[\vartheta_r(t) = \prob(\textrm{Arm}_\infty(t,r)),\]
and define
\begin{equation}
\label{e:vartheta}
 \vartheta(t) := \lim_{r \to \infty} \vartheta_r(t) = \prob \left(0 \text{ is contained in an unbounded component of $\{F_\infty \le t\}$} \right) .
\end{equation}

We use an ergodic argument inspired by \cite{NaSoGen} to show that $\textrm{Area}(\Vc^a(U))$ concentrates near $\vartheta(t)$ if $U$ is a spherical cap or `square' at the scales $\gg s_n$:

\begin{proposition}[Qualitative concentration]
\label{p:qualcon}
Assume that the local limit $F_\infty$ is ergodic.
\begin{enumerate}[i.]

\item
For every $t \in \R$ and $\eps > 0$ there exists $r_0, n_0> 0$ such that, for  all $n \ge n_0$ and
$U \in  \bigcup\limits_{ r \ge r_0 s_n} \{\Sc_r, \Dc_r\}$,
\[ \prob \big(  \area(\Vc^a(U)  ) / \area(U) \ge \vartheta(t)  + \eps  \big) \le \eps .\]

\item Suppose that, in addition, $t \in \R$ is such that there exists a monotone decreasing function $g(r)$ satisfying
\begin{equation}
\label{e:econd}
e_{r s_n} \le g(r)  \quad \text{and} \quad \sum_{k \ge 1}{g(2^k)} < \infty ,
\end{equation}
where $e_r =e_{r,1/100;t}$ is defined as in \eqref{d:e}. Then for every $\eps > 0$ there exists $r_0, n_0> 0$ such that, for $n \ge n_0$ and $U \in  \bigcup\limits_{ r \ge r_0 s_n} \{\Sc_r, \Dc_r\}$,
\[ \prob \big(  \area(\Vc^a(U)  ) / \area(U) \le \vartheta(t)  - \eps  \big) \le \eps .\]
\end{enumerate}
\end{proposition}

Before proving Proposition \ref{p:qualcon} we make a preliminary observation on the `local averages' of arm events. For a spherical cap $\Dc = \Dc_r(x)$, let $\textrm{Arm}(\Dc) = \textrm{Arm}(t,\Dc)$ be the rotation of $\textrm{Arm}(t,r)$ such that $\Dc_r$ is mapped to $\Dc_r(x)$ (i.e.\ the event that $x$ is connected to $\partial \Dc_r(x)$ in $\Uc(t)$).

\begin{proposition}
\label{p:qcprelim}
Let $P \in \Sc^2$ be a reference point, and let $(F_n)_{n \ge 1}$ and $F_\infty$ be the coupling in Proposition \ref{p:lsl}.

\begin{enumerate}[i.]
\item For every $0 < r_1 < r_2$, as $n \to \infty$
\[  \frac{1}{\area(\Dc_{r_2 s_n})}  \int_{x \in \Dc_{r_2 s_n}(P)}  \id_{ \textrm{Arm}(\Dc_{r_1 s_n }(x)  )}  \, dx \to   \frac{1}{\area(B_{r_2})}  \int_{x \in B_{r_2}}  \id_{ x + \textrm{Arm}_\infty(r_1) }  \, dx \]
in probability.

\item  For every $r_1 > 0$, as $r_2 \to \infty$,
\[ \frac{1}{\area(B_{r_2})}  \int_{x \in B_{r_2}}  \id_{ x + \textrm{Arm}_\infty(r_1) }  \, dx \to  \vartheta_{r_1}(t)  \]
almost surely and in $L^1$.
\end{enumerate}
\end{proposition}

The following is a direct consequence of Proposition \ref{p:qcprelim}:

\begin{corollary}
\label{c:qcprelim}
For every $y \in \Sc^2$, $t \in \R$, and $\eps > 0$, then for sufficiently large $r_1 > 0$, sufficiently large $r_2 > 0$ depending on $r_1$, and sufficiently large $n \ge 1$ depending on $r_1$ and $r_2$,
\[  \E \Big[\Big| \frac{1}{\area(\Dc_{r_2 s_n})}  \int_{z \in \Dc_{r_2 s_n}(y)}  \id_{ \textrm{Arm}(t; \Dc_{r_1 s_n }(z)  )} \, dz -  \vartheta(t)    \Big|\Big] < \eps. \]
\end{corollary}
\begin{proof}[Proof of Proposition \ref{p:qcprelim}]
(i.) By the triangle inequality it suffices to prove that, for every $0 < r_1 < r_2$, as $n \to \infty$
\begin{equation}
\label{e:ti} \sup_{x \in B_{r_2}}   \prob \big( \id_{ \textrm{Arm}(\Dc_{r_1 s_n }( \exp_P(x s_n) )  )}  \bigtriangleup \id_{ x + \textrm{Arm}_\infty(r_1) }  \big) \to 0 .
\end{equation}
Fix $\delta > 0$. By the first statement of Proposition \ref{p:lsl} and the monotonicity of the arm event, there exists $\eps > 0$ such that, for sufficiently large $n$, uniformly over $x \in B_{r_2}$,
\begin{equation}
\label{eq:Sandwich Euclidean Riemannian}
 \big\{ x + \textrm{Arm}_{\infty}\left(t-\epsilon, r_1+\epsilon \right) \big\} \subseteq \textrm{Arm}(t,\Dc_{r_{1}s_{n}}\exp_P( x s_{n})) \subseteq \big\{ x \in \textrm{Arm}_{\infty} \left(t+\epsilon, r_1 - \eps \right) \big\},
\end{equation}
outside an event of probability $<\delta$. Moreover, by the stability of Proposition \ref{p:stab} and the stationarity of $F_{\infty}$,
\begin{equation*}
\lim\limits_{\eps \rightarrow 0}\prob\left( x + \textrm{Arm}_{\infty} \left(t+\eps, r_1 - \eps)\right)\right) = \prob\left(\textrm{Arm}_{\infty}(t, r_1)\right),
\end{equation*}
and similarly
\begin{equation*}
\lim\limits_{\eps \rightarrow 0}\prob\left( x + \textrm{Arm}_{\infty} \left(t-\eps, r_1 + \eps)\right)\right) = \prob\left(\textrm{Arm}_{\infty}(t, r_1)\right).
\end{equation*}
Hence taking first $\eps \to 0$ in \eqref{eq:Sandwich Euclidean Riemannian}, and then $\delta \to 0$, gives \eqref{e:ti}.

(ii.) This is a consequence of the ergodicity of $F_\infty$ and Weiner's multi-dimensional version of Birkhoff's  ergodic theorem \cite[Theorem II'']{wei39}.
\end{proof}

We are now ready to prove Proposition \ref{p:qualcon}:

\begin{proof}
Let $t,\eps,\delta > 0$ be given, and fix $ r_1 > 0$ sufficiently large, $r_2$ sufficiently large (depending on $r_1$), and $r_0$ sufficiently large (depending on $r_2$). Then for every $U \in  \bigcup\limits_{r \ge r_0 s_n} \{\Sc_r, \Dc_r\}$ we have
\begin{equation}
\label{eq:AreaV/U triangle}
\begin{split}
&  \frac{\area(\Vc^a(U)) }{ \area(U) } -  \vartheta(t) =  \frac{1}{\area(U)}   \int\limits_{y \in U }  \left(\id_{ \textrm{Arm}( \Dc_{r_1 s_n}(y) )} -  \vartheta(t) \right)\, dy    \\
& \qquad \qquad +   \frac{1}{\area(U)}   \int\limits_{y \in U}\left( \id_{y \in \Vc^a(U)} -  \id_{ \textrm{Arm}(\Dc_{r_1 s_n}(y) )} \right)\, dy
\\ & \quad =  \frac{1}{\area(U)}   \int\limits_{y \in U}  \left(  \frac{1}{\area(\Dc_{r_2 s_n})}  \int\limits_{z \in \Dc_{r_2 s_n}(y)}  \left(\id_{ \textrm{Arm}(\Dc_{r_1 s_n }(z)  )}-  \vartheta(t)\right) \, dz \right) dy   +\\
& \qquad \qquad   +
 \frac{1}{\area(U)}   \int\limits_{y \in U } \left(\id_{y \in \Vc^a(U )} -  \id_{ \textrm{Arm}( \Dc_{r_1 s_n}(y) )} \right)\, dy   +  E,
\end{split}
\end{equation}
where
\[  |E|  \le  \frac{  \area( \{x : d_{\Sc^2}(x, \partial U) \le  r_2s_n \} )  }{ \area(U) }  . \]

If $r_2 > 0$ is fixed, then we can take $r_0$ and $n_0$ sufficiently large so that $|E| < \eps/4$ for all $n \ge n_0$. Moreover, by Corollary \ref{c:qcprelim} we can set $r_1$ and $r_2$ sufficiently large, and then take $n_0$ sufficiently large, to guarantee  that
\[  \E \left[\left| \frac{1}{\area(\Dc_{r_2 s_n})}  \int\limits_{z \in \Dc_{r_2 s_n}(y)} \left( \id_{\textrm{Arm}(\Dc_{r_1 s_n }(z)  )}-  \vartheta(t) \right)\, dz  \right|\right] < \eps^2/8 . \]
Averaging over $y \in U$, and using Markov's inequality, we therefore have that the first term on the r.h.s.\ of \eqref{eq:AreaV/U triangle} satisfies
\[ \prob \left(  \left| \frac{1}{\area(U) }   \int\limits_{y \in U}  \left(  \frac{1}{\area(\Dc_{r_2 s_n})}  \int\limits_{z \in \Dc_{r_2 s_n}(y)}  \left(\id_{ \textrm{Arm}( \Dc_{r_1 s_n }(z)  )}-  \vartheta(t)\right) \, dz \right) dy  \right|  > \eps/4\right) < \eps/2. \]

To conclude the proof of the first statement of the proposition it remains to notice that if there exists $y \in \Vc^a(U)$ such that $\textrm{Arm}( \Dc_{r_1 s_n}(y) )$ does not occur then $  \area(\Vc^a(U)  ) \le  \area(\Dc_{r_1 s_n}  ) $, which is less than $ \eps\cdot  \area(U)$ if $r_0$ is set sufficiently large depending on $r_1$.  Hence on the event $ \area(\Vc^a(U)  ) / \area(U) \ge \vartheta(t)  + \eps $,  the second term on the r.h.s.\ of  \eqref{eq:AreaV/U triangle}  satisfies
\[ \frac{1}{\area(U)}   \int\limits_{y \in U } \left(\id_{y \in \Vc^a(U)}  - \id_{\textrm{Arm}( \Dc_{r_1 s_n}(y) )} \right) \, dy = 0 ,\]
concluding the proof.

\vspace{2mm}
To prove the second statement of Proposition \ref{p:qualcon}, since $\area(\Vc^d(U)) \le \area(\Vc^a(U))$ it is sufficient to prove the result with $\Vc^d(U)$ in place of $\Vc^a(U)$. Arguing as in the proof of the first statement, it remains to show that
 \[  \prob \left( \frac{1}{\area(U)}   \int\limits_{y \in U } \left(\id_{\textrm{Arm}( \Dc_{r_1 s_n}(y) )} - \id_{y \in \Vc^d(U)}  \right)\, dy   > \eps/4 \right) < \eps/2 , \]
 which is a consequence of
 \begin{equation}
 \label{e:ergodic2}
 \frac{1}{\area(U)}  \int\limits_{y \in U} \prob \left(  \textrm{Arm}( \Dc_{r_1 s_n}(y) ) \setminus \{y \in \Vc^d(U) \} \right) \, dy
 < \eps^2 / 8
  \end{equation}
  via Markov's inequality. Recall the local existence and uniqueness event $\widetilde{\textrm{EU}}_r(x) =  \widetilde{\textrm{EU}}_{r,1/100}(x) $. Suppose that $\textrm{Arm}(\Dc_{r_1 s_n}(y) ) $ and $\widetilde{\textrm{EU}}_{{r_1 s_n}}(y) $ both hold. Then $y$ is contained in $\Vc^d(\Dc_{r_1 s_n}(y) ) $. Moreover, applying the third claim of Lemma \ref{l:lu}, $\Vc^d(\Dc_{r_1 s_n}(y) )$ is contained in  $\Vc^d( U )$ outside an event of probability at most
  \[  A_y := \min \Big\{1, \frac{cr}{ d_{\Sc^2}(y, \partial U) } \sum_{k = 0}^{\lfloor \log_2(\pi/r) \rfloor} e_{2^k r} \Big\} \le \min \Big\{1, \frac{c  h(r_0) r }{ d_{\Sc^2}(y, \partial U) } \Big\} , \]
  where, using the assumption \eqref{e:econd}, $h(r) := 2   \sum_{k \ge 0} g(r 2^k)$ is non-increasing and satisfies $h(r) \to 0$ as $r \to \infty$.

We deduce that
   \begin{align}
 \label{eq:prob(Arm- V) bnd}
\nonumber \prob \left(  \textrm{Arm}(\Dc_{r_1 s_n}(y) ) \setminus \{y \in \Vc^d(U ) \} \right) & \le \big(1 -  \prob(  \widetilde{\textrm{EU}}_{{r_1 s_n}}(y)  \big) + A_y    \\
 &  \le c_2 \Big( h(r_1) +  \min\{1,  h(r_0)   / (d_{\Sc^2}(y, \partial U)) \} \Big) .
 \end{align}
 Integrating both sides of the inequality \eqref{eq:prob(Arm- V) bnd} in polar coordinates, the l.h.s.\ of \eqref{e:ergodic2} is at most a constant times
  \[   h(r_1) +  \frac{1}{r^2}   \Big( \int\limits_{r-rh(r_0)}^{r}  s \, ds  + \int\limits_{rh(r_0)}^{r}   s \times r h(r_0)  / (r-s)  ds   \Big)   \le h(r_1) + h(r_0) + h(r_0) \log (1/h(r_0) )  ,\]
and so, taking $r_1$ and $r_0$ sufficiently large, the inequality \eqref{e:ergodic2} follows. As mentioned above, this implies the second statement of Proposition \ref{p:qualcon}.
\end{proof}

\subsection{The renormalisation scheme}
In this section we set up the renormalisation scheme described in \S \ref{sec:proof outline}.

\begin{definition}[$(u,\eps)$-tiling]
\label{def:tiling}
Two rotated copies $S^i$ and $S^j$ of $\Sc_u$ are called \textit{$u$-connected} if $d_{\Sc^2}(S^i,S^j) \le u$; a collection of $S^i$ is \textit{$u$-connected} if it is connected as a graph under this pairwise relation. For $0 < u < v < \pi/2$ and $\eps \in (0,1)$, a \textit{$(u,\eps)$-tiling} of $\Dc_v$ is a collection $\{S^i\}_{i \le n}$ of rotated copies of $\Sc_u$, called \textit{tiles}, satisfying the following properties:
\begin{enumerate}[i.]
\item (Small overlap)
\[ \Dc_v \subseteq \bigcup\limits_{i} S^i  \ , \quad \frac{\area(S^i \setminus \cup_{j \neq i} S^j ) }{\area(\Sc_u)} \ge 1 - \eps \quad \text{and} \quad   n \le \frac{\area(\Dc_v)(1+\eps)}{\area(\Sc_u) }  .\]
\item (Local boundedness) For every tile $S^i$, there are at most $8$ distinct tiles $S^j$ that are $u$-connected to $S^i$.
\item (Isoperimetry) For every integer $s \ge 1$, every collection of at least $n -s$ tiles has a subset which is $u$-connected and contains at least $n - 4 s^2$ tiles.
\end{enumerate}
A $(u,\eps)$-tiling of $\Sc_v$ is defined analogously, with $\Dc_v$ replaced by $\Sc_v$ in the first item of Definition~\ref{def:tiling}.
\end{definition}

\begin{lemma}
\label{l:tiling}
Let $\eps > 0$. Then there exists $\rho > 0$ such that, if $0 < u <\rho$ and $v > u$, $(u,\eps)$-tilings of $\Dc_v$ and $\Sc_v$ exist.
\end{lemma}
\begin{proof}
Since $\Sc^2$ has locally Euclidean geometry, one can construct such a tiling by mapping of the standard tiling of the plane $([-u,u]^2 + z)_{z \in \Z^2}$ by the exponential map (see, e.g., \cite[(2.1)]{DP96} for the analogue of the isoperimetric property for $\Z^2$).
\end{proof}

Recall the local uniqueness event $\widetilde{\textrm{EU}}_{r,\delta}(x)$ of Definition \ref{def:EU tilde def}. For $S^i =\Sc_r(x)$ we will
abbreviate $\widetilde{\textrm{EU}}_\delta(S^i) = \widetilde{\textrm{EU}}_{r,\delta}(x)$.

\begin{lemma}[Upper bound]
\label{l:ub1}
There exists $\rho > 0$ such that the following holds. Let $\eps \in (0,1/4)$, $s \in (0,1)$, and $\delta \in (0, \min\{s/6,1/100\})$. For $v > u > 0$ and $u < \rho$, let $U \in  \{ \Dc_v,\Sc_v \}$, and assume that $\{S^i\}$ is a $(u,\eps)$-tiling of $U$. Suppose that
\begin{equation}
\label{eq:assmp V(U)>s U}
\area( \Vc^a(U ) ) > s\cdot \area(U) .
\end{equation}
Then at least one of the following holds:
\begin{itemize}
\item There exist $\ge (\delta/2)  \area(U)  / \area(\Sc_u) $ elements $S^i$ such that
\[ \area( \Vc^a(S^i) ) > (s-6\delta)(1-\eps) \area(\Sc_u) ; \]
\item There exist $\ge (\delta/2)  \area(U)  / \area(\Sc_u) $ elements $S^i$ such that $\widetilde{\textrm{EU}}_\delta(S^i)$ fails.
\end{itemize}
\end{lemma}
\begin{proof}
Suppose neither holds. Then there exists at least  $n -\delta  \area(U)  / \area(\Sc_u) $ elements $S^i$ such that $ \area( \Vc^a(S^i) ) \le  (s-6\delta)(1-\eps) \area(\Sc_u)$ and  $\widetilde{\textrm{EU}}_\delta(S^i)$ holds. We apply Lemma \ref{l:volbound} for such $S^i$ to deduce that
 \[ \area(  \Vc^a(U) \cap S^i ) \le 3 \delta \area(\Sc_u)  + \area(\Vc^a(S^i)) \le   (3\delta + (s-6\delta)(1-\eps)  ) \area(\Sc_u) . \]
 Hence
  \begin{align*}
    \area(  \Vc^a(U)) & \le \sum_{1 \le i \le n} \area(  \Vc^a(U) \cap S^i ) \\
    &  \le  \left(n -\delta  \frac{\area(U)}{\area( \Sc_u)} \right)  \cdot (3\delta + (s-6\delta)(1-\eps))  \area(\Sc_u) \\
    & \qquad  \qquad +  \delta  \frac{\area(U)}{\area(\Sc_u)}  \cdot  \area(\Sc_u)  \\
    & \le \big( (1 + \eps - \delta)(3 \delta +  (s-6\delta)(1-\eps)) + \delta \big) \area(U) ,
    \end{align*}
  where we used assumption (i) of the Definition \ref{def:tiling} of a $(u,\epsilon)$-tiling in the last inequality. Noticing that
\[     (1 + \eps - \delta)(3 \delta +  (s-6\delta)(1-\eps)) + \delta   < s  - \delta(2 - 6 \eps^2 - 3 \eps) < s ,\]
where the first inequality is elementary to check for $\delta< s/6$, and the second uses that $\eps < 1/4$, this gives the required contradiction to \eqref{eq:assmp V(U)>s U}.
\end{proof}

We now present an alternative upper bound that is useful in the subcritical regime in which the event $\widetilde{EU}$ is atypical.  For $\Sc$ a rotated copy of $\Sc_u$, we let $\textrm{AnnCross}(\Sc;t)$ be the rotation of $\textrm{AnnCross}(t,u)$ such that $\Sc_u$ is mapped to $\Sc$.

\begin{lemma}[Upper bound; subcritical regime]
\label{l:ub2}
Let $\eps > 0$ and $s \in (0,1)$. For $v > u > 0$, let $U \in \{\Dc_v,\Sc_v\}$, and assume that $\{S^i\}$ is a $(u,\eps)$-tiling of $U$. Suppose that
\[ \area( \Vc^a(U ) ) > s \area(U) .\]
Then at least one of the following holds:
\begin{itemize}
\item There exist $\ge s  \area(U)/\area(\Sc_u) $ elements $S^i$ such that $\textrm{AnnCross}(S^i)$ holds;
\item  $\area(\Vc^a(U)) \le \area(\Sc_{2u})$.
\end{itemize}
\end{lemma}
\begin{proof}
Suppose neither hold. Observe that if, for some $i$, the event $\textrm{AnnCross}(S^i)$ does not occur, then either $\Vc^a(U)$ is contained in a rotated copy of $\Sc_{2u}$ or $\Vc^a(U)$ does not intersect $S^i$. In the former case $\area(\Vc^a(U)) \le \area(\Sc_{2u})$ which is a contradiction to the second event not holding. In the latter case $ \area(  \Vc^a(U) \cap S^i ) = 0$, and hence
 \begin{align*}
    \area(  \Vc^a(U)) & \le \sum_{1 \le i \le n} \id_{\textrm{AnnCross}(S^i)}  \area(  \Vc^a(U) \cap S^i ) \\
    &  \le  s  \area(U)  / \area(\Sc_u)  \times   \area(\Sc_u)  = s  \area(U)  ,
    \end{align*}
    which is also a contradiction.
\end{proof}

\begin{lemma}[Lower bound]
\label{l:lb}
There exists $\rho > 0$ such that the following holds. Let $\eps \in (0,1)$, $s \in (0,1)$, $\delta \in (0,1/100)$. For $0 < u < v$ and $u < \rho$, let $U \in \{\Dc_v,\Sc_v\}$, and assume that $\{S^i\}$ is a $(u,\eps)$-tiling of $U$. Suppose that
\[ \area( \Vc^a(U ) ) < s \area(U) .\]
Then at least one of the following holds:
\begin{itemize}
\item There exist $\ge (\delta/2)  \area(U)  / \area(\Sc_u) $ elements $S^i$ such that
\[ \area( \Vc^a(S^i) ) < (s+\delta + \eps) \area(\Sc_u) ; \]
\item There exist $\ge (\delta/2)  \sqrt{ \area(U)  / \area(\Sc_u) }$ elements $S^i$ such that $\widetilde{\textrm{EU}}_\delta(S^i)$ fails.
\end{itemize}
\end{lemma}
\begin{proof}
Suppose neither holds. Then by the isoperimetric property (iii) of the Definition \ref{def:tiling} of a $(u,\epsilon)$-tiling, there exists a $u$-connected subset of $S^i$ with at
least $n  -  \delta^2 \area(U)  / \area(\Sc_u) $ elements such that $\widetilde{\textrm{EU}}_\delta(S^i)$ holds. Further, since assumption (i) of Definition \ref{def:tiling} of $(u,\epsilon)$-tiling implies, in particular, that  $n \ge \area(U)  / \area(\Sc_u) $ we may deduce that
at least
\[ n - ( \delta^2  + \delta/2)\area(U)  / \area(\Sc_u)  \ge  (1-  \delta^2  - \delta/2) \area(U)  / \area(\Sc_u) \]
elements of this subset have
\[ \area( \Vc^a(S^i) ) \ge  (s+\delta +\eps) \area(\Sc_u)  \implies \area( \Vc^a(S^i) \setminus \bigcup\limits_{j \neq i} S^j ) \ge  (s+\delta) \area(\Sc_u),\]
where we used assumption (i) of Definition \ref{def:tiling} again. Let us denote $N\subseteq \{1,\ldots,n\}$ to be the set of indexes $i\le n$ satisfying those properties. By Lemma \ref{l:smallgiant}, for each $i \in \N$ we have $\Vc^a(S^i) = \Vc^d(S^i)$, and by the second statement of Lemma \ref{l:inherit} each of these $\Vc^d(S^i)$ are contained within the same component of $\Uc \cap U$.
 Hence
  \begin{align*}
    \area(  \Vc^a(U) ) & \ge \sum_{i \in N} \area( \Vc^a(S^i) \setminus \bigcup\limits_{j \neq i} S^j  )  \\
    & \ge  (1 -   \delta^2 - \delta/2)(s+\delta) \area(U) .
    \end{align*}
Observing that $(1 - \delta^2 - \delta/2)(s+\delta) > s$ since $\delta < 1/100$ and $s < 1$, this gives the required contradiction.
\end{proof}

\subsection{Application to Kostlan's ensemble: Proof of Theorems \ref{thm:unique giant Kostlan}-\ref{thm:unique giant Kostlan2}}

Let $\vartheta(t)$ be the constant \eqref{e:vartheta} in Proposition \ref{p:qualcon}, defined for Kostlan's ensemble (i.e.\ with $F_\infty = h_{BF}$). Given the results in section \ref{s:lu}, to complete the proof of Theorems \ref{thm:unique giant Kostlan}-\ref{thm:unique giant Kostlan2} it suffices to establish the following bounds on the volume-concentration of the giant.

\begin{proposition}
\label{prop:upper/lower conc Kostlan}
$\,$
\begin{enumerate}[i.]
\item For every $t \in \R$ and $\varepsilon > 0$ there exists a constant $c>0$ such that
\begin{equation}
\label{e:ke1}
\prob\big( \area(\Vc^a(t) )> 4 \pi ( \vartheta(t) + \eps ) \big) \le e^{- cn}.
\end{equation}
\item For every $t > 0$  and $\varepsilon > 0$ there exists a constant $c>0$ such that
\begin{equation}
\label{e:ke2}
\prob\big( \area(\Vc^a(t) ) < 4 \pi ( \vartheta(t) - \eps )  \big) \le e^{- c \sqrt{n}} .
\end{equation}
\end{enumerate}
\end{proposition}

\begin{proof}[Proof of Theorems \ref{thm:unique giant Kostlan}-\ref{thm:unique giant Kostlan2} assuming Proposition \ref{prop:upper/lower conc Kostlan}]
Recalling that $\vartheta(t) = 0$ for $t \le 0$ by Proposition \ref{p:limitfield}, \eqref{e:ke1} implies Theorem \ref{thm:unique giant Kostlan2}. Theorem  \ref{thm:unique giant Kostlan} is proven by combining \eqref{e:ke1}, \eqref{e:ke2}, Lemma \ref{l:lu}, and the estimates on $e_r = e_{r,t}$ in Proposition \ref{p:lu}. In particular Proposition \ref{p:lu} gives that
\begin{equation}
\label{e:eboundke}
e_r \le c_1 e^{-(1/c_1) r \sqrt{n}}
\end{equation}
for a constant $c_1 > 0$, and so the uniqueness of the (diametric) giant component is a consequence of the first statement of Lemma \ref{l:lu}; as a consequence of the uniqueness, the bound \eqref{e:ke2} transfers to the diametric giant $\Vc^d(t)$. Moreover, setting $r = c_2 \log n  /\sqrt{n}$ for sufficiently large $c_2 >0$, \eqref{e:eboundke} implies that there are constants $c_3,c_4 > 0$ such that
\[  r^{-2} (\log 1/r) \sup_{r' \ge r} e_{r'}   \le c_3 n^{-c_4} . \]
Hence the ubiquity and local uniqueness of the giant component is a consequence of the second statement of Lemma \ref{l:lu}.
\end{proof}

We turn to the proof of Proposition \ref{prop:upper/lower conc Kostlan}. Let $t \in \R$ and $\eps > 0$ be given. In the proof $c > 0$ and $c_i$ will be constants which are independent of $n$, and the constant $c > 0$ may change from line to line. We will also assume that $n$ is taken sufficiently large, depending on $t$ and $\eps$, without explicitly mentioning this. We shall assume the continuity of $t \mapsto \vartheta(t)$ on $\R$, as stated in Proposition \ref{prop:theta} and formally proven in the following section.

\begin{proof}[Proof of Proposition \ref{prop:upper/lower conc Kostlan}(i)] \hfill

\textbf{Case $t>0$.}
We may assume that $\eps < \min\{1/100, 1-\vartheta(t) \}$ (if $\vartheta(t) = 1$ there is nothing to prove), and fix $\delta' \in (0,t)$ sufficiently small so that
 \[ (\vartheta(t) +3\eps/4)(1 - \eps/3) = \vartheta(t) + \eps ( 3/4 - \eps/4 -  \vartheta(t)/3) >  \vartheta(t) + 7\eps/24   > \vartheta(t+\delta') + \eps/4 ,\]
 which is possible by the continuity of $t \mapsto \vartheta(t)$, and where we used that $\eps < 1/2$ and $\vartheta(t) \le 1$. Then take $c_0 > 0$ sufficiently large so that
 \begin{equation}
 \label{e:c0choice}
 c_1 c_0^2 e^{-c_2 c_0^2 } \le \delta'
 \end{equation}
 where $c_1,c_2 > 0$ are as in Proposition \ref{p:sdk}.

Define $u = c_0   / \sqrt{n}$, $v = \pi$, and $U = \Dc_v = \Sc^2$. By Lemma \ref{l:tiling}, we may fix a $(u,\eps/3)$-tiling $\{S^i\}$ of $D$. Assume that $\area(\Vc^a(t)) > 4 \pi ( \vartheta(t) + \eps)$. Then by Lemma \ref{l:ub1} (with $\eps \mapsto \eps/3$, $s \mapsto  \vartheta(t) + \eps$, and $\delta \mapsto \eps/24$) at least one of the following events occurs: $E_{1}$ that there exist at least $(\eps/48) (4 \pi) / \area(\Sc_u)$ tiles $S^i$ such that
\begin{equation}
\label{e:ke3}
 \area(\Vc^a(S^i;t) ) >   ( \vartheta(t) + 3\eps/4)(1 - \eps/3) \area(\Sc_u)   >    (\vartheta(t+\delta') + \eps/4) ) \area(\Sc_u) ;
 \end{equation}
 or $E_{2}$ that there exist at least $(\eps/48) (4 \pi) / \area(\Sc_u)$ tiles $S^i$ such that $\widetilde{\textrm{EU}}_{\eps/24}(S^i;t)$ fails.

\vspace{2mm}

Cover $\Sc^2$ with a finite number $N$ of compact subsets $\Uc^j$, each of which is contained in a rotated copy of the open orthant $\Sc^2 \cap \mathcal{O}$, and such that each tile $S^i$ is contained in at least one $\Uc^j$. If $E_1$ occurs, then, by the pigeonhole principle, \eqref{e:ke3} holds for at least $((\eps/48) (4 \pi) / N) / \area(\Sc_u)$ tiles $S^i$ contained in some $\Uc^j$. By extracting a further subset, and by the local boundedness property of the tiling, and the rotational symmetry, we can assume that \eqref{e:ke3} holds for a subset of at least $c_3 / \area(\Sc_u)$ tiles $S^i$ which are contained in $\Uc^1 \subseteq \Sc^2 \cap \mathcal{O}$ and which are pairwise separated by distance at least $u$. Since there are at most  $c / \area(\Sc_u)$ total tiles, there are at most
\[  2^{c  / \area(\Sc_u) } \le   e^{c_4/ \area(\Sc_u)} \]
 possible choices for this subset.

Observe that the event $\{\area(\Vc^a(S^i;t)) >s\}$ is increasing in $t$. Then by the union bound and the second statement of Proposition \ref{p:sdk} (applied to $r = u$, and valid by \eqref{e:c0choice}), $\prob(E_1)$ may be bounded as
 \begin{align}
 \label{e:keproof1}
&\prob(E_1)\le      \Big( c^ {-1} \times \prob \big( \area(\Vc^a(\Sc_u;t+\delta') >  ( \vartheta(t+\delta') + \eps/4) \area(\Sc_u)  \big) \Big)^{c_3 / \area(\Sc_u) } \\
\nonumber & \qquad \qquad + e^{c_4/ \area(\Sc_u)}  \times e^{- (c_5  / \area(\Sc_u) )  e^{c_2 u^2 n} }   .
\end{align}
By the qualitative convergence in Proposition \ref{p:qualcon}, the first term on the r.h.s.\ of \eqref{e:keproof1} is at most
\[  e^{-c/ \area(\Sc_u)} \le e^{- c n}.\]
Taking $c_0 > 0$ sufficiently large, the second term on the r.h.s.\ of \eqref{e:keproof1} is at most
\[ e^{c_4/ \area(\Sc_u)}   e^{-  (c_5 / \area(\Sc_u) )  e^{c_2 u^2 n} }  \le   e^{c n ( c_4 / c_0^2  - c_5 /c_0^2 e^{c_2 c_0^2} ) } \le e^{-c n} . \]
Combining, we deduce that $\prob(E_1) \le e^{-c n}$.

One may argue similarly for $E_2$. More precisely, since $\delta' < t$, by Proposition \ref{p:lu} we can make $e_{c_0 / \sqrt{n};t-\delta'} $ arbitrarily small by taking $c_0$ sufficiently large. Since also $\widetilde{\textrm{EU}}_{\eps/24}(S_u)^c$ is decreasing in $t$, a similar argument shows that $\prob(E_2) \le e^{-c n }$. This concludes the proof of \eqref{e:ke1} if $t > 0$.

\vspace{2mm}

\textbf{Case $t=0$.}
The case $t = 0$ follows from the case $t > 0$. More precisely, since $\lim_{t \to 0} \vartheta(t) = 0$, we may choose $t'  > 0$ sufficiently small so that $\vartheta(t') < \eps/2$. Then since $\{\area(\Vc^a(t)) >s\}$ is increasing in $t$,
\[ \prob\big( \area(\Vc^a(0) ) > 4 \pi  \eps  \big) \le \prob\big( \area(\Vc^a(t') ) > 4 \pi (\vartheta(t') +  \eps/2 ) \big)  \le e^{- cn}  .\]

\vspace{2mm}

\textbf{Case $t<0$.}
In the case $t < 0$, we assume that $\area(\Vc^a(t) ) >  \eps$, and fix $\delta' > 0$ sufficiently small so that $t + \delta' < 0$. Then by Lemma \ref{l:ub2} either there exists $4 \pi \eps / \area(\Sc_u)$ tiles $S^i$ such that $\textrm{AnnCross}(S^i;t)$ holds or else $\area(\Vc^a(0)) \le \area(\Sc_{u})$. The latter case gives a contradiction. In the former case, since at level $t + \delta' < 0$, by Proposition \ref{p:ce}, $\prob(\textrm{AnnCross}(t+\delta',c_0/\sqrt{n})$ can be taken sufficiently small by setting $c_0$ sufficiently large, and since $\textrm{AnnCross}(t,u)$ is increasing in $t$, arguing as above the probability is at most  $e^{-c n }$. The proof of \eqref{e:ke1} is now complete.
\end{proof}

\begin{proof}[Proof of Proposition \ref{prop:upper/lower conc Kostlan}(ii)]
Compared to the proof of Proposition \ref{prop:upper/lower conc Kostlan}(i) we redefine $u =  n^{-1/4} \gg n^{-1/2}$, and set again $v = \pi$, $U = \Dc_v = \Sc^2$, and $\{S^i\}$ a $(u,\eps/3)$-tiling. We assume that $\eps < \min\{1/100, \vartheta(t) \}$ and that $\area(\Vc^a(t))< 4 \pi ( \vartheta(t) - \eps)$ occurs. Then by Lemma \ref{l:lb} (with $\eps \mapsto \eps/3$, $s \mapsto \vartheta(t) - \eps$, and $\delta = \eps/6$) at least one of the following events occurs: $E_{3}$ that there exists  $(\eps/12) (4 \pi) / \area(\Sc_u)$ tiles $S^i$ such that
\begin{equation}
\label{e:ke4}
 \area(\Vc^a(S^i;t) ) < ( \vartheta(t) -  \eps + \eps/6 + \eps/3 )\area(\Sc_u) =  ( \vartheta(t) -  \eps/2 )\area(\Sc_u)    ;
 \end{equation}
 or $E_{4}$ that there exists $(\eps/12) \sqrt{ (4 \pi)}/ \sqrt{ \area(\Sc_u)}$ tiles $S^i$ such that $\widetilde{\textrm{EU}}_\delta(S^i;t)$ fails.
  Using the same covering of $\Sc^2$ by compact subsets $\Uc^j$ as in the proof of Proposition \ref{prop:upper/lower conc Kostlan}(i), if $E_3$ occurs then by the pigeonhole principle and extracting a further subset, \eqref{e:ke4} holds for at least $c /  \area(\Sc_u)$ tiles $S^i$ contained in $U^i \subseteq \Sc^2 \cap \mathcal{O}$ and which are mutually separated by a distance at least $u$. By the continuity of $t \mapsto \vartheta(t)$, we can fix $\delta' > 0$ sufficiently small so that
 \[ \vartheta(t) - \eps/2 <  \vartheta(t-\delta') - \eps/4.  \]

 By the union bound and Proposition \ref{p:sdk}, $\prob(E_3)$ is at most
  \begin{align}
  \label{e:keproof2}
&   \Big(  c^{-1}  \times \prob\big( \area(\Vc^a(\Sc_u;t-\delta') <  ( \vartheta(t-\delta') - \eps/4) \area(\Sc_u)   \Big)^{c / \area(\Sc_u) } \\
\nonumber &  \qquad \qquad   +  e^{c^{-1} / \area(\Sc_u)}   \times e^{-  c /\area(\Sc_u )  e^{c_2 u^2 n} } .
\end{align}
The first term is bounded, as before, by $ e^{-c / \area(\Sc_u)} \le e^{-c \sqrt{n} }  $. The second term is much smaller in this case, since
\[  e^{c^{-1} / \area(\Sc_u) }   \times e^{-  c /\area(\Sc_u)   e^{c_2 u^2 n } }  \le  e^{ - e^{c n^{1/2} }} , \]
so we have $\prob(E_3) \le e^{-c \sqrt{n} }  $.

Similarly, if $E_4$ occurs then by the pigeonhole principle and extracting a subset, $\widetilde{\textrm{EU}}_\delta(S^i;t)$ fails for at least $c_6 /  \sqrt{ \area(\Sc_u) }$ elements $S^i$ contained in $U^i \subseteq \Sc^2 \cap \mathcal{O}$ and which are mutually separated by a distance at least $u$. There are at most $e^{c \log n / \sqrt{ \area(\Sc_u) } } $ possible choices for this subset. By Proposition \ref{p:sdk} and recalling \eqref{e:eboundke}, $\prob(E_4)$ is at most
\begin{equation}
\label{e:keproof3}
 e^{c^{-1} \log n / \sqrt{\area(\Sc_u) } }  \times (e_{u,\delta;t-\delta'})^{c_6 / \sqrt{\area(\Sc_u)} }   + e^{c^{-1}  \log n / \sqrt{ \area(\Sc_u) } } \times  e^{-  c / \sqrt{ \area(\Sc_u) }  e^{c_2 u^2 n} }    .
 \end{equation}
Using that $e_{u,\delta;t-\delta'} \le e^{-c u /\sqrt{n}}$ by Proposition \ref{p:lu}, this is at most
 \[  e^{c^{-1} \log n / \sqrt{ \area(\Sc_u) }}  e^{-c u / \sqrt{\area(\Sc_u) }  }  + e^{-e^{c \sqrt{n} } }  \le e^{ c \sqrt{n}}  \]
 which concludes the proof of of Proposition \ref{prop:upper/lower conc Kostlan}(ii).
\end{proof}

\subsection{Application to band-limited ensembles: Proof of Theorems \ref{thm:unique giant band-lim}-\ref{thm:unique giant bl2}}
Fix $\alpha \in [0,1]$ and let $\varphi(t) = \varphi(\alpha,t)$ be the constant \eqref{e:vartheta} in Proposition \ref{p:qualcon}, defined with respect to the band-limited ensemble (i.e.\ taking $F_\infty = h_\alpha$). Recall the exponent $p \in \{1,\beta\}$ from \eqref{e:p}. As before, it remains to establish bounds on the volume-concentration of the giant (c.f.\ Proposition \ref{prop:upper/lower conc Kostlan}).

\begin{proposition}
\label{prop:upper/lower conc band-lim}
$\,$
\begin{enumerate}[i.]
\item For every $t \neq 0$ and $\epsilon>0$ there exists a constant $c>0$ such that
\begin{equation}
\label{e:bl1}
\prob\big( \area(\Vc^a(t)) > 4 \pi ( \varphi(t) + \eps ) \big) \le e^{- c\ell^{4p/3}} .
\end{equation}
\item For every $t>0$ and $\epsilon,\delta>0$ there exists a constant $c>0$ such that
\begin{equation}
\label{e:bl2}
\prob\big( \area(\Vc^a(t)) < 4 \pi ( \varphi(t) - \eps ) \big) \le e^{- c\ell^{p-\delta}} .
\end{equation}
\end{enumerate}
\end{proposition}

\begin{proof}[Proof of Theorems \ref{thm:unique giant band-lim}-\ref{thm:unique giant bl2} assuming Proposition \ref{prop:upper/lower conc band-lim}]
Theorems \ref{thm:unique giant band-lim} and \ref{thm:unique giant bl2} follow by combining \eqref{e:bl1}--\eqref{e:bl2}, Lemma \ref{l:lu}, and the estimates on $e_{r}$ in Proposition \ref{p:lu}. In particular Proposition \ref{p:lu} gives that, for every $\delta > 0$, there are constants $c_1,c_2 > 0$ such that
\begin{equation}
\label{e:eboundbl}
e_r \le c_1  e^{-c_2 r \ell^{p-\delta} } ,
\end{equation}
so the uniqueness of the giant is given by the first statement of Lemma \eqref{l:lu}; as a consequence the bound in \eqref{e:bl2} transfers to the diametric giant $\Vc^d(t)$. Moreover, setting $r = \ell^{-p+\delta'}$ for some $\delta' > \delta$, there are constants $c_3,c_4> 0$ such that
\[ r^{-2} (\log 1/r) \sup_{r' \ge r} e_{r'}   \le  e^{- c_3 \ell^{c_4}} ,\]
and so the ubiquity and local uniqueness of the giant component follows from the second statement of Lemma \ref{l:lu}.
\end{proof}

In the proof of Proposition \ref{prop:upper/lower conc band-lim} we again let $t \in \R$ and $\eps > 0$ be given, and denote by $c > 0$ and $c_i$ constants which are independent of $\ell$, with the constant $c > 0$ possibly changing from line to line. We will also assume that $\ell$ is taken sufficiently large, depending on $t$ and $\eps$, without explicitly mention. We shall assume the continuity of $t \mapsto \vartheta(t)$  on $\R_{>0}$, as stated in Proposition \ref{prop:phi} and formally proven in the following section.

\begin{proof}[Proof of Proposition \ref{prop:upper/lower conc band-lim}(i)] \hfill

\textbf{Case $t>0$.}
The proof of \eqref{e:bl1} proceeds along similar lines to the proof of \eqref{e:ke1} for Kostlan's ensemble, so we only highlight the new aspects of the analysis. The main difference is the choice of scale $u$: here we define $u = c_0 \ell^{-2p/3} \gg \ell^{-1}$ for sufficiently large $c_0 > 0$. Other than these, the parameters $\delta', v = \pi$, and the events $E_1$ and $E_2$ are defined as within the proof of Proposition \ref{prop:upper/lower conc Kostlan}(i), with the exception that $\vartheta(t)$ is replaced by $\varphi(t)$. In this case we do not need to choose $c_0$ satisfying \eqref{e:c0choice}. Using the second statement of Proposition \ref{p:sdble} in place of Proposition \ref{p:sdk}, the bound on $\prob(E_1)$ in the context of band-limited ensembles becomes (c.f.\ \eqref{e:keproof1})
 \begin{equation}
 \label{eq:pr(E1) band-lim bnd}
 \begin{split}
&\prob(E_1) \le    \Big( c^{-1} \times \prob\big( \area(\Vc^a(\Sc_u;t+\delta') >  ( \varphi(t+\delta') + \eps/2) \area(\Sc_u)   \Big)^{c / \area(S_u) } \\
\nonumber & \qquad \qquad + e^{ c_2/ \area(\Sc_u) } \times e^{ -  c_3(1 / \area(\Sc_u) )^{3/4}  ( u \ell^p) } .
\end{split}
\end{equation}
By the qualitative convergence in Proposition \ref{p:qualcon}, the first term on the r.h.s.\ of \eqref{eq:pr(E1) band-lim bnd} is at most
\[  e^{-c/ \area(S_u)} \le e^{- c \ell^{4p/3} }.\]
Since $ \area(\Sc_u)^{-3/4}  ( u \ell^p)  = c \ell \cdot \ell^{p/3} = c \ell^{4p/3} $, for sufficiently large $c_0 > 0$ the second term on the r.h.s. of \eqref{eq:pr(E1) band-lim bnd} is at most
\[   e^{ c\ell^{4p/3}  ( c_4 c_0^{-2}  - c_5  c_0^{-1/2}  )  } \le e^{-c \ell^{4p/3} } . \]
Hence we conclude that $\prob(E_1)  \le e^{-c \ell^{4p/3} } $. The bound on $\prob(E_2)$ is similar to the above. Note that case $t =0$ {\em cannot be reduced} to the case $t > 0$ as was done for Kostlan's ensemble, since we have not established that $\lim\limits_{t \to 0}{\varphi(t)} = 0$.
\vspace{2mm}

\textbf{Case $t<0$.}
The proof proceeds along the same lines as for $t>0$, and therefore is omitted.
\end{proof}

\begin{proof}[Proof of Proposition \ref{prop:upper/lower conc band-lim}(ii)]
To prove \eqref{e:bl2} we need to iterate the renormalisation procedure used to prove \eqref{e:ke2} in the case of Kostlan's ensemble. Let $\delta < p/8$ be such that $p/(4\delta) \in \N_{\ge 2}$, and denote $k = p/(4\delta) -1 \ge 1$. For $j \in \{0,1,\ldots , k\}$ define $\eps_j = \eps 8^{j - k}$, and fix a strictly decreasing positive sequence
\[  t_j \downarrow t_{k} = t > 0 \ , \quad j = 0,\ldots,k, \]
satisfying
 \begin{equation}
  \varphi(t_{j+1}) - \eps_{j+1}/2    <  \varphi(t_{j}) - \eps_{j+1}/4 < \varphi(t_j) - \eps_j  ,
  \end{equation}
   which is possible by the continuity of $\varphi$ on $\R_{>0}$. Define $u_j = \ell^{-p+4(j+1)\delta}$, and for $j = \{0,1,\ldots,k-1\}$
 let $U_j = S_{u_j}$, and define $U_k = \Dc_v = \Sc^2$. We will show by induction that, for $j  = 0, \ldots, k$ and $c_0 > 0$ chosen sufficiently small,
\begin{equation}
\label{e:bl4}
 p_j := \prob\big( \area(\Vc^a(U_j;t_j) < (\varphi(t_j) - \eps_j) \area(U_j) \big) \le e^{- c_0 \ell^{\gamma_j}  },
 \end{equation}
 where $\gamma_j = (4j+3) \delta$. Since $U_k = \Sc^2$, $t_k = t$, $\eps_k = \eps$, and $\gamma_k =  p-\delta$, the case $j=k$ in \eqref{e:bl4} coincides with \eqref{e:bl2}.

\vspace{2mm}

\textbf{Base case.} Observing that $U_0 = S_{u_0}$, $u_0 \ell   \to \infty$ as $\ell \to \infty$, $t_0 > 0$, $\eps_0 > 0$, $\gamma_0 = 0$, and $\ell^{p-1} \le 1$, the base case $j = 0$ follows from Proposition \ref{p:qualcon} (for any choice of $c_0$).

\vspace{2mm}

\textbf{Inductive step.} This step proceeds along similar lines to the proof of Proposition \ref{prop:upper/lower conc Kostlan}(ii), but the analysis is somewhat different. Suppose that \eqref{e:bl4} holds for some $j = 0,\ldots,k-1$. Assume that $$\area(\Vc^a( U_{j+1};t_{j+1}) > (\varphi(t_{j+1}) + \eps_{j+1}) \area(U_{j+1}) ,$$ and fix a $(u_j,\eps_{j+1}/3)$-tiling $\{S^i\}$ of $U_{j+1}$. Then applying Lemma \ref{l:lb} (with $v \mapsto u_{j+1}$, $u \mapsto u_j$, $\eps \mapsto \eps_{j+1}/3$ and $\delta = \eps_{j+1}/6$), defining events $E_3$ and $E_4$ as in the proof of Proposition \ref{prop:upper/lower conc Kostlan}(ii) (with $t \mapsto t_{j+1}$, $\eps \mapsto \eps_{j+1}$, $u \mapsto u_j$, and $4 \pi \mapsto \textrm{Area}(U_{j+1})$), and using the second statement of Proposition \ref{p:sdble} in place of Proposition \ref{p:sdk},
 \eqref{e:keproof2} becomes
  \begin{align*}
&    \Big( c^{-1}  \times \prob\big( \area(\Vc^a(U_j;t_{j}) < ( \varphi(t_{j}) - \eps_j) \area(U_j)   \Big)^{c \area(U_{j+1})  / \area(U_j) } \\ &  \qquad \qquad   +  e^{c^ {-1}  \area(U_{j+1})  / \area(U_j)}   \times  e^{ -  c (\area(U_{j+1}) / \area(U_j) )^{3/4}  ( u_j \ell^p)  } ,
\end{align*}
and
\eqref{e:keproof3} becomes
\begin{align*}
&  (e^{c^{-1} (\log \ell) }   \times e_{u,\delta;t_j})^{c   \sqrt{ \area(U_{j+1})  /  \area(U_j) }  }   \\
& \qquad  + e^{c^{-1} (\log \ell) \sqrt{ \area(U_{j+1})  / \area(U_j) } } \times    e^{ -  c_6 \big( \sqrt{\area(U_{j+1}) / \area(U_j)} \big)^{3/4}  ( u_j \ell^p) }    .
 \end{align*}

 Noting that $p_{j+1}$ is at most the sum of these two expressions, using the inductive assumption \eqref{e:bl4}, the fact that $e_{\delta,u; t_j} \le    e^{-  c u_j \ell^{p-\delta'}  } $ for arbitrarily small $\delta' > 0$ by Proposition \ref{p:lu}, and since
 \[ c  \ell^{8 \delta}   \le \area(U_{j+1}) / \area(U_j) \le   \ell^{8 \delta} / c  , \]
 we conclude that
 \begin{align*}
  p_{j+1} & \le e^{-c \ell^{\gamma_j  + 8 \delta}} + e^{c^{-1} \ell^{8\delta}  -  c \ell^{6\delta + 4(j+1) \delta } }   + e^{-c \ell^{4(j+1)\delta - \delta'}  \ell^{4\delta}  }   +  e^{c^{-1} \ell^{4\delta} (\log \ell)  -  c_6 \ell^{3\delta + 4(j+1) \delta } }  \\
&    \le e^{-c \ell^{\gamma_j  + 8 \delta}}      +  e^{-  c_6 \ell^{3\delta + 4(j+1) \delta } }  \le  2e^{-  c_6 \ell^{3\delta + 4(j+1) \delta } }  \le e^ {-c_7 \ell^{\gamma_{j+1}} }   ,
\end{align*}
where the final step uses that $$\gamma_j + 8\delta = (4j+3)\delta + 8\delta > 3\delta + 4(j+1)\delta = \gamma_{j+1}.$$ Choosing $c_0 = c_7$, this completes the inductive step.
\end{proof}

\medskip
\section{Concentration for the random spherical harmonics: Proof of Theorem \ref{thm:spher harm}}
\label{s:rsh}

In this section we use a separate argument to establish the asserted concentration bounds for random spherical harmonics $T_\ell$, which are somewhat weaker than for the other ensembles. Again we work with the volumetric giant $\Vc^a(t)$ so as to exploit monotonicity.

\subsection{Concentration of the threshold}
We will infer the concentration bounds on $\area(\Vc^a(t))$ from corresponding concentration bounds for the `threshold' random variables
\[ \mathcal{T}_s := \mathcal{T}_s(T_\ell) := \min\{ t  \in \R: \area(\Vc^a(t)) \ge s \}  \ , \quad s \in (0,4\pi]  .\]
The relevance of $\mathcal{T}_s$ is that, since the event $\{\area(\Vc^a(t)) \ge s\}$ is increasing in $t$, the random variables $\mathcal{T}_s$ are amenable to sharp threshold arguments. This general strategy was employed in \cite{EH21} in the context of Bernoulli percolation on transitive graphs.

 \begin{proposition}
 \label{p:st}
 There exists a countable set $C \subseteq \R$ and a constant $c > 0$ such that, for every $s \in (0,4\pi] \setminus C$, $\ell \ge 2$ and $u \ge 0$,
 \[ \prob( |\mathcal{T}_s -  \E[\mathcal{T}_s] | \ge u ) \le e^{-c u \sqrt{\log \ell} }  .   \]
 \end{proposition}

The proof of Proposition \ref{p:st} uses only that $T_\ell$ a smooth isotropic Gaussian field which admits a polynomial bound on its decay of correlations, namely that for sufficiently small $\gamma > 0$,
$$|\E[T_\ell(\eta) \cdot T_\ell(x)]| \le \ell^{-\gamma}$$ for all $x$ such that $d_{\Sc^2}(\eta,x) \in [\ell^{-\gamma}, \pi - \ell^{-\gamma}]$. Hence an analogous statement is also true for Kostlan's ensemble and band-limited ensembles.
\smallskip

Towards proving Proposition \ref{p:st} we state two preliminary lemmas. We will take the view that $\mathcal{T}_{s}$ is a functional $\mathcal{T}_s = \mathcal{T}_s(f)$ on $f \in C^\infty(\Sc^2)$, thus extending the original meaning of $\mathcal{T}_{s}$ as a random variable.

 \begin{lemma}
 \label{l:lip}
 The functional $f \mapsto \mathcal{T}_s(f)$ is increasing and $\mathcal{T}_s(f + h) \le \mathcal{T}_s(f) + \|h\|_{\infty}$.
 \end{lemma}

\begin{proof}
The monotonicity of $\mathcal{T}_s(f)$ in $f$ is clear from the definition. Moreover, since, for every $v \in \R$,
\[ \mathcal{T}_s(f +  v) = \mathcal{T}_s(f) + v, \]
the second property of Lemma \ref{l:lip} follows from the first one.
\end{proof}

  Recall that the \textit{Gateaux derivative} of $\mathcal{T}_s$ at $f \in C^\infty(\Sc^2)$ in direction $h \in \C^\infty(\Sc^2)$ is defined as
  \[ \partial_h \mathcal{T}_s(f) :=   \frac{d}{dv} \mathcal{T}_s(f + v h) \big|_{v = 0}\]
  whenever it exists. The following lemma shows that the Gateaux derivative is a.s.\ upper semi-continuous at $f = T_\ell$:

 \begin{lemma}
 \label{l:morse}
For every $\ell$ there exists a countable set $C_\ell \subseteq \R$ such that, for every $s \in (0,4\pi] \setminus C_\ell$ and $h \in  C^\infty(\Sc^2)$, almost surely the Gateaux derivative $\partial_h \mathcal{T}_s(f)$ exists and is upper semi-continuous at $f = T_\ell$, i.e.\ almost surely
\[ \limsup\limits_{f \to T_\ell} \partial_h \mathcal{T}_s(f) \le \partial_h \mathcal{T}_s(T_\ell),\]
for every sequence $f \to T_\ell$ in $C^\infty(\Sc^2)$.
 \end{lemma}
 \begin{remark}
 We introduce $C_\ell$ to handle the possibility that $\mathcal{T}_s$ is a critical value of $T_\ell$, which would result in a discontinuity in the derivative $\partial_h \mathcal{T}_s$. While we believe this occurs with probability zero for every fixed $s$, we prefer to avoid this by excluding the (possibly empty) set $C_\ell$.
  \end{remark}

 \begin{proof}[Proof of Lemma \ref{l:morse}]
  Recall that the \textit{critical values} of $T_\ell$ is set
  \[ \textrm{Cr} = \{ t \in \R : \text{ there exists } x \in \Sc^2 \text{ such that } \nabla_{\Sc^2} T_\ell(x) = 0 \text{ and } T_\ell(x) = t\}. \]
 By Bulinskaya's lemma, almost surely $\textrm{Cr}$ is finite, and for each $t \in \textrm{Cr}$, there is a unique $x \in \Sc^2$ such that $\nabla_{\Sc^2} T_\ell(x) = 0$ and $T_\ell(x) = t$. We observe that, by the Morse lemma, the components of $\Uc(t)$ vary smoothly w.r.t.\ $t \in \R \setminus \textrm{Cr}$, and the areas of these components are analytic functions of $t$ (see e.g ~\cite[Lemma 10]{BW} and its proof).

 For $t \in \R$ and $s > 0$ define the events
 \[ F(t,s) = \{ \text{there is a component of $\Uc(t)$ of area } s \} \cap \{ t \text{ is a critical value of } T_\ell \} . \]
 Then define the (possibly, empty) subset
\[  C_\ell =   \big\{ s > 0 :  \prob\big( E(s) \big)  > 0 \big\} \ , \quad   E(s) = \bigcup\limits_{t \in \R} F(t,s) .\]
Since almost surely $\textrm{Cr}$ is finite, and for each $t \in \textrm{Cr}$, $\Uc(t)$ contains a finite number of components, the set $\{s > 0 : E(s) \text{ occurs}\}$ defines an almost surely finite point process on $\R_{>0}$, with $C_\ell$ its set of atoms. This implies that $C_\ell$ is countable.

It remains to fix $s \notin C_\ell$ and $h \in C^\infty(\Sc^2)$ and argue that the Gateaux derivative $\partial_h \mathcal{T}_s(f)$ almost surely exists and is upper-semicontinuous at the point $f = T_\ell$. To see this observe that, since $s \notin C_\ell$, one of the following must occur (i) $\area(\Vc^a(\mathcal{T}_s)) = s$, $\mathcal{T}_s$ is not a critical value of $T_\ell$ (recall that $\mathcal{T}_s=\mathcal{T}_s(T_{\ell})$), and all components of $\Uc(\mathcal{T}_s) \setminus \Vc^a(\mathcal{T}_s)$ have area $< s$; or (ii) $\area(\Vc^a(\mathcal{T}_s)) = s$, $\mathcal{T}_s$ is not a critical value of $T_\ell$, and there is a finite number $n \ge 2$ of components $(W_i)_{1 \le i \le n}$ of $\Uc(\mathcal{T}_s)$ with area equal to $s$; or (iii) $\area(\Vc^a(\mathcal{T}_s)) = s^+ > s$.

Case (i). In this case, by the Morse lemma, the components of $\Uc(t)$ for the field $T_\ell + h$ at the level $t = \mathcal{T}_s(T_\ell)$ vary smoothly for $h$ in a neighbourhood of zero in $C^\infty(\Sc^2)$, so $\mathcal{T}_s(T_\ell + h)$ is also varying smoothly for $h$ in this neighbourhood. This implies that $\partial_h \mathcal{T}_s(f)$ exists and varies smoothly for $f$ in a neighbourhood of $f = T_\ell$.

Case (ii). In this case the components of $\Uc(t)$ for the field $T_\ell + h$ at the level $t = \mathcal{T}_s(T_\ell)$ vary smoothly for $h$ in a neighbourhood of zero, and the derivative $\partial_h \mathcal{T}_s(T_s)$ exists and can be written as a maximum over functionals $g_i$ that depend only the field $T_\ell$ in a neighbourhood of the component $W_i$. Notice that this may result in a discontinuity in $\partial_h \mathcal{T}_s(f)$ $f = T_\ell$, since a perturbation may cause some components $W_i$ to no longer have area $s$ at the threshold $\mathcal{T}_s(f)$. However since the derivative is described as a maximum, the derivative is still upper-semicontinuous in this case.

Case (iii). In this case $\mathcal{T}_s$ is a critical value of $T_\ell$, almost surely there is a unique point $x_{0}=x_{0}(T_{\ell}) \in \Sc^2$ such that $\nabla_{\Sc^2} T_\ell(x_{0}) = 0$ and $T_\ell(x_{0}) = \mathcal{T}_s$, and $x_{0}$ is a smooth function of $f$ in a neighbourhood of $f = T_\ell$. Moreover $\Vc^a(\mathcal{T}_s) \setminus \{x\}$ contains two components each of whose areas is $< s$, and all other components of $\Uc(\mathcal{T}_s)$ have area $<s$. Then $\mathcal{T}_s(T_{\ell} + h) = T_{\ell}(x_{0}) +  h(x_{0})$ for $h$ in a neighbourhood of zero of $C^\infty(\Sc^2)$; in particular
$\partial_h \mathcal{T}_s(T_\ell) = h(x_{0})$, and thus, the Gateaux derivative $\partial_h \mathcal{T}_s(f)$ is a smooth function of $f$ in a neighbourhood of $f = T_\ell$.
\end{proof}

In the proof of Proposition \ref{p:st} it will be convenient to approximate the field $T_\ell$ by a `discretised' version that only depends on the values of $T_\ell$ on a fine mesh, defined as follows:

\begin{definition}[Spherical partition of unity]
\label{d:pou}
Fix a constant $c_0 > 0$, sequences $\eps_j \to 0$ and $n_j \to \infty$, a sequence of finite collections of points $\{p^j_i\}_{i \le n_j}$ on $\Sc^2$, and a sequence of collections of smooth non-negative functions $\{\zeta^j_i\}_{i \le n_j}$ on $\Sc^{2}$, with the following properties:
\begin{itemize}
\item $n_j \le c_0 \eps_j^{-2}$;
\item Each $\zeta^j_i$ is supported on the spherical cap $\Dc_{\eps_j}(p^j_i)$;
\item For all $j \ge 1$, $\zeta^j_i$ form a partition of unity of $\Sc^2$, i.e.\
\[ \sum_{i \le n_j} \zeta^j_i(x) = 1  \ , \quad \text{for all } x \in \Sc^2 . \]
\end{itemize}
\end{definition}

\begin{definition}[Spherical discretisation]
\label{d:disc}
Let $n_j$, $p^j_i$, and $\zeta^j_i$ be as in Definition \ref{d:pou}. For a function $f \in C^\infty(\Sc^2)$, the \textit{(level-$j$) spherical discretisation of $f$} is the smooth function
\begin{equation}
\label{e:sd}
 (I^j  f)(x) :=\sum_{i \le n_j}  f(p^j_i) \cdot \zeta^j_i (x) .
 \end{equation}
Clearly $I^j f$ depends only on the finite set of values $f(p^j_i)$, and approximates $f$ in the sense that, as $j \to \infty$,
\[ (I^j f)(x) \to f(x)   \quad  \text{ in } C^\infty(\Sc^2) . \]
\end{definition}

 \begin{proof}[Proof of Proposition \ref{p:st}]
We will apply a general sharp threshold argument for Gaussian fields that originates in the work of Chatterjee  \cite{cha14} and Tanguy  \cite{tan15} on the maximum of a Gaussian field \cite{cha14}, and later adapted to a percolative context in \cite{MRVK}.

Let $C = \bigcup\limits_{\ell} C_\ell$, where $C_\ell$ are the countable sets in Lemma \ref{l:morse}. Let $s \in (0,4\pi] \setminus C$, and $\ell \ge 2$ be given. Since $s$ and $\ell$ are fixed we shall abbreviate $T = T_\ell$ and $\mathcal{T}=\mathcal{T}_s$. First note that, since
\[ \mathcal{T} \le  \sup_{x \in \Sc^2} T(x) ,\]
the Borell--TIS inequality implies that $\mathcal{T}$ has exponential moments. Then by a classical concentration result valid for arbitrary random variables with exponential moments (see \cite[Lemma 6]{tan15}), it suffices to prove that, for every $\theta \in \mathbb{R}$,
\begin{equation}
\label{e:concen}
  \textrm{Var}\left(e^{\theta \mathcal{T}}\right) \le  c \theta^2 /(\log \ell) \times   \mathbb{E}\left[e^{2\theta \mathcal{T}}\right] ,
  \end{equation}
  where $c > 0$ is an absolute constant.

To establish \eqref{e:concen} we work with the discretisation $T^j :=   I^j T$ defined in \eqref{e:sd}. We define the corresponding threshold functional $\mathcal{T}^j := \mathcal{T}(T^j)$ which, by Definition \ref{d:disc} and Lemma \ref{l:lip}, is a Lipschitz function of the finite-dimensional Gaussian vector $(T(p^j_i))_{i \le n_j}$. Define also the derivatives $\partial_i \mathcal{T}^j$ of $\mathcal{T}^j$ with respect to the values $T(p^j_i)$, i.e.\
\[  \partial_i \mathcal{T}^j := \partial_v  \mathcal{T} \Big(  v \zeta^j_i (x)  + \sum_{k \neq i \le n_j}  T(p^j_k) \zeta^j_k (x) \Big) \Big|_{v = T(p^j_i)}, \]
which exist almost surely, since $\mathcal{T}^j$ is Lipschitz, and satisfy $\partial_i \mathcal{T}^j  \ge 0$.

Let $\gamma \in (0, 1)$ be an arbitrary constant, and assume that $j$ is sufficiently large so that $\eps_j \le \ell^{-\gamma}$. Abbreviate the spherical cap $\Dc := \Dc_{\ell^{-\gamma}}$ and define the sets
\[ \Dc^+ :=  \{ x \in \Sc^2 : d_{\Sc^2}(\eta,x) \in [ 0 , 2\ell^{-\gamma}] \cup [\pi-2\ell^{-\gamma}, \pi] \}  \]
\[ \Dc^{+2} :=  \{ x \in \Sc^2 : d_{\Sc^2}(\eta,x) \in [0, 3\ell^{-\gamma}] \cup [\pi-3\ell^{-\gamma}, \pi] \} \]
and
\[ \Dc^{+3} :=  \{ x \in \Sc^2 : d_{\Sc^2}(\eta,x) \in [0, 4\ell^{-\gamma}] \cup [\pi-4\ell^{-\gamma}, \pi] \}  .\]
 Let $h : \Sc^2 \to [0,1]$ be a smooth zonal function with value $1$ on $ \Dc^{+2}$ and value $0$ outside~$ \Dc^{+3}$.

Let $\{D_\alpha\}_{\alpha}$ be a finite covering of $\Sc^2$ by rotated copies of $\Dc$ with the properties that
\[  |\{D_\alpha\}| \ge   c_1^{-1} \ell^{2\gamma} \quad  \text{and}  \quad  | \{ \alpha : x \in D^{+3}_{\alpha}\} | \le c_1 \text{, for every $x \in \Sc^2$}  \]
for an absolute constant $c_1 > 0$, where $D_\alpha^{+3}$ denotes the rotated copy of $\Dc^{+3}$ under a rotation $\pi_\alpha$ that sends $\Dc$ to $D_\alpha$. Define $D_\alpha^{+}$, $D_\alpha^{+2}$, and $h_\alpha$, analogously as rotated copies of $ \Dc^{+}$, $\Dc^{+2}$, and $h$ respectively under the rotation $\pi_\alpha$.

 Let us observe that, as a consequence of Definition \ref{d:disc} and of Lemma \ref{l:lip},
 \begin{equation}
\label{e:derivprop}
0 \le  \sum_{p^j_i \in D_\alpha^+} \partial_i \mathcal{T}^j  \le  \partial_{h_\alpha} \mathcal{T}^j \le 1   \quad \text{and} \quad \sum_\alpha \partial_{h_\alpha} \mathcal{T}^j \le c_1,
 \end{equation}
 and similarly
  \begin{equation}
\label{e:derivprop2}
 \partial_{h_\alpha} \mathcal{T} \ge 0   \quad \text{and} \quad \sum_\alpha \partial_{h_\alpha} \mathcal{T} \le c_1.
 \end{equation}

 The main step in establishing \eqref{e:concen} is the following claim:
 \begin{claim}
 \label{c:int}
 There exists a constant $c_2 > 0$ such that
\begin{equation}
\label{e:int}
  \textrm{Var}\left(e^{\theta \mathcal{T}^j}\right)  \le c_2 \theta^ 2 \left(    \frac{1}{| \log   \sup_\alpha   \E\big[ \partial_{h_\alpha} \mathcal{T}^j \big] |  } +     \ell^{(\gamma-1)/2 } \right) \times  \mathbb{E}\big[  e^{2\theta \mathcal{T}^j} ]  .
  \end{equation}
\end{claim}

 Let us conclude the proof of \eqref{e:concen} assuming Claim \ref{c:int}, whose proof will be given immediately below. By the definition of the discretisation, we have that $T^j \to T$ in $C^\infty(\Sc^2)$ as $j \to \infty$. By Lemma \ref{l:morse}, this implies that almost surely $\lim_{j \to \infty} \mathcal{T}^j = \mathcal{T}$ and $\limsup_{j \to \infty} \partial_{h_\alpha}  \mathcal{T}^j \le \partial_{h_\alpha} \mathcal{T}$. Since both
 \begin{equation}
 \label{e:uei}
  \partial_{h_\alpha} \mathcal{T}^j \in [0,1] \quad \text{and} \quad \mathcal{T}^j \le  \sup_{x \in \Sc^2} T^j(x) \le  \sup_{x \in \Sc^2} T(x)
  \end{equation}
 are uniformly exponentially integrable, taking $j \to \infty$ in \eqref{e:int} we deduce from the dominated convergence theorem that
\begin{equation}
\label{e:concen2}
 \textrm{Var}[e^{\theta \mathcal{T}}]  \le c_2 \theta^ 2 \left(    \frac{1}{  | \log \sup\limits_\alpha   \E\left[ \partial_{h_\alpha} \mathcal{T} \right] |  } +     \ell^{(\gamma-1)/2 } \right) \times  \mathbb{E}\big[  e^{2\theta \mathcal{T}} ] .
 \end{equation}
By \eqref{e:derivprop2}, and since, by the rotational invariance, $\E[ \partial_{h_\alpha} \mathcal{T} ] $ does not depend on $\alpha$, we have
\[    \sup\limits_\alpha \E\big[ \partial_{h_\alpha} \mathcal{T}_s \big]    \le  \frac{1}{| \{D_\alpha\}|}  \le c_1 \ell^{-2\gamma}.\]
Inserting this into \eqref{e:concen2} proves \eqref{e:concen}.
 \end{proof}

 \begin{proof}[Proof of Claim \ref{c:int}]
Let $\tilde{T}$ be an independent copy of $T$, and for each $w \ge 0$ let
\begin{equation}
\label{e:inter}
T_w(\cdot) = e^{-w} T(\cdot) + \sqrt{1 - e^{-2w} } \tilde{T}(\cdot) ;
\end{equation}
this defines an interpolation from $T$ to $\tilde{T}$ along the Ornstein-Uhlenbeck semigroup \cite{cha14}. Define the discretised field $\tilde{T}^j =  I^j  \tilde{T}$ as in \eqref{e:sd}, and the corresponding threshold functional $\mathcal{T}^j_w = \mathcal{T}_s(T^j_w)$. Via a classical interpolation argument for finite-dimensional Gaussian vectors (see \cite{tan15,cha14}, valid by the integrability in \eqref{e:uei}), one has the exact formula
\begin{equation}
\label{e:var1}
    \textrm{Var}[e^{\theta \mathcal{T}^j}] =  \theta^2 \sum_{ i, k} \E[ T(p^j_i) T(p^j_k) ] \int\limits_{0}^{\infty}  e^{-w}  \mathbb{E}\Big[  e^{\theta \mathcal{T}^j }e^{\theta \mathcal{T}^j_w  } \partial_i \mathcal{T}^j \partial_k \mathcal{T}^j_w  \Big] \,  dw .
    \end{equation}
where $ \E[ T(p^j_i)\cdot  T(p^j_k) ]  = P_\ell( d_{\Sc^2} (x,y) )$ is the covariance kernel of $T = T_\ell$.

Recalling the definitions of $D_\alpha$ and $D_\alpha^+$, and the apriori positivity $\partial_i \mathcal{T}^j  \ge 0$, we can bound the sum in \eqref{e:var1} as
\begin{equation}
\label{e:var2}
\begin{split}
\textrm{Var}[e^{\theta \mathcal{T}^j}]& \le   \theta^2 \sum_{\alpha}  \sum_{p^j_i \in D_\alpha, p^j_k \in D_{\alpha}^+ }  |P_\ell( d_{\Sc^2} (p^j_i,p^j_k) )|  \int\limits_{0}^{\infty}  e^{-w}  \mathbb{E}\Big[  e^{\theta \mathcal{T}^j }e^{\theta \mathcal{T}^j_w  } \partial_i \mathcal{T}^j \partial_k \mathcal{T}^j_w  \Big]  \, dw  \\
 &+   \theta^2  \sum_{\alpha}  \sum_{p^j_i \in D_\alpha, p^j_k \notin D^+_\alpha}  |P_\ell( d_{\Sc^2} (p^j_i,p^j_k) ) | \int\limits_{0}^{\infty}  e^{-w}  \mathbb{E}\Big[  e^{\theta \mathcal{T}^j }e^{\theta \mathcal{T}^j_w  } \partial_i \mathcal{T}^j \partial_k \mathcal{T}^j_w  \Big]  \, dw.
\end{split}
\end{equation}
To bound the second term on the r.h.s.\ of \eqref{e:var2} we observe that, by Lemma \ref{l:ubrsh}, for $x \in D_\alpha$ and $y \notin D_\alpha^+$,
\[  |P_\ell( d_{\Sc^2} (p^j_i,p^j_k) ) | \le c_3 \ell^{(\gamma-1)/2 } ,\]
for a constant $c_3 > 0$. Combining with $\partial_i \mathcal{T}^j  \ge 0$ and $\sum\limits_{i} \partial_i \mathcal{T}^j \le c_1$ we have
\begin{align*}
\textrm{second term in \eqref{e:var2}}  & \le  c_3  \ell^{(\gamma-1)/2 }   \theta^2   \sup\limits_{w \ge 0}    \mathbb{E}\left[  e^{\theta \mathcal{T}^j }e^{\theta \mathcal{T}^j_w  }   \sum_{ i,k} \partial_i \mathcal{T}^j \partial_k \mathcal{T}^j_w  \right]   \\
  &    \le  c_1^2 c_3  \ell^{(\gamma-1)/2 } \theta^2    \sup\limits_{w \ge 0}    \mathbb{E}\left[  e^{\theta \mathcal{T}^j }e^{\theta \mathcal{T}^j_w  }  \right]   \\
  & \le  c_1^2 c_3  \ell^{(\gamma-1)/2 }   \theta^2   \mathbb{E}\left[  e^{2 \theta \mathcal{T}^j } \right]
  \end{align*}
  where the final step used the Cauchy-Schwarz inequality and the equality in law of $T$ and $T_w$.

\vspace{2mm}

Let us now bound the first term in \eqref{e:var2}. Using  $ |P_\ell( d_{\Sc^2} (x,y) )| \le 1$, and \eqref{e:derivprop},
\begin{align*}
  \textrm{first term in \eqref{e:var2}}  &\le   \theta^2 \sum_{\alpha} \int_0^\infty  e^{-w}   \sum_{p^j_i ,p^j_k \in  D_{\alpha}^+ }  \mathbb{E}\left[  e^{\theta \mathcal{T}^j }e^{\theta T^j_w  } \partial_i \mathcal{T}^j \partial_k \mathcal{T}^j_w  \right]  \, dw \\
 & \le    \theta^2 \sum_{\alpha}  \int_0^\infty  e^{-w}  \mathbb{E}\left[  e^{\theta \mathcal{T}^j } e^{\theta \mathcal{T}^j_w  }  \partial_{h_\alpha} \mathcal{T}^j \partial_{h_\alpha}  \mathcal{T}^j_w  \right]   \, dw  .
 \end{align*}
  To control this integral we exploit the hypercontractivity of the Ornstein-Uhlenbeck semigroup (see \cite{cha14}), namely that for any functional $G: C^\infty(\Sc^2) \to \mathbb{R}$ and any $p, q \ge 1$ and $w\ge 0$ such that $e^{2w} \ge (q-1)/(p-1)$,
\[ \mathbb{E} \left[ \Big| \mathbb{E} \big[ G( \tilde{T}_w) \big| \mathcal{F}_T  \big] \Big|^q  \right]^{1/q}  \le   \mathbb{E}\left[ |G(T)|^p \right]^{1/p} , \]
where $\mathcal{F}_T$ denotes the $\sigma$-algebra generated by $T$. Define $p(w) = 1 + e^{-w} \in (1,2]$ and its H\"{o}lder complement $q(w) = 1+e^w \in [2, \infty)$, and note that $e^{2w} = (q(w) -1)/(p(w)-1)$ and $(2-p(w))/p(w) = \textrm{tahn}(w/2)$.

Recalling the equality in law of $T$ and $T_w$, applying first H\"{o}lder's inequality (with $p' = p(w)$ and $q' = q(w)$), then hypercontractivity (to the function $G = e^{\theta \mathcal{T}^j } \partial_{h_\alpha} \mathcal{T}^j$, still  with $p' = p(w)$ and $q' = q(w)$), then again H\"{o}lder's inequality (with $p' = 2/p(w)$ and $q' = 2/(2-p(w))$), we have
\begin{align*}
 \mathbb{E}\left[  e^{\theta \mathcal{T}^j } e^{\theta \mathcal{T}^j_w  }  \partial_{h_\alpha} \mathcal{T}^j \partial_{h_\alpha}  \mathcal{T}^j_w  \right]   &   \le  \mathbb{E} \left[ \left(  e^{\theta \mathcal{T}^j }  \partial_{h_\alpha} \mathcal{T}^j  \right)^{p(w)} \right]^{1/p(w)}  \times  \mathbb{E} \left[ \mathbb{E}\left[ e^{\theta \mathcal{T}^j  }  \partial_{h_\alpha} \mathcal{T}^j_w  \Big| \mathcal{F}_T \right]^{q(w)} \right]^{1/q(w)}  \\
 & \le  \mathbb{E} \left[ \left(  e^{\theta \mathcal{T}^j }  \partial_{h_\alpha} \mathcal{T}^j \right)^{p(w)} \right]^{2/p(w)}  \\
 & \le   \E\left[ \partial_{h_\alpha} \mathcal{T}^j \right]^{\textrm{tahn}(w/2) + 1}   \mathbb{E} \big[  e^{2\theta \mathcal{T}^j }  \big]   .
 \end{align*}
Summing over $\alpha$ and integrating over $w$ we have
 \[ \textrm{first term in \eqref{e:var2}} \le    \sum_\alpha \E\big[ \partial_{h_\alpha} \mathcal{T}^j \big]  \int_0^\infty e^{-w}  u^{\textrm{tahn}(w/2)} \, dw \times \mathbb{E} \big[  e^{2\theta \mathcal{T}^j} \big] . \]
 where $u := \sup\limits_{\alpha}  \E\big[ \partial_{h_\alpha} \mathcal{T}^j \big] \in [0,1]$.  Since  $  \sum_\alpha \E\big[ \partial_{h_\alpha} \mathcal{T}^j \big] \le c_1$ and, for every $u \in (0,1)$,
\[ \int_0^\infty e^{-w} u^{\textrm{tahn}(w/2)}  \, dw \le 2 /|\log u| , \]
we conclude that
\[ \textrm{first term in \eqref{e:var2}} \le  2 c_1   \theta^2  \times \frac{1}{\left| \log   \sup\limits_\alpha    \E\big[ \partial_{h_\alpha} \mathcal{T}^j \big]\right|  } \times \mathbb{E}\left[  e^{2\theta \mathcal{T}^j} \right] .  \]
Combining with our bound on the second term in \eqref{e:var2}, this gives the claim.
\end{proof}

\subsection{Proof of Theorem \ref{thm:spher harm}}

The following lemma establishes a precise link between the concentration properties of $\mathcal{T}_s$ and $\area(\Vc^a(t))$:

\begin{lemma}
\label{l:lem}
For $\ell \in \N$, let $F_\ell(t,s): \R \times \R \to [0,1]$ be non-decreasing in $t$ and non-increasing in~$s$, and let $C \subseteq \R$ be a countable set. Suppose there exist functions $\varphi(t)$, $\psi_\ell(s)$, and $w_\ell$ such that, for every $t$ and $\delta >0$, as $\ell \to \infty$,
\begin{equation}
\label{e:l1}
 F_\ell(t, \varphi(t)+\delta) \to 0 \quad \text{and} \quad   F_\ell(t, \varphi(t)-\delta) \to 1 ,
 \end{equation}
and, for all $s \notin  C$, $\delta > 0$, and $\ell$,
\begin{equation}
\label{e:l2}
F_\ell(\psi_\ell(s)+\delta,s) \le  e^{-\delta w_\ell}  \quad \text{and} \quad F_\ell(\psi_\ell(s)-\delta,s) \ge 1- e^{-\delta w_\ell}   .
\end{equation}
Assume further that $\varphi(t)$ is continuous on $\R_{>0}$. Then for every $t, u > 0$ there exists a number $\eps > 0$ such that, for sufficiently large $\ell$,
\[F_\ell(t, \varphi(t) + u)   \ge 1-e^{-\eps w_\ell} \quad \text{and} \quad  F_\ell(t, \varphi(t) - u) \le e^{-\eps w_\ell}. \]
\end{lemma}
\begin{proof}
We shall prove only that $F(t, \varphi(t) + u)   \ge 1-e^{-\eps w_\ell} $, since the proof of $ F_\ell(t,\varphi(t)-u) \ge 1- e^{-\eps w_\ell} $ is similar. Let $t$ and $u > 0$ be given. By the monotonicity, we may assume that $\varphi(t) + u \notin C$. Since $\varphi$ is continuous, we may find $t' > t$ such that $\varphi(t') < \varphi(t) + u$. Hence by \eqref{e:l1} we have $F(t',\varphi(t)+u) \to 0$ as $\ell \to \infty$. Hence, by \eqref{e:l2}, for sufficiently large $\ell$, $\psi_\ell(\varphi(t)+u)) > t'$ and
\[ F(t, \varphi(t) + u) \le F \big(\psi_\ell(\varphi(t)+u))  - (t'-t), \varphi(t) + u \big) \le  e^{-(t'-t) w_\ell} \]
as required.
\end{proof}

\begin{proof}[Proof of Theorem \ref{thm:spher harm}]
The second statement \eqref{eq:spher harm uniqueness} of Theorem \ref{thm:spher harm} follows from the first statement of Lemma \ref{l:lu} and the estimate $e_r \le c (r \ell)^ {-1/2} \log (r\ell)$ of Proposition \ref{p:lu}. The third statement \eqref{eq:spher harm subcrit} of Theorem \ref{thm:spher harm} follows from Proposition \ref{p:roughgiant}. It then remains to prove the concentration bound
\[ \prob \big( |\area(\Vc^a)-4\pi\cdot \varphi(1,t) | \ge \epsilon \big) < \exp(-c \sqrt{ \log \ell}) , \]
assuming $t > 0$; by the just established uniqueness \eqref{eq:spher harm uniqueness} of the diametric giant outside an event of admissible probability, these bounds transfer to the diametric giant $\Vc^d$.

Consider the function $F_\ell(t,s) = \prob( \area(\Vc^a(t)) \ge 4 \pi s )$. By Proposition \ref{p:qualcon}, \eqref{e:l1} holds (with $\varphi(t) = \varphi(1,t)$). Moreover, since
\[ \{  \mathcal{T}_{4\pi s} < t \}  \implies \{ \area(\Vc^a(t) ) \ge 4\pi s \} \implies \{ \mathcal{T}_{4\pi s} \le t \} ,\]
 Proposition \ref{p:st} implies that \eqref{e:l2} holds (with $\psi_\ell(s) = \E[\mathcal{T}_s]$ and $w_\ell = c \sqrt{\log \ell}$). Since $\varphi$ is continuous by Proposition \ref{prop:phi} (whose proof, independent of Theorem \ref{thm:spher harm}, is given below), the result follows by an application of Lemma \ref{l:lem}.
\end{proof}

\medskip
\section{On the limiting densities and the optimality of the volume concentration}
\label{s:rr}

In this section we prove the remaining results, namely Propositions \ref{prop:theta} and \ref{prop:phi}, and in addition establish the optimality of the concentration bounds for Kostlan's ensemble in Theorems \ref{thm:unique giant Kostlan}(i) and \ref{thm:unique giant Kostlan2}. These appeal to the well-established techniques in percolation theory, so we only sketch the proof.

\subsection{Properties of the limiting densities: Proof of Propositions \ref{prop:theta} and \ref{prop:phi}}

\begin{proof}[Proof of Proposition \ref{prop:theta}]
Recall that $\vartheta(t)$ is the probability \eqref{e:vartheta} that the origin is contained in an unbounded component of $\{h_{BF} \le t\}$. By Proposition~\ref{p:limitfield}, for  every $t > 0$, $\{h_{BF} \le t\}$ almost surely has a unique unbounded component in $\R^{2}$ (that may or may not contain the origin) which we denote by $\Vc_\infty(t)$.

First we establish the right-continuity of $t \mapsto \vartheta(t)$ on $\R$. By Proposition \ref{p:stab}, $t \mapsto \vartheta_r(t)= \prob[\textrm{Arm}(t,r)]$ is a continuous function of $t \in \R$ for every fixed $r > 0$. Since $t \mapsto \vartheta_r $ is non-decreasing, and $\vartheta(t) = \lim_{r \to \infty} \vartheta_r(t)$, it follows that $\vartheta(t)$ is right-continuous. To extend this to continuity, we need a separate argument for $t > 0$ and $t = 0$. In the case $t = 0$ we deduce this from the right-continuity and the fact that $\vartheta(0) = 0$ by Proposition~\ref{p:limitfield}. In the case $t > 0$, recall the events $\textrm{Arm}_\infty(t,r)$ and $\textrm{TruncArm}_\infty(t,r)$ from Proposition \ref{p:ae}, and observe that for every $t > t'$ and $r > 0$
\begin{align}
\nonumber \vartheta(t) - \vartheta(t')  &= \prob[  \{0 \in \Vc_\infty(t)  \}  \setminus  \{ 0 \in \Vc_\infty(t') \}]  \\
\label{e:vardiff} &   \le  \prob[  \textrm{Arm}_\infty(t,r) \setminus \textrm{Arm}_\infty(t',r) ]  + \prob[\textrm{TruncArm}_\infty(t',r) ] .
\end{align}
Fix $t > t' > 0$ and let $t_j \to t$ be such that $t_j \ge t'$, and let $\delta > 0$ be given. Since $t, t_j \ge t'$, by Proposition \ref{p:ae} we can set $r$ sufficiently large so that
 \[ \max \big\{ \prob[ \textrm{TruncArm}_\infty(t,r) ] , \prob[ \textrm{TruncArm}_\infty(t_j,r) ]  \big\} \le \delta . \]
 Moreover, by Proposition \ref{p:stab}, we can then set $j$ sufficiently large so that
 \[ \max \big\{ \prob[  \textrm{Arm}_\infty(t,r) \setminus \textrm{Arm}_\infty(t_j,r) ]  , \prob[  \textrm{Arm}_\infty(t_j,r) \setminus \textrm{Arm}_\infty(t,r) ]   \big\} \le \delta  .\]
 Hence, given \eqref{e:vardiff},  $|\vartheta(t) - \vartheta(t_j)| \le 2 \delta$, which completes the proof.

We now show that $\vartheta(t)$ is strictly increasing. Let $t > t' > 0$ be given. By continuity of $h_{BF}$, almost surely we can find $r > 0$ sufficiently large so that $\area( ( \Vc_\infty(t)  \setminus \Vc_\infty(t') ) \cap B_r) > 0$. Hence by stationarity,
\[ \vartheta(t) - \vartheta(t') = \prob(0 \in \Vc_\infty(t)  \setminus \Vc_\infty(t')) > 0 . \]

Finally we show that $\lim_{t \to \infty} \vartheta(t) = 1$. Let $\delta, t> 0$ be given, and let $r > 0$ be sufficiently large so that $\prob( B_r \cap \Vc_\infty(t) = \emptyset) < \delta$. Let $S =
\E\left[ \sup\limits_{x \in B_r}h(x)\right] < \infty$, and let $t' = S / \delta$. Then
\begin{align*}
1- \theta(t') = \prob \big( 0 \notin \Vc_\infty(t') \big)   \le \prob \big( B_r \cap \Vc_\infty(t') = \emptyset \big)  + \prob \Big(  \sup\limits_{x \in B_r}h(x) > t' \Big) \le 2 \delta ,
\end{align*}
where we used Markov's inequality in the final step.
\end{proof}

\begin{proof}[Proof of Proposition \ref{prop:phi}]
The proof that $t \mapsto \varphi(\alpha,t)$ is strictly increasing and tends to $1$ as $t \to \infty$ is identical to the one of Proposition \ref{prop:theta}, so it remains to show that $(\alpha,t) \mapsto \varphi(\alpha,t)$ is continuous on $\R_{>0} \times \R_{>0}$. For this we observe the following:
\begin{itemize}
\item By Proposition \ref{p:ae}, $\lim_{r\to\infty} \prob[ \textrm{TruncArm}_\infty(t,r) ]  =0$ uniformly over $t \ge t'$, $\alpha \in [0,1]$.
\item Fix $r > 0$ and let $\alpha_j \to \alpha$ and $t_j \to t$ and $r > 0$ be fixed. There is a coupling of $h_{\alpha_j}$ and $h_\alpha$ so that, almost surely, $h_{\alpha_j} \to h_\alpha$ uniformly on $B_r$. Hence by Proposition \ref{p:stab},
\[  \lim_{j\to\infty}  \prob \big[ \{h_\alpha \in  \textrm{Arm}(t,r) \}  \bigtriangleup   \{h_{\alpha_j} \in \textrm{Arm}(t_j,r) \} \big]  \to 0 .\]
\end{itemize}
The proof of the bivariate continuity then follows from the same argument as in Proposition~\ref{prop:theta}.
\end{proof}

\subsection{Lower bounds on the concentration for Kostlan's ensemble}

For simplicity we will work on the subsequence of Kostlan's ensemble even degrees, to take advantage of the fact that $\kappa_n(x,y) =  \langle x, y \rangle^n \ge 0$ if $n$ is even. In particular this implies that, for a collection of domains $(D_i)_i \subseteq \Sc^2$ and $t \in \R$, the events $A_i = \{ \sup\limits_{x \in D_i} f_n(x) \le t \} $ are positively correlated \cite{Pitt}.

\begin{proposition}
\label{p:lowerbound}
Let $n \in 2\N$ and let $\Vc^d(t)$ be as in Theorem \ref{thm:unique giant Kostlan}. Then for every $t \in \R$ and $\epsilon >0$ there exists a number $c=c(t,\epsilon)>0$ such that
\[ \prob( \Vc^d(t) = \Sc^2  ) \ge  \exp(-cn ) \quad \text{and} \quad \prob(  \area(\Vc^d(t)) <  \eps ) \ge \exp(-cn^{1/2}) .\]
\end{proposition}
\begin{proof}
For the first statement, cover $\Sc^2$ with $<c_1 n$ rotated copies of the spherical cap $\Dc_{1/\sqrt{n}}$. Then by positive association and isotropy,
\[  \prob( \Vc^d(t) = \Sc^2  ) =  \prob \Big( \sup\limits_{x \in \Sc^2} f_n(x) \le t \Big)  \ge  \prob \Big( \sup\limits_{x \in \Dc_{1/\sqrt{n}}} f_n(x) \le t \Big)^{c_1 n }.\]
Moreover by the local convergence in Proposition \ref{p:lsl},
\[ \prob \Big( \sup\limits_{x \in \Dc_{1/\sqrt{n}}} f_n(x) \le t \Big) \to  \prob \Big( \sup\limits_{x \in B(1)} h_{BF}(x) \le t \Big) > c_2 > 0 ,\]
and so, for sufficiently large $n$, $\prob( \Vc^d(t) = \Sc^2  ) \ge (c_2/2)^{c_1 n }$.

For the second statement, let $L \subseteq \Sc^2$ be the union of a finite number of smooth simple loops on $\Sc^2$ such that $\Sc^2 \setminus L$ has no component of diameter $> \eps$, and cover $L$ with $< c_3 \sqrt{n}$ rotated copies of the spherical cap $\Dc_{1/\sqrt{n}}$. Then
\[  \prob(  \area(\Vc^d(t)) <  \eps  ) \ge  \prob \Big( \inf_{x \in L} f_n(x) > t \Big)  \ge  \prob \Big( \inf_{x \in \Dc_{1/\sqrt{n}}} f_n(x) > t \Big)^{c_3 \sqrt{n}},\]
and we conclude similarly using the local convergence.
\end{proof}


\end{document}